\documentclass[11pt,letterpaper]{amsart}
\usepackage{graphicx, standalone} % Required for inserting images
\usepackage{float}

\usepackage{amsthm, amssymb, amsmath, hyperref, dsfont, enumitem, tcolorbox, esint, mathtools,contour}
\usepackage[title]{appendix}
\usepackage{bbm,bm}

\usepackage[normalem]{ulem}

\usepackage{ytableau,tikz,tikz-cd}
\usetikzlibrary{decorations.pathreplacing}
\usetikzlibrary{calc,intersections}

\mathtoolsset{showonlyrefs,showmanualtags}

\usepackage{orcidlink}

\usepackage{todonotes}  % option '[disable]' to remove the notes, [textwidth=2.5cm]
\newtheorem{theorem}{Theorem}[section]
\newtheorem{lemma}[theorem]{Lemma}
\newtheorem{prop}[theorem]{Proposition}
\newtheorem{definition}[theorem]{Definition}
\newtheorem{corollary}[theorem]{Corollary}

\newtheorem{remark}[theorem]{Remark}

\numberwithin{equation}{section}

\newcommand{\oline}{\overline}
\newcommand{\oct}{\hat{\mathbbm{O}}}

\usepackage{pdfcomment}
\newcommand\alttext[1]{% 
\useasboundingbox (current bounding box);
\draw let  \p1 = (current bounding box.south west),
        \p2 = (current bounding box.north east)
    in
  node at (current bounding box.center) 
     {\pdftooltip{\phantom{\rule{\dimexpr\x2-\x1\relax}{\dimexpr\y2-\y1\relax}}}%
                 {#1}};
}

\usepackage{pdfcomment}
\newsavebox\alttextbox
\NewDocumentEnvironment{altfig}{m}
  {\begin{lrbox}{\alttextbox}}
  {\end{lrbox}\pdftooltip{\usebox{\alttextbox}}{#1}}

\title[$G_2$ and the Maximally Symmetric $(3, 6,\ldots , 8)$ Distribution]{$G_2$ and the Maximally Symmetric (3, 8) Distribution with 6-Dimensional Square}

\thanks{ I.\ Zelenko is partly supported by NSF grant DMS 2105528 and Simons Foundation Collaboration Grant for Mathematicians 524213.}

\author{Nicklas Day}
\address{Nicklas Day\\
	Department of Mathematics\\
	North Carolina State University\\
	Raleigh\\
	North Carolina \ 27695\\
	USA}
\email{ncday@tamu.edu}
\urladdr{\url{https://sites.google.com/tamu.edu/nicklasday/home}}

\author{Boris Doubrov}
\address{Boris Doubrov, Belarusian State University, 
Nezavisimosti Ave.~4, Minsk 220030, Belarus;
 E-mail: boris.doubrov@gmail.com}

\author{Igor Zelenko}
\address{Igor Zelenko\\
         Department of Mathematics\\
         Texas A\&M University\\
         College Station\\
         Texas \ 77843\\
         USA}
\email{zelenkotamu@tamu.edu}
\urladdr{\url{https://people.tamu.edu/~zelenkotamu/}}

\begin{document}
\subjclass[2020]{58A30, 58A17, 37J60, 53A55, 17B25, 17B66, 17A75, 53C17, 49K15}
\keywords{distributions (subbundles of tangent bundles), exceptional Lie group $G_2$, symplectification, the algebra of infinitesimal symmetries, Tanaka prolongation, abnormal extremals, split octonions}

\begin{abstract}
In 1910, \'{E}lie Cartan realized the split real form of the exceptional Lie group $G_2$ as the symmetry group of the maximally symmetric rank 2 distribution on a $5$-dimensional manifold with the small growth vector $(2,3,5)$ \cite{FiveVariables}. In this paper, we discover a new appearance of $G_2$ in the geometric theory of distributions, arising from a rank 3 distribution on an $8$-dimensional manifold with growth vector $(3,6,8)$. The algebra of infinitesimal symmetries of this distribution at any point is $29$-dimensional and isomorphic to $(\mathfrak{g}_2 \oplus \mathbb{R}) \ltimes W$, where $\mathfrak{g}_2$ is the Lie algebra of $G_2$ and $W$ is its adjoint module. Our model possesses three remarkable properties. First, it is maximally symmetric among all bracket-generating rank 3 distributions with a $6$-dimensional square (a family that includes both $(3,6,8)$ and $(3,6,7,8)$ distributions). To the best of our knowledge, this is the first example of a family of distributions defined by a set of prescribed small growth vectors in which maximal symmetry is achieved by a member whose growth vector is not the longest.
Second, this model provides the first counterexample to the conjecture that all bracket-generating rank 3 distributions with a $6$-dimensional square are of maximal class at a generic point (a property known to hold in dimensions $6$ and $7$). Third, further analysis yields the control-theoretic consequence that all abnormal extremal trajectories of this model originating at any point of the ambient manifold have corank at least $2$. As far as we are aware, this is the first example with this property among bracket-generating distributions with generic small growth vector for a given rank and ambient dimension. We also give an interpretation of our model in terms of the natural algebraic structure on the tangent bundle to split octonions.
\end{abstract}
\maketitle 

\section{Introduction}
We construct a bracket-generating rank $3$ distribution $\mathfrak{D}$ on an $8$-di\-men\-sion\-al manifold with square $\mathfrak{D}^2 = \mathfrak{D}+ [\mathfrak{D},\mathfrak{D}]$ of constant rank $6$ which is maximally symmetric among all such distributions. The Lie algebra of infinitesimal symmetries for $\mathfrak{D}$ at any point is $29$-dimensional with semisimple factor isomorphic to the split real form of $\mathfrak{g}_2$ (see the diagram below). The distribution $\mathfrak{D}$ can be constructed using the negatively graded Lie algebra 
\begin{equation}
    \label{n basis}
    \mathfrak{n} = \mathfrak{n}_{-1}\oplus\mathfrak{n}_{-2}\oplus\mathfrak{n}_{-3} = \langle \oline x,\oline e_0, \oline f_0\rangle \oplus \langle \oline e_{-1},\oline f_{-1},\oline \eta\rangle \oplus \langle \oline f_{-2},\oline \nu\rangle
\end{equation}
with nontrivial bracket relations
\begin{gather}
    \label{n rels}
    [\oline x, \oline e_0] = \oline e_{-1},\quad [\oline x, \oline f_{0}] = \oline f_{-1},\quad [\oline x, \oline f_{-1}] = \oline f_{-2},\\
    [\oline e_0, \oline f_0]=\oline\eta,\quad [\oline f_0, \oline f_{-1}]=\oline\nu.
\end{gather}
The distribution $\mathfrak D$ is the left-invariant distribution $\mathfrak n_{-1}$ on the simply connected Lie group $N$ with Lie algebra $\mathfrak n$. The algebra of infinitesimal symmetries of $\mathfrak{D}$ at any point of $N$ is isomorphic to $(\mathfrak{g}_2\oplus \mathbb{R})\ltimes W$, where $\mathfrak{g}_2$ is the split real form of the exceptional Lie algebra $\mathfrak g^\mathbb C_2$ and $W$ is an adjoint module of $\mathfrak{g}_2$. This Lie algebra can be visualized by means of the root system of $G_2$ as follows:

% \begin{figure}[H]
\label{fig:TG2 1}
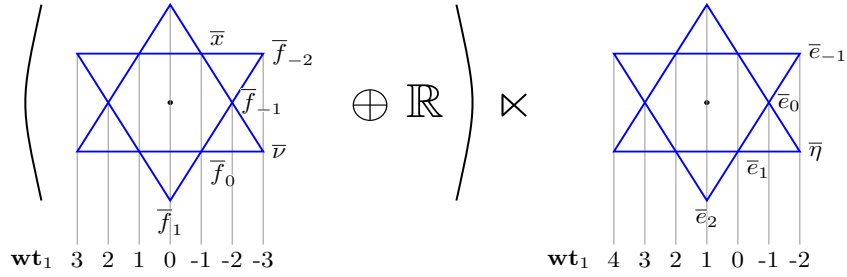
\begin{center}
\resizebox{11.5cm}{1.8cm}{
    \begin{tikzpicture}[baseline=(current bounding box.center)]    
        % LEFT PARENTHESIS
    % \draw[thick] (-1.4,-1.5) to (-1.7,0)
    %              to (-1.4,1.5);
    \draw[thick] (-4.3,-1.5) .. controls (-4.6,0) .. (-4.3,1.5);
        
    \begin{scope}[shift={(-2.5,0)}, scale=1.5]
        % Triangle 1 vertices (original angles 0, 120, 240 rotated by -30)
        \coordinate (T0) at ({cos(-30)}, {sin(-30)});
        \coordinate (T120) at ({cos(90)}, {sin(90)});
        \coordinate (T240) at ({cos(210)}, {sin(210)});
        % Triangle 2 vertices (original angles 60, 180, 300 rotated by -30)
        \coordinate (R60) at ({cos(30)}, {sin(30)});
        \coordinate (R180) at ({cos(150)}, {sin(150)});
        \coordinate (R300) at ({cos(270)}, {sin(270)});
        \filldraw[black] (0,0) circle (0.02);
        \path[name path=TA] (T0) -- (T120);
        \path[name path=TB] (T120) -- (T240);
        \path[name path=TC] (T240) -- (T0);
        \path[name path=RA] (R60) -- (R180);
        \path[name path=RB] (R180) -- (R300);
        \path[name path=RC] (R300) -- (R60);
        \path[name intersections={of=TA and RC, by=B0}];
        \path[name intersections={of=TA and RA, by=B60}];
        \path[name intersections={of=TB and RA, by=B120}];
        \path[name intersections={of=TB and RB, by=B180}];
        \path[name intersections={of=TC and RB, by=B240}];
        \path[name intersections={of=TC and RC, by=B300}];
        
        % Fixed y-coordinate for the weight rows
        \coordinate (wght1row) at (0,-1.6);
    
        % Fixed y-coordinates for top and bottom of dotted lines
        \coordinate (toprow) at (0,-0.8);
        \coordinate (bottomrow) at (0,-1.45);
    
        % Row labels
        \coordinate (labelcol) at (-1.0, 0);
        \node[anchor=east] at (labelcol |- wght1row) {$\mathrm{\bf{wt}}_1$};
    
        % Drop-down lines for wght1
        \draw[white!60!black] (R180) -- (T240 |- bottomrow);
        \draw[white!60!black] (B180) -- (B180 |- bottomrow);
        \draw[white!60!black] (B120) -- (B240 |- bottomrow);
        \draw[white!60!black] (T120) -- (R300 |- bottomrow);
        \draw[white!60!black] (B60) -- (B60 |- bottomrow);
        \draw[white!60!black] (R60) -- (R60 |- bottomrow);
        \draw[white!60!black] (B0) -- (B0 |- bottomrow);
        
        % wght1 labels
        \node at (T240 |- wght1row) {3};
        \node at (B180 |- wght1row) {2};
        \node at (B240 |- wght1row) {1};
        \node at (R300 |- wght1row) {0};
        \node at (B60 |- wght1row) {-1};
        \node at (B0 |- wght1row) {-2};
        \node at (R60 |- wght1row) {-3};
    
        \draw[blue, thick] (T0) -- (T120) -- (T240) -- cycle;
        \draw[blue, thick] (R60) -- (R180) -- (R300) -- cycle;
    
        \node[below right] at (B300) {\contour{white}{$\oline f_0$}};
        \node[right] at (B0) {\contour{white}{$\oline f_{-1}$}};
        \node[right] at (T0) {\contour{white}{$\oline \nu$}};
        \node[right] at (R60) {\contour{white}{$\oline f_{-2}$}};
        \node[below] at (R300) {\contour{white}{$\oline f_1$}};
        % \node[above left] at (B120) {\contour{white}{${F_0'}$}};
        % \node[left] at (B180) {\contour{white}{$F_{-1}'$}};
        % \node[above left] at (R180) {\contour{white}{$S'$}};
        % \node[below left] at (T240) {\contour{white}{$F_{-2}'$}};
        % \node[above] at (T120) {\contour{white}{$F_1'$}};
        \node[above right] at (B60) {\contour{white}{$\oline x$}};
        % \node[below left] at (B240) {\contour{white}{$X'$}};
    \end{scope}
    
    \begin{scope}[shift={(5,0)}, scale=1.5]
        % Triangle 1 vertices (original angles 0, 120, 240 rotated by -30)
        \coordinate (T0) at ({cos(-30)}, {sin(-30)});
        \coordinate (T120) at ({cos(90)}, {sin(90)});
        \coordinate (T240) at ({cos(210)}, {sin(210)});
        % Triangle 2 vertices (original angles 60, 180, 300 rotated by -30)
        \coordinate (R60) at ({cos(30)}, {sin(30)});
        \coordinate (R180) at ({cos(150)}, {sin(150)});
        \coordinate (R300) at ({cos(270)}, {sin(270)});
        \filldraw[black] (0,0) circle (0.02);
        \path[name path=TA] (T0) -- (T120);
        \path[name path=TB] (T120) -- (T240);
        \path[name path=TC] (T240) -- (T0);
        \path[name path=RA] (R60) -- (R180);
        \path[name path=RB] (R180) -- (R300);
        \path[name path=RC] (R300) -- (R60);
        \path[name intersections={of=TA and RC, by=B0}];
        \path[name intersections={of=TA and RA, by=B60}];
        \path[name intersections={of=TB and RA, by=B120}];
        \path[name intersections={of=TB and RB, by=B180}];
        \path[name intersections={of=TC and RB, by=B240}];
        \path[name intersections={of=TC and RC, by=B300}];
        
        % Fixed y-coordinate for the weight rows
        \coordinate (wght1row) at (0,-1.6);
    
        % Fixed y-coordinates for top and bottom of dotted lines
        \coordinate (toprow) at (0,-0.8);
        \coordinate (bottomrow) at (0,-1.45);
    
        % Row labels
        \coordinate (labelcol) at (-1.0, 0);
        \node[anchor=east] at (labelcol |- wght1row) {$\mathrm{\bf{wt}}_1$};
    
        % Drop-down lines for wght1
        \draw[white!60!black] (R180) -- (T240 |- bottomrow);
        \draw[white!60!black] (B180) -- (B180 |- bottomrow);
        \draw[white!60!black] (B120) -- (B240 |- bottomrow);
        \draw[white!60!black] (T120) -- (R300 |- bottomrow);
        \draw[white!60!black] (B60) -- (B60 |- bottomrow);
        \draw[white!60!black] (R60) -- (R60 |- bottomrow);
        \draw[white!60!black] (B0) -- (B0 |- bottomrow);
        
        % wght1 labels
        \node at (T240 |- wght1row) {4};
        \node at (B180 |- wght1row) {3};
        \node at (B240 |- wght1row) {2};
        \node at (R300 |- wght1row) {1};
        \node at (B60 |- wght1row) {0};
        \node at (B0 |- wght1row) {-1};
        \node at (R60 |- wght1row) {-2};

        \draw[blue, thick] (T0) -- (T120) -- (T240) -- cycle;
        \draw[blue, thick] (R60) -- (R180) -- (R300) -- cycle;
    
        \node[below right] at (B300) {\contour{white}{$\oline e_1$}};
        \node[right] at (B0) {\contour{white}{$\oline e_{0}$}};
        \node[right] at (T0) {\contour{white}{$\oline \eta$}};
        \node[right] at (R60) {\contour{white}{$\oline e_{-1}$}};
        \node[below] at (R300) {\contour{white}{$\oline e_2$}};
        % \node[above left] at (B120) {\contour{white}{${E_1'}$}};
        % \node[left] at (B180) {\contour{white}{$E_{0}'$}};
        % \node[above left] at (R180) {\contour{white}{$T'$}};
        % \node[below left] at (T240) {\contour{white}{$E_{-1}'$}};
        % \node[above] at (T120) {\contour{white}{$E_2'$}};
        % \node[above right] at (B60) {\contour{white}{$Y$}};
        % \node[below left] at (B240) {\contour{white}{$Y'$}};
    \end{scope}
        
        % MATH SYMBOL
        \node at (1.3,0) {\Huge $\oplus\ \mathbb R\hspace{0.7cm} \ltimes$};
        
        % RIGHT PARENTHESIS
    \draw[thick] (1.5,-1.5) .. controls (1.8,0) .. (1.5,1.5);
    \alttext{Two weight diagrams side by side, showing how a Lie algebra decomposes into a reductive factor and a nilpotent factor. Inside a large pair of parentheses, reading left to right: a first star-shaped diagram, then a direct-sum sign and a factor R; the parenthesis closes, followed by a semidirect-product sign and a second star-shaped diagram. So the whole expression is (first star direct sum R) semidirect product (second star). Each star is a Star-of-David hexagram of two overlapping triangles, the root diagram of the exceptional root system G 2, with a dot at the center for the zero weight. The left star is a copy of the Lie algebra g 2; the right star is an adjoint module of g 2. A single weight, wt 1, is read off the vertical columns, with an integer scale along the bottom: 3 down to minus 3 on the left, and 4 down to minus 2 on the right. Vertices are labeled by basis vectors: on the left x, nu, and f subscript i (from f subscript 1 down to f subscript minus 2); on the right e subscript i and eta.}
    \end{tikzpicture}
}
\end{center}
 % \caption{A decomposition of the symmetry algebra of $\mathfrak D$ into its reductive and nilpotent factors, represented with weight diagrams}
% \end{figure}

The most symmetric previously known example of a $(3,6,\ldots)$ distribution on an $8$-dimensional manifold admits $23$-dimensional infinitesimal symmetries and was of \textit{maximal class} \cite{doubrov2025large}. Maximality of class is a technical condition required for the application of the \textit{symplectification procedure} Doubrov and Zelenko developed in their papers \cite{doubrov2008rank3} \cite{doubrov2016JacobiCurves}.\footnote{We outline the symplectification procedure and define maximality of class in Section \ref{section: sympl}.} The distribution $\mathfrak{D}$ was discovered in attempts to prove that all bracket-generating rank $3$ distributions with $6$-dimensional square are of maximal class at a generic point, a conjecture analogous to the results for rank $2$ distributions in \cite{day2025canonicalframes}. However, this conjecture holds only in dimension $6$, which yields a parabolic geometry of type $\big(\mathrm{SO}(3,4),P_{3}\big)$, and in dimension $7$. The distribution $\mathfrak{D}$ arose as the first counterexample to the conjecture; namely, $\mathfrak{D}$ is nowhere of maximal class. While $\mathfrak{D}$ is notable for its symmetries, similar rank $3$ distributions $D$ with $\mathrm{rank}\, D^2=6$  of nonmaximal class can be also constructed in dimensions greater than $8$. 
This raises the natural question: In dimensions greater than $8$, is the maximally symmetric rank $3$ distribution with $6$ dimensional square of maximal or nonmaximal class? Our experiments in dimensions 9 and 10 show that, in contrast to dimension 8, the maximum dimension of the symmetry algebra is achieved by distributions of maximal class. We therefore conjecture that this phenomenon persists in all dimensions $n > 8$.

We refer to a rank $k$ distribution on an $n$-dimensional manifold as a $(k,n)$ distribution. All distributions are assumed smooth and bracket-generating (i.e., iterated brackets of local sections span the tangent bundle). For $i>0$, we use $D^{-i}$ as a filtrand in the \textit{weak derived flag of $D$}, which is a distribution defined by 
\begin{equation}
    D^{-1} = D,\quad\text{and}\quad  D^{-i} = D^{1-i} + [D, D^{1-i}]\text{ for }i>1
\end{equation}
where the brackets are Lie brackets of local sections. The \textit{small growth vector} of $D$ is $$\big(\mathrm{rank}(D^{-1}),\mathrm{rank}(D^{-2}),\ldots, \mathrm{rank}(D^{-\mu})\big),$$ where $\mu$ is the least integer such that $D^{-\mu-1}=D^{-\mu}$. In this notation, $\mathfrak D$ is maximally symmetric among the family of all $(3,6,8)$ and $(3,6,7,8)$ distributions. Surprisingly, the distribution $\mathfrak D$ is maximally symmetric in this family of distributions despite having the \textit{shorter} of the two small growth vectors. To the best of our knowledge, this provides the first example of such a phenomenon.

In Section \ref{section: sympl}, we review the symplectification procedure for rank $3$ distributions originally given in \cite{doubrov2008rank3}. The symplectification procedure is a method for assigning to a rank $3$ distribution $D\to M$ of maximal class a bundle over $M$ with a canonical frame, thereby obtaining local invariants which classify $D$ up to local diffeomorphism. This method was dubbed the \textit{symplectification procedure} because it utilizes the canonical symplectic structure on $T^*M$ to obtain from $D$ a distinguished submanifold of $T^*M$. Projectivizing then yields a filtered submanifold of $\mathbb{P}T^*M$, to which one can apply Tanaka theory \cite{DayDoubrovZelenko2026} or other methods (described explicitly in \cite{doubrov2008rank3} and \cite{doubrov2009london}) to assign a bundle with canonical frame. We say $D$ is \textit{of maximal class} at a point $p$ if the maximal filtrand of this submanifold coincides with the restriction of contact distribution on $\mathbb{P}T^*M$ at a generic point of the fiber over $p$.

The primary advantage of the symplectification procedure over other approaches to the local geometry of distributions is in the simplicity of its basic invariant, the \textit{Jacobi symbol}, which is encoded by a skew Young diagram; Jacobi symbols are easily classified. Further, because only finitely many Jacobi symbols arise from those distributions of a fixed rank on an ambient manifold of fixed dimension, the Jacobi symbol of a smooth distribution is generically constant. 

This differs from the situation in Tanaka theory, where the basic invariant, called the \textit{Tanaka symbol}, is a graded nilpotent Lie algebra which models the filtered tangent space. Although such algebras can be classified in low dimensions, the computations quickly become infeasible as the dimension of the ambient manifold grows. Even in small dimensions, the Tanaka symbols are classified by continuous parameters, so that generic distributions do not have constant Tanaka symbol and therefore cannot be treated by standard Tanaka theory. For example, the Tanaka symbol of a $(3,6,8)$ distribution is determined up to isomorphism by a plane in $\mathfrak{sl}_3(\mathbb R)$ up to the adjoint action of $\mathrm{SL}_3(\mathbb R)$ on its Lie algebra. An explicit description of the resulting space of orbits is difficult to obtain and involves multiple continuous moduli. This problem is partially treated in \cite{Procesi1976, Formanek1979, Teranishi1986}. The resulting size of the set of nonisomorphic Tanaka symbols obscures which Tanaka symbol yields the maximally symmetric $(3,6,8)$ distribution. In fact, the Tanaka symbol of $\mathfrak D$ corresponds to the plane in $\mathfrak{sl}_3(\mathbb R)$ of nilpotent endomorphisms of rank $1$ having the common image; see Appendix \ref{Appendix} for details.

For distributions of nonmaximal class, the basic invariant of the symplectification procedure is not only the Jacobi symbol, but the \textit{Jacobi-Tanaka algebra}, which is the Tanaka symbol of the filtered submanifold resulting from the symplectification procedure. In general, the Jacobi symbol can be recovered from the Jacobi-Tanaka algebra and not vice versa, but the two invariants are equivalent for distributions of maximal class.

In Section \ref{section: dim 6-7}, we prove our first result, which is an analogue of the results on rank $2$ distributions in \cite{day2025canonicalframes} and is proven using arguments similar to those in that work:
\begin{theorem}
    \label{Thm: (3,6) and (3,7)}
    Let $n=6$ or $7$. Every bracket-generating $(3,n)$ distribution with $6$-dimensional square is of maximal class at a generic point.
\end{theorem}

In Section \ref{section: dim 8}, we provide a $(3,8)$ distribution $\mathfrak D$ with $6$-dimensional square which is nowhere of maximal class. We then use Tanaka theory and the results of the works \cite{doubrov2008rank3} and \cite{doubrov2025large} of the second and third authors of this paper to prove the main theorem:

\begin{theorem}
    \label{Thm: main theorem}
    The $(3,8)$ distribution $\mathfrak D$ described after \eqref{n rels} has at every point a $29$-dimensional algebra of infinitesimal symmetries isomorphic to a semidirect sum
\[
    (\mathfrak{g}_2\oplus\mathbb{R})\ltimes W,
\]
where $\mathfrak{g}_2$ is the split real form of the exceptional Lie algebra $\mathfrak g_2^\mathbb C$, $W$ is an adjoint module of $\mathfrak{g}_2$ and $\mathbb R$ acts on $W$ by scaling. The germ of $\mathfrak D$ is the unique maximally symmetric germ of among $(3,8)$ distributions with $6$-dimensional square.
\end{theorem}

In Section \ref{octonions_sec}, we offer an interpretation of the distribution $\mathfrak D$ using the split octonions $\oct$. More specifically, we use the tangent bundle $T\oct$, which is naturally an algebra with automorphism group given by $ G_2\ltimes W$, where $G_2$ is the split real form of the exceptional Lie group $G_2^\mathbb C$ and $W$ is an adjoint module of $G_2$. The tangent bundle $TK\subseteq T\oct$ of the cone $K$ of pure imaginary zero divisors is preserved by the automorphisms of $T\oct$. We exhibit a rank $7$ distribution $\Delta'$ on $TK$ that is invariant under the automorphisms of $T\oct$; quotienting $\Delta'$ by its rank $4$ Cauchy characteristic, we obtain a distribution locally equivalent to $\mathfrak D$.

In Section \ref{section_corank}, we exploit the nonmaximality of class for $\mathfrak D$ to derive the following control-theoretic consequences: all abnormal extremal trajectories of $\mathfrak{D}$ have corank at least $2$, and generic abnormal extremal trajectories have corank exactly $2$. (The corank of an abnormal extremal trajectory is the codimension of the image of the differential of the endpoint map at this trajectory.) In contrast, for generic $(3,n)$ distributions with $n \geq 6$, through a generic point there passes at least one abnormal extremal trajectory of corank $1$.

Finally, in the second part of Appendix \ref{Appendix}, for completeness, we classify the Tanaka symbols of $(3,6,7,8)$ distributions and find another interesting phenomenon: There are exactly two Tanaka symbols, both of which correspond to the same Jacobi-Tanaka algebra so that all such distributions are of maximal class. As previously mentioned, the flat distribution of one of the symbols is the most symmetric distribution among all bracket-generating (3,8) distributions of maximal class with $6$-dimensional square, and it has 23-dimensional symmetry group.
The symmetry group of the flat distribution of the second Tanaka symbol is the split real form of $C_3$, so this distribution is the flat model for the parabolic geometry $C_3/P_{12}$ (i.e., associated with the grading of $C_3$ where both short roots are marked). To our knowledge, this is the first example in which the symplectification procedure unifies a parabolic geometry with a broader class of distributions sharing its Jacobi-Tanaka algebra, and the flat model of the unified class is \textit{not} the parabolic one.

%\IZ{The control aspect of our example should be briefly given here, i.e., that in this example all abnormal extremal trajectories through a point have corank at least $2$. This can be shown using the same arguments as in the proof of Theorem 4.1 of \cite{day2025canonicalframes} . One only needs to take into account that in the real analytic category---which is the case for left-invariant distributions on Lie groups---the relation analogous to (4.3) becomes an equality rather than an inequality.}

\section{Symplectification for Rank 3 Distributions and Maximality of Class}
\label{section: sympl}
We now summarize the symplectification of rank 3 distributions with $6$-dimensional square first constructed in \cite{doubrov2008rank3}. Let $D$ be a rank 3 distribution on an $n$-dimensional manifold $M$, $n\geq 6$, and assume the derived distribution $D^2=D + [D,D]$ has dimension $6$.

\subsection{The Characteristic Line Distribution}
\label{subsection:characteristic}
Let $s$ be the tautological (or \textit{Liouville}) $1$-form on the cotangent bundle $T^*M$, and let $\sigma = ds$ be the canonical symplectic $2$-form. Passing to the fiberwise projectivization $\mathbb{P}T^*M$ of the cotangent bundle, the $1$-form $s$ passes to a conformal class $\oline{s}$ of 1-forms which defines a natural contact structure on $\mathbb{P}T^*M$.

The annihilator of $D$ is defined by 
\begin{equation}
    D^\perp = \{(p,q)\in T^*M: p\in T_q^*M, p(v) = 0\text{ for all }v\in D(q)\}
\end{equation}
and forms a vector subbundle of $T^*M$; similarly define $(D^2)^\perp$. 

Let $\mathbb P(D^\perp)$ be the fiberwise projectivization of the bundle $D^\perp$ with the canonical projection $\pi: \mathbb P (D^\perp)\to M$. Because $D$ has rank $3$, the submanifold $\mathbb P (D^\perp)$ has codimension $3$ in $\mathbb{P}T^*M$ and therefore even dimension. The contact structure defined by $\oline{s}$ on $\mathbb PT^*M$ restricts to a corank $1$ distribution $\xi := \ker(\oline {s}|_{D^\perp})$ on $\mathbb P (D^\perp)$ 
with a conformal class of antisymmetric $2$-forms $\oline{\sigma}|_\xi$ on $\xi$. 
\begin{definition}
\label{abnorm_def}
A nontrivial \footnote{nontrivial means that its image is not a point} Lipschitz curve $\gamma:[0, T]\to \mathbb P D^\perp$ is called an \emph{abnormal extremal} of $D$ if $\gamma'(t)\in \ker \oline\sigma(\gamma(t))|_{\xi}$ for almost every $t$. The projection of $\gamma$ to $M$ is called an \emph{abnormal extremal trajectory} of $D$. \footnote{An equivalent definition of abnormal extremal trajectories via the endpoint map is given in Section \ref{section_corank}.}
\end{definition}
Note that because $\xi(\lambda)$ has odd dimension for each $\lambda\in \mathbb PD^\perp$, the kernel of the antisymmetric $2$-form $\oline{\sigma}|_{\xi}$ must have dimension at least $1$. In \cite{doubrov2008rank3} (see also Appendix \ref{Appendix} and \eqref{Characteristic_rank3} below), it is demonstrated that $\ker(\oline{\sigma}|_{\xi})$ has dimension exactly $1$ on the following subset $W_D$ of $\mathbb P D^\perp$:
\begin{equation}
    W_D:=\mathbb{P}\left(D^\perp\setminus\left(D^2 \right)^\perp\right)\subseteq \mathbb{P}T^*M,
\end{equation}
so that 
%$(W_D,\xi,\oline\sigma)$ is an even contact manifold. 
$\big(\xi,\oline{\sigma}|_\xi\big)$ defines an \emph{even contact structure} on $W_D$; that is, a corank $1$ distribution with a conformal class of antisymmetric $2$-forms with $1$-dimensional kernel at each point \cite[\S 1.1, \S2.1]{doubrov2008rank3}. Now define
\begin{equation}
    \mathcal C = \ker(\oline{\sigma}|_{\xi}),
\end{equation}
a canonical line distribution on $W_D$ called the \emph{characteristic line distribution}. The line $\mathcal{C}$ is also generated by Cauchy characteristic vector fields of the even contact distribution $\xi$ on $W_D$. A nontrivial curve tangent to $\mathcal{C}$ in $W_D$ is obviously an abnormal extremal in the sense of Definition \ref{abnorm_def}. Such abnormal extremals will be called \emph{regular abnormal extremals associated with $D$} to distinguish them from abnormal extremals that pass through $\mathbb P((D^2)^\perp)$. By construction, exactly one abnormal extremal germ passes through each point of $W_D$.

\subsection{The Jacobi Flag of a Distribution}
\label{subsection: CanFlag}
The manifold $W_D$ also carries a canonical \emph{vertical distribution} $\mathcal{V}$ given by the tangents to the fibers of the projection $\pi:W_D\to M$. The sum $\mathcal{C}\oplus \mathcal{V}$ then forms a canonical distribution which splits into involutive subdistributions; such a splitting was called a \emph{pseudo-product structure} by N. Tanaka in \cite{Tanaka1970}. Define
\begin{gather}
    \label{J1 def}
    \mathcal{J}^{(1)} = \mathcal{C}\oplus \mathcal{V}
    \\
    \label{J(-i) def}
    \mathcal{J}^{(-i)} = \mathcal{J}^{(1-i)} + [\mathcal{C},\mathcal{J}^{(1-i)}]\quad\text{for all }i\geq 0
\end{gather}
A small computation shows that for all $\lambda\in W_D$,
\begin{equation}
    \label{J0 fact}
    \mathcal{J}^{(0)}(\lambda) = (\pi^*D)_\lambda = \{v\in T_\lambda W_D: T\pi(v) \in D\}
\end{equation} is the lift of $D$ over $\pi$. Also define
\begin{equation}
    \mathcal{J}^{(i)} =(\mathcal{J}^{(1-i)})^{\angle}\quad \text{for all }1<i\leq m+1
\end{equation}
to be the skew-orthogonal complement of $\mathcal{J}^{(1-i)}$ with respect to $\oline{\sigma}$. One should note that $\mathcal{J}^{(1)} = \big(\mathcal{J}^{(0)}\big)^\angle$. Also, it is shown in \cite[\S2.2]{doubrov2008rank3} that for each $i\geq 1$,
\begin{equation}
    \label{osculation property}
    \mathcal{J}^{(i)}(\lambda) = \big\{Y(\lambda): Y\text{ is a local section of }\mathcal{J}^{(i-1)}\text{ with } [Y,\mathcal{C}]_\lambda\subseteq \mathcal{J}^{(i-1)} \big\}.
\end{equation}
 We call $\{\mathcal{J}^{(i)}\}$ the \emph{Jacobi flag of $D$}.
 Let $\mathcal{R}_D$ be the generic subset of $W_D$ where the rank of $\mathcal{J}^{(i)}$ is locally constant for each $i$; we will now use $\pi:\mathcal{R}_D\to M$ for the canonical projection from $\mathcal{R}_D$ rather than from $W_D$. On each connected component of $\mathcal{R}_D$, one obtains a flag of distributions of constant rank
\begin{equation}
    \label{Char flag}
    \mathcal{J}^{(m+1)}\subsetneq \mathcal{J}^{(m)}\subsetneq \cdots\subsetneq \mathcal{J}^{(1)}= \mathcal{C}\oplus \mathcal{V} \subsetneq \mathcal{J}^{(0)} = \pi^*D \subsetneq \cdots\subsetneq \mathcal{J}^{(-m)}
\end{equation}
where $m$ is the least natural number with $\mathcal{J}^{(-m)} = \mathcal{J}^{(-m-1)}$ on the connected component of $\mathcal{R}_D$. For fixed $k\in \mathbb{N}$, $i>0$, and $x\in M$, the set
\begin{equation}
    \{\lambda\in \pi^{-1}(x): \dim\mathcal{J}^{(-i)}>k\}
\end{equation}
is Zariski open in the fiber $\pi^{-1}(x)$. Therefore, the connected components of $\mathcal{R}_D$ are saturated with respect to the projection $\pi$, and the integer $m$ depends only upon the point $x\in M$; we will write $m(x)$ for this integer. Wherever we write $m$, omitting the argument $x$, we assume the function $m(x)$ is constant by restricting to the image of a connected component of $\mathcal{R}_D$ under $\pi$.

Because $\mathcal C$ is generated by characteristic vector fields of $\xi$, each $\mathcal{J}^{(i)}$ is included in $\xi$. The distribution $D$ is said to be \emph{of maximal class} at a point $x\in M$ if $\mathcal{J}^{(-m(x))}(\lambda)=\xi(\lambda)$. Among all rank $3$ distributional germs, those of maximal class are generic: maximality of class is a nonempty Zariski open condition on a certain jet of the distribution. The symplectification procedure for rank 3 distributions developed in \cite{doubrov2008rank3} assigned to $D$ a canonical frame on a fiber bundle near points of maximal class.

Intersecting the flag ${\mathcal{J}^{(i)}}$ with the vertical distribution gives the distributions
\begin{equation}
    \label{def V_i}
    V_i = \mathcal{J}^{(i)}\cap \mathcal V \quad\text{for each }0\leq i \leq m+1.
\end{equation}
Since $\mathcal{V}\subseteq \mathcal{J}^{(1)}$, we have that $V_1=\mathcal V$, so we can use \eqref{J1 def} to write
\begin{equation}
    \label{J(i) = V_i + C}
    \mathcal{J}^{(i)} = V_i\oplus \mathcal{C}\quad\text{for each }i\geq 1
\end{equation}
Observe that \eqref{J0 fact} implies that $[V_0,\mathcal{J}^{(0)}]\subseteq \mathcal{J}^{(0)}$. The proof of the following lemma exploits this fact and the involutivity of $V_0=\mathcal{V}$; we omit the proof here because it is essentially identical to that of Lemma $2$ in \cite{doubrov2009london}.
\begin{lemma}
    \label{Lemma: invol cond} 
    {\ }
    \begin{enumerate}
        \item $[V_i,V_i]\subseteq V_i$ for each $0\leq i \leq m+1$; that is, $V_i$ is involutive
        \item $[V_i,\mathcal{J}^{(1-i)}]\subseteq \mathcal{J}^{(1-i)}$ for each $1\leq i \leq m+1$.
    \end{enumerate}
\end{lemma}

The following weaker corollary will later be useful and can be demonstrated by applying the involutivity conditions and the decomposition \eqref{J(i) = V_i + C}.

\begin{corollary}
    \label{Corollary: Involutivity}
    For each $1-m\leq i,j  \leq m$ satisfying $i+j\geq 1$, we have that
    \begin{equation}
        [\mathcal{J}^{(i)},\mathcal{J}^{(j)}] \subseteq \mathcal{J}^{(\min\{i,j\}-1)}
    \end{equation}
\end{corollary}

% In section \ref{section: dim 6-7}, we will demonstrate that all rank 3 distributions with $6$-dimensional square in dimensions $6$ and $7$ are of maximal class at a generic point of the base manifold. On the other hand, in section \ref{section: dim 8}, we will provide an example of a rank 3 distribution with $6$-dimensional square on an $8$ dimensional manifold which is not of maximal class. 

\subsection{The Jacobi Symbol}
\label{subsection: Jacobi Symbol}

We now introduce a skew Young diagram associated to $D$ called the \textit{Jacobi symbol}. Consider the tuple of natural numbers $T=(t_{-m},t_{1-m},\ldots, t_{m-1}, t_m)$, where
\begin{equation}
\label{t_i}
    t_i = \mathrm{dim}\ \mathcal{J}^{(i)} - \mathrm{dim}\ \mathcal{J}^{(i+1)}
\end{equation}

 Using the definition of $\mathcal{J}^{(-1)}$ and the fact that $\mathcal{J}^{(0)}= \pi^*D$, we have that $t_0 = 2$. Further, it is demonstrated in \cite[Proposition 2]{doubrov2008rank3} that $t_1=2$ precisely if $\mathrm{dim}\ D^2 = 6$. It is also easy to see from the definition of $\mathcal{J}^{(-i)}$ that $t_{i+1}\leq t_i$ and $t_{-i} = t_i$ for each $i>0$. Therefore, $T$ is a symmetric tuple of the form
 \begin{equation}
    \big(
    \overbrace{1,1,\ldots,1}^{\ell},
    \overbrace{2,2,\ldots,2}^{2k+1},
    \overbrace{1,1,\ldots,1}^{\ell}
    \big)
 \end{equation}
 where $k\geq 1$ and $\ell\geq 0$. \footnote{Our convention for the integer $k$ in the pair $(k,\ell)$ defining a skew Young diagram is one less than that in \cite{doubrov2008rank3}. We choose this notation to simplify later indexing.} 
 % \IZ{It is not so essential but $k$ in our 2008 paper with Boris \cite{doubrov2008rank3} in the description of the Young diagram we use was $1$ larger than the $k$ here. For the sake of consistency and for the convenience of a reader who would like to examine also this paper \cite{doubrov2008rank3}, I would modify this accordingly or at least mention it as a footnote.} 
 
 The tuple $T$ can be equivalently written either as the pair of integers $(k,\ell)$ or as the \emph{skew Young diagram associated to $D$}
 \begin{equation}
 \label{Jacobi Symb Diagram}
\begin{tikzpicture}[baseline=(current bounding box.center)]
    % Place the ytableau in a node
    \node (table) {
        \begin{ytableau}
            *(white) & {\cdots} & *(white) & *(white) & {\cdots} & *(white) & \none & \none & \none \\
            \none & \none & \none & *(white) & {\cdots} & *(white) & *(white) & {\cdots} & *(white) \\
        \end{ytableau}
    };

    % Brackets and labels over the columns
    % First \ell over first 3 columns
    \draw[decorate,decoration={brace,amplitude=5pt},yshift=1.2cm]
        ([xshift=0.1cm]table.north west) -- ++(1.6cm,0) node[midway,above=6pt] {\(\ell\)};

    % 2k+1 over middle 3 columns
    \draw[decorate,decoration={brace,amplitude=5pt},yshift=1.2cm]
        ([xshift=1.8cm]table.north west) -- ++(1.6cm,0) node[midway,above=6pt] {\(2k+1\)};

    % Second \ell over last 3 columns
    \draw[decorate,decoration={brace,amplitude=5pt},yshift=1.2cm]
        ([xshift=3.5cm]table.north west) -- ++(1.6cm,0) node[midway,above=6pt] {\(\ell\)};

    % Bottom brace spanning entire diagram
    \draw[decorate,decoration={brace,amplitude=5pt,mirror},yshift=-1.2cm]
        ([xshift=0.1cm]table.south west) -- ++(5.0cm,0)
        node[midway,below=6pt] {\(2m+1\)};
    \alttext{A two-row skew Young diagram: a top row of a length ell block then a length 2k
plus 1 block, over a bottom row offset right of a length 2k plus 1 block then a
length ell block, overlapping in a central 2k plus 1 columns. Braces above mark
ell, 2k plus 1, and ell; a brace below marks the total width, 2m plus 1.}
\end{tikzpicture}
\end{equation}
This skew Young diagram is also known as the \emph{Jacobi symbol} associated to $D$. Note that the Jacobi symbol of $D$ is locally constant on $M$. 

Observe that $\mathrm{dim}(\mathcal{R}_D) = 2n-4$ and that $m=k+\ell$. Because $\mathcal{J}^{(-k-\ell)}$ is contained in the even contact distribution $\xi$, we have that 
\begin{equation}
    \mathrm{dim}\ \mathcal{J}^{(-k-\ell)} = \mathrm{dim}\ \mathcal{J}^{(0)}+2k+\ell = n-1+2k+\ell \leq 2n-5
\end{equation}
or equivalently that \begin{equation}
    \label{num_of_boxes}
    2k+\ell\leq n-4,
\end{equation}where equality obtains precisely when $D$ is of maximal class and precisely when the skew Young diagram \eqref{Jacobi Symb Diagram} has $2n-6$ boxes.

The significance of the Jacobi symbol can be described as follows: The leaf space for the foliation of $\mathcal{R}_D$ by abnormal extremals is equipped with the natural contact structure. By pushing flag of distributions $\{\mathcal{J}^{(i)}\}$ forward to this leaf space one obtains at each fiber of this contact distribution a curve of symplectic flags, i.e. flags consisting of isotropic and coisotropic subspaces which are stable under the operation of skew-symmetric complement. The Jacobi symbol is the principal invariant of this curve up to the action of the conformal symplectic group on the fibers of the contact distribution \cite{Doubrov2012curves} \cite{doubrov2016JacobiCurves}. 

For our purposes, the essential facts about the skew Young diagram are captured by the following proposition, which is proven in \cite{doubrov2016JacobiCurves}

\begin{prop}
\label{Prop: JacobiFrame}
    Let $D$ be a $(3,n)$ distribution on $M$ with constant Jacobi symbol $(k,\ell)$. Fix a nonvanishing local section $x$ of the line distribution $\mathcal{C}$ near $\lambda$. One can choose a tuple of locally defined vector fields $\big(\{e_i\}_{i=-k}^{m},\{f_j\}_{j=-m}^k,\eta\big)$ on $\mathcal{R}_D$ which satisfy the following properties:
    \begin{enumerate}
        \item For each $-k-\ell\leq i \leq m$, 
        \[
            \mathcal J^{(i)} = \begin{cases}
                \mathcal J^{(1+i)}+\mathrm{span}(\{e_i\})& \text{if } k < i
                \\
                \mathcal J^{(1+i)}+\mathrm{span}(\{e_i,f_i\}) & \text{if }-k\leq i \leq k
                \\
                \mathcal J^{(1+i)}+\mathrm{span}(\{f_i\})& \text{if } i < -k
            \end{cases}
        \]
        \item For each $-k\leq i<m$ and each $-m\leq j<k$,
        \[
            e_i = [x,e_{i+1}] \mod \langle x\rangle \quad \text{and}\quad f_j = [x,f_{j+1}] \mod \langle x\rangle,
        \]
        and in the case that $i,j>0$, we have $e_i,f_j\in \mathcal{V}$
        \item Define $\eta = (-1)^{m}[e_{m},f_{-m}]$. The vector field $\eta$ is not contained in $\mathcal J^{(-m)}$, and for each $-k \leq i\leq m$,
        \[
            [e_i,f_{-i}] \equiv (-1)^i\eta \mod \mathcal{J}^{(-m)}
        \]
        
    \end{enumerate}
\end{prop}
The contents of this Proposition are well-summarized by decorating the Jacobi symbol as follows:

% \begin{figure}[H]
\begin{equation}
    \label{figure: decorated young}
    \begin{tikzpicture}[baseline=(current bounding box.center)]
        % Place the ytableau in a node
        \node (table) {
            \ytableausetup{boxsize=2.5em}
            \begin{ytableau}
                \scriptstyle e_{m} & {\cdots} & \scriptstyle e_{k+1} & \scriptstyle e_{k} & {\cdots} & \scriptstyle e_{-k} & \none & \none & \none \\
                \none & \none & \none & \scriptstyle f_{k} & {\cdots} & \scriptstyle f_{-k} & \scriptstyle f_{-k-1} & {\cdots} & \scriptstyle f_{-m} \\
            \end{ytableau}
        };
    
        % % Brackets and labels over the columns
        % % First \ell over first 3 columns
        % \draw[decorate,decoration={brace,amplitude=5pt},yshift=1.2cm]
        %     ([xshift=0.3cm]table.north west) -- ++(2.5cm,0) node[midway,above=6pt] {\(\ell\)};
    
        % % 2k+1 over middle 3 columns
        % \draw[decorate,decoration={brace,amplitude=5pt},yshift=1.2cm]
        %     ([xshift=3cm]table.north west) -- ++(2.5cm,0) node[midway,above=6pt] {\(2k+1\)};
    
        % % Second \ell over last 3 columns
        % \draw[decorate,decoration={brace,amplitude=5pt},yshift=1.2cm]
        %     ([xshift=5.7cm]table.north west) -- ++(2.5cm,0) node[midway,above=6pt] {\(\ell\)};
    \alttext{A two-row skew Young diagram (decorated Jacobi symbol): a top row of boxes e sub m down to e sub minus k over a bottom row f sub k down to f sub minus m, offset so the rows overlap in a central 2k+1 column block (e sub i above f sub i) and the full width is 2m+1 columns.}
    \end{tikzpicture}
\end{equation}

% \end{figure}

The first and second items imply that for each $0\leq r$, the subspace $V_r$ defined by \eqref{def V_i} can be written as
\[
    V_r = V_{m+1}\oplus \mathrm{span}\big(\{e_i:r\leq i\leq m\}\cup \{f_j: r\leq j\leq k\}\big).
\]

\subsection{The Jacobi-Tanaka Algebra}
Although the Jacobi symbol is the basic invariant of the filtered manifold $\mathcal R_D'$ obtained from symplectification, a finer invariant called the \textit{Jacobi-Tanaka algebra} can be defined, which is simply the Tanaka symbol of the filtration $\{\mathcal{J}^{(-i)}\}$ after quotienting by a common Cauchy characteristic and reindexing. 

The Jacobi symbol can be recovered by truncating the Jacobi-Tanaka algebra; that is, by quotienting the algebra by elements of graded weight at most $-2m-3$. For distributions of maximal class, the two invariants agree, but for distributions of nonmaximal class, the Jacobi-Tanaka algebra cannot be recovered from the Jacobi symbol alone.

\subsubsection{Reindexing the Jacobi Flag}
Recall that $m=k+\ell$. Our choice of indexing for the flag $\{\mathcal{J}^{i}\}_{i=-m}^m$ is convenient for the statement of the involutivity conditions in Lemma \ref{Lemma: invol cond}. However, in order to obtain a filtration which respects the bracket, we also define
\begin{equation}
    \label{F def 1}
    \mathcal{F}^{-i} = \mathcal{J}^{(m+1-i)} = \mathcal{J}^{(k+\ell+1-i)}\quad\text{for each }1\leq i \leq \ell+1
\end{equation}
Extending this filtration for $\ell+1<i$ by
\begin{equation}
    \label{F def 2}
    \mathcal{F}^{-i} = \mathcal{F}^{1-i}+\sum_{j=1}^{i-1} [\mathcal{F}^{-j},\mathcal{F}^{j-i}]
\end{equation}
We obtain a filtration $\{\mathcal{F}^{-i}\}_{i\in \mathbb{N}}$ on $\mathcal{R}_D$ which respects brackets; that is,
\begin{equation}
    \label{bracket-resp cond}
    [\mathcal{F}^{-i},\mathcal{F}^{-j}]\subseteq \mathcal{F}^{-i-j}\quad\text{for all }i,j\in \mathbb{N}.
\end{equation}
This can be shown using the involutivity conditions of Lemma \ref{Lemma: invol cond}.

Now let $\mathcal{R}_D'$ be the generic subset of $\mathcal{R}_D$ where each of the filtrands $\{\mathcal{F}^{-i}\}$ has locally constant rank. The following proposition illuminates the structure of the filtration $\{\mathcal F^{-i}\}_{i>0}$ and its relation to the Jacobi symbol. Recall that $m = k+ \ell$.
\begin{prop}
    \label{Prop: Technical}
    For each $\lambda\in\mathcal R_D'$, the following hold:
    \begin{enumerate}[label=(\roman*)]
        \item $\mathcal{F}^{-i}(\lambda) = \mathcal{J}^{(m+1-i)}(\lambda)$ for each $1\leq i\leq 2m+1$
        \item $\mathcal{F}^{-2m-2}(\lambda) = \mathcal{F}^{-2m-1}(\lambda)\oplus \big\langle \eta(\lambda)\big\rangle$
        \item $[V_{m+1},\mathcal{F}^{-i}](\lambda)\subseteq \mathcal{F}^{-i}(\lambda)$ for each $i\in\mathbb N$
    \end{enumerate}
    where $\eta$ is a (locally defined) vector field as in Proposition \ref{Prop: JacobiFrame}.
\end{prop}

Let us prove each part of the proposition above in sequence.
\begin{proof}[Proof of (i)]
    For each $1\leq i\leq \ell+1$, the claim holds by the definition of $\mathcal{F}^{-i}$ given in \eqref{F def 1}. For higher $i$, we use induction. Fix some $\ell+1\leq a \leq 2m$, and assume the conclusion holds for each $i\leq a$. 

    By Remark \ref{Remark: generators of m}, we can write
    \begin{gather}
        \label{F-a-1}
        \mathcal{F}^{-a-1} = \mathcal{F}^{-a} + \sum_{j=1}^{\ell+1}[\mathcal{F}^{-j},\mathcal{F}^{-a-1+j}] 
        =
        \mathcal{J}^{(m+1-a)} + \sum_{j=1}^{\ell+1}[\mathcal{J}^{(m+1-j)},\mathcal{J}^{(m-a+j)}]
    \end{gather}
    Because $a\leq 2m$, we have $(m+1-j) + (m-a+j) = 2m+1-a\geq 1$, so that Corollary \ref{Corollary: Involutivity} can be applied to obtain for each $j$
    \begin{equation}
        \label{J summand bracket}
        [\mathcal J^{(m+1-j)},\mathcal J^{(m-a+j)}]\subseteq \mathcal J^{(\min\{m+1-j,\ m-a+j\}-1)}
    \end{equation}
    Considering only $1\leq j \leq \ell+1$ as in the summation and recalling that $\ell+1\leq a$, we have $m+1-j\geq m-\ell\geq m-a+1.$ Similarly, since $j\geq 1$, $m-a+j\geq m-a+1$. Consequently, \eqref{J summand bracket} is contained in $\mathcal J^{(m-a)}$, and \eqref{F-a-1} gives
    \begin{equation}
        \mathcal F^{-a-1}\subseteq \mathcal J^{(m+1-a)}+ \mathcal J^{(m-a)}\subseteq \mathcal J^{(m-a)}
    \end{equation}
    as desired. The opposite inclusion follows from the definition \eqref{J(-i) def} and the fact that $\mathcal{C}\subseteq \mathcal{J}^{(m+1-j)}$ for each $1\leq j \leq a$.
\end{proof}

In order to prove (ii), we need the following lemma.

\begin{lemma}
    \label{Lemma: [e_i,e_-i]}
    Let $\big(\{e_i\}_{i=-k}^{k+\ell},\{f_j\}_{j=-k-\ell}^k,\eta\big)$ be locally defined vector fields as in Proposition \ref{Prop: JacobiFrame}. For each $0\leq i \leq k$,
    \begin{gather}
        [e_{i},e_{-i}]\in \mathcal{J}^{(-m)}\quad\text{and}\quad [f_{i},f_{-i}]\in \mathcal{J}^{(-m)}
    \end{gather}
\end{lemma}
\begin{proof}
    By induction on $i$; the base case $i=0$ is trivial. For the induction case, assume the result holds for some $0\leq i <k$. Using the second property of Proposition \ref{Prop: JacobiFrame}, we have
    \begin{align}
        [e_{i+1},e_{-i-1}] &\equiv \big[e_{i+1},[x,e_{-i}]\text{ mod } \mathcal C\big] 
        \\
        &\equiv \big[[e_{i+1},x],e_{-i}\big] + \big[x,[e_{i+1},e_{-i}]\big] \mod \mathcal{J}^{(i)}\\
        &\equiv -[e_{i}, e_{-i}] \mod \mathcal{J}^{(-m)}
    \end{align}
    where the last equivalence holds because by the second involutivity condition of Lemma \ref{Lemma: invol cond}, $[e_{i+1},e_{-i}]\in \mathcal{J}^{(-i)}$. Now applying the induction hypothesis gives the result for $i+1$. Similar argument can be made for the $f_i$.
\end{proof}

\begin{proof}[Proof of (ii)]
    By the definition \eqref{F def 2},
    \begin{align}
        \label{Last Fi}
        \mathcal{F}^{-2m-2} & = \mathcal{F}^{-2m-1} + \sum_{j=1}^{m+1}[\mathcal{F}^{-j},\mathcal{F}^{-2m-2+j}]
    \end{align}
    Now for each index $j$ with $1\leq j \leq \ell$, use property 1 of Proposition \ref{Prop: JacobiFrame} and item (i) of Proposition \ref{Prop: Technical} to observe that 
    \begin{align}
        [\mathcal{F}^{-j},\mathcal{F}^{-2m-2+j}] & =[\mathcal{F}^{1-j}\oplus \langle e_{m+1-j}\rangle,\mathcal{F}^{-2m-1+j}\oplus \langle f_{-m-1+j}\rangle]
        \\
        &\subseteq \mathcal{F}^{-2m-1} + \Big\langle [e_{m+1-j},f_{-m-1+j}]\Big\rangle = \mathcal{F}^{-2m-1}\oplus \langle \eta\rangle
    \end{align}
    where the containment follows from \eqref{bracket-resp cond}, and the last equality follows from property 3 of Proposition \ref{Prop: JacobiFrame}. 
    
    Similarly, for each index $j$ with $\ell+1\leq j \leq m+1$, we have
    \begin{align}
        &[\mathcal{F}^{-j},\mathcal{F}^{-2m-2+j}]
        \\
        &\ = \big[\mathcal{F}^{1-j}\oplus \langle e_{m+1-j},f_{m+1-j}\rangle,\mathcal{F}^{-2m-1+j}\oplus \langle e_{-m-1+j}, f_{-m-1+j}\rangle\big]
        \\
        &\ \subseteq \mathcal{F}^{-2m-1} + \big[\langle e_{m+1-j},f_{m+1-j}\rangle,\langle e_{-m-1+j},f_{-m-1+j}\rangle\big]
    \end{align}
    Applying Lemma \ref{Lemma: [e_i,e_-i]} and property 3 of Proposition \ref{Prop: JacobiFrame} then gives that 
    \begin{equation}
        [\mathcal{F}^{-j},\mathcal{F}^{-2m-2+j}]\subseteq \mathcal{F}^{-2m-1}\oplus \langle \eta\rangle.
    \end{equation}
    Returning to \eqref{Last Fi} then gives $\mathcal{F}^{-2m-2}(\lambda) \subseteq  \mathcal{F}^{-2m-1}(\lambda)\oplus \big\langle \eta(\lambda)\big\rangle$. The opposite inclusion is immediate from item (3) of Proposition \ref{Prop: JacobiFrame}.
\end{proof}

\begin{proof}[Proof of (iii)]
    For $1\leq i \leq 2m+1$, the claim is easily established by applying item (i) of Proposition \ref{Prop: Technical} along with the involutivity conditions of Lemma \ref{Lemma: invol cond}. Now fix some $a>2m+1$ and assume the claim for each $i<a$. We have by the Jacobi identity that
    \begin{gather}
        [V_{m+1},\mathcal{F}^{-a}] =  \Big[V_{m+1},\mathcal{F}^{1-a}+\sum_{j=1}^{a-1}[\mathcal{F}^{-j},\mathcal{F}^{j-a}]\Big]
        \\
         \subseteq [V_{m+1},\mathcal{F}^{1-a}] + \sum_{j=1}^{a-1}\big[V_{m+1},[\mathcal{F}^{-j},\mathcal{F}^{j-a}]\big]
        \\
         \subseteq [V_{m+1},\mathcal{F}^{1-a}]+ \sum_{j=1}^{a-1}\big[[V_{m+1},\mathcal{F}^{-j}],\mathcal{F}^{j-a}\big] + \big[\mathcal{F}^{-j},[V_{m+1},\mathcal{F}^{j-a}]\big]
        \\
        \subseteq \mathcal{F}^{-a} + \sum_{j=1}^{a-1}[\mathcal{F}^{-j},\mathcal{F}^{j-a}] \subseteq \mathcal{F}^{-a}
    \end{gather}
    where the last line follows from the induction hypothesis and the definition of $\mathcal{F}^{-a}$.
\end{proof}

\begin{remark}
    \label{Remark: D and F}
    Two observations should be made about the relation of the filtration $\{\mathcal{F}^{-i}\}_{i>0}$ to the original distribution $D$.
    \begin{enumerate}
        \item By \eqref{J0 fact}, $\mathcal{F}^{-k-\ell-1}=\pi^*(D)$. Thus if the distribution $D$ is bracket-generating, then at each $\lambda\in \mathcal{R}_D$, $\mathcal{F}^{-i}(\lambda)=T\mathcal R_D$ for sufficiently large $i$.
        \item For any $\lambda\in \mathcal R_D$, the algebra of infinitesimal symmetries of $D$ near $\pi(\lambda)$ is naturally included in the algebra of infinitesimal symmetries of the filtration $\{\mathcal{F}^{-i}\}$ near $\lambda$, since this filtration was obtained canonically from $D$.
    \end{enumerate}
\end{remark}

\subsubsection{Defining the Jacobi-Tanaka Algebra}
Before considering the basic local invariants of the filtration obtained via the symplectification procedure, we quotient by the Cauchy characteristic $V_{m+1}$, which has rank $n-4-2k-\ell$. Item (iii) of Proposition \ref{Prop: Technical} implies that the (locally defined) quotient $q:\mathcal R_D'\to \check{\mathcal R}_D'$ by the foliation generated by $V_{m+1}$ carries a filtration $\{\check {\mathcal F}^{-i}\}$ induced by $\{\mathcal F^{-i}\}$. Further, the vertical distribution ${\mathcal{V}}\subseteq T\mathcal R_D'$ passes to a distribution $\check{\mathcal{V}}$ on $\check{\mathcal R}_D'$ which can be used to recover the germ of the original distribution, $D$. We will use $\check \pi:\check{\mathcal R}_D'\to M$ for the projection induced by $\pi:\mathcal R_D'\to M$. Note that $\check{\mathcal{V}}$ is the vertical distribution for $\check \pi$. Near each $\lambda\in\mathcal R_D'$, we have maps of filtered manifolds
\begin{center}
\begin{altfig}{A commutative triangle. Top-left, the pair script R prime subscript D with the collection F superscript minus i; top-right, the checked version, script R check prime subscript D with F check superscript minus i; bottom, the pair M with D superscript minus i. A horizontal map q goes top-left to top-right, a diagonal map pi goes top-left to bottom, and a vertical map pi check goes top-right to bottom, with pi check after q equal to pi.}

    \begin{tikzcd}[column sep=large, row sep=large]
    \label{Diagram: reduced sympl}
        \big(\mathcal R_D', \{\mathcal F^{-i}\}\big)
        \arrow[r, "q"]
        \arrow[dr, "\pi"']
        &
        \big(\check{\mathcal R}_D', \{\check{\mathcal F}^{-i}\}\big)
        \arrow[d, "\check{\pi}"]
        \\
        & \big(M, \{D^{-i}\}\big)
    % \alttext{A commutative triangle. Top-left, the pair script R prime subscript D with the collection F superscript minus i; top-right, the checked version, script R check prime subscript D with F check superscript minus i; bottom, the pair M with D superscript minus i. A horizontal map q goes top-left to top-right, a diagonal map pi goes top-left to bottom, and a vertical map pi check goes top-right to bottom, with pi check after q equal to pi.}
    \end{tikzcd}
\end{altfig}
\end{center}
At each $\check{\lambda}\in \check{\mathcal{R}}_D'$, define the \emph{Tanaka symbol of $\{\check{\mathcal{F}}^{-i}\}_{i>0}$ at $\check{\lambda}$} to be
\begin{equation}
    \label{symbol def}
    \mathfrak{m}(\check{\lambda}) = \bigoplus_{i=1}^{\infty}\check{\mathcal{F}}^{-i}(\check{\lambda})/\check{\mathcal{F}}^{1-i}(\check{\lambda}) = \bigoplus_{i>0}\mathfrak{m}_{-i},
\end{equation}
which is a negatively graded Lie algebra with bracket induced by the bracket of vector fields and serves as a model to the filtered tangent space $T_{\check{\lambda}}\mathcal R_D'$. We also call $\mathfrak{m}(\check{\lambda})$ the \emph{Jacobi-Tanaka algebra of $D$ at $\check{\lambda}$}. If $\mathfrak{m}(\check{\lambda})$ does not depend on $\check\lambda$, then we say $D$ is \textit{of constant Jacobi-Tanaka symbol}.

\begin{remark}
    \label{Remark: Jacobi-Tanaka identification}
    For any $\lambda\in \check{\mathcal{R}}_D'$, the Jacobi-Tanaka algebra at $\check{\lambda}=q(\lambda)$ is naturally identified with the quotient Tanaka symbol of the filtration $\{\mathcal F^{-i}\}_{i>0}$ by $V_{m+1}(\lambda)$; explicitly, this Tanaka symbol is given by
    \begin{equation}
        \bigoplus_{i\geq 1} \mathcal F^{-i}(\lambda)/\mathcal F^{1-i}(\lambda).
    \end{equation}
\end{remark}

\begin{definition}
    \label{def: Tanaka prol}
    Given a negatively graded Lie algebra $\mathfrak{a}$, the \textit{universal Tanaka prolongation} of $\mathfrak{a}$ is the maximal graded Lie algebra 
    \[
        \mathfrak{g}(\mathfrak{a}) = \displaystyle{\bigoplus_{i\in\mathbb{Z}}\mathfrak{g}_i}
    \]
    satisfying the properties
    \begin{enumerate}[label=(\roman*)]
        \item For each $i<0$, $\mathfrak{g}_i = \mathfrak{a}_i$
        \item For any $i\geq0$ and any $x\in \mathfrak{g}_i$, if $[x,\mathfrak{g}_j] = 0$ for each $j<0$, then $x=0$.
    \end{enumerate} 
\end{definition}

The universal Tanaka prolongation of $\mathfrak{a}$ is unique up to natural graded Lie algebra isomorphism and can be constructed recursively by defining $\mathfrak{g}_{i} = \mathfrak{a}_{i}$ for $i<0$ and

\begin{equation}
\mathfrak{g}_i=\left\{\varphi\in\bigoplus_{j>0}\operatorname{Hom}(\mathfrak a_{-j},\mathfrak g_{i-j})\ \middle|\
\begin{aligned}
&[\varphi(X),Y]+[X,\varphi(Y)]=\varphi\big([X,Y]\big)\\
&\text{for all } X,Y\in\mathfrak a
\end{aligned}\right\}
\end{equation}
for each $i\geq 0$. Note that the universal Tanaka prolongation of a graded Lie algebra $\mathfrak{a}$ might be infinite-dimensional.
\begin{remark}
    \label{Remark: Tanaka symm}
    It is well-known that if $\mathfrak{a}(\lambda)$ is the Tanaka symbol of a (non-positively) filtered manifold $\big(N,\{\mathcal G^{-i}\}\big)$ at a point $\lambda$, then the algebra of infinitesimal symmetries for $\big(N,\{\mathcal G^{-i}\}\big)$ at $\lambda$ has dimension at most $\mathrm{dim}\big(\mathfrak g(\mathfrak a(\lambda)\big)$. See, for example, \cite{cap2017} \cite{Tanaka1970} \cite{Zelenko2009-qt}. Although these resources treat only the case of fundamental Tanaka symbols (that is, the symbols $\mathfrak m$ generated by $\mathfrak m_{-1}$), the proofs can be reproduced with obvious modification for non-fundamental symbols.
\end{remark}

\begin{remark}
    \label{Remark: generators of m}
    Suppose $\mathfrak{m}$ is the Jacobi-Tanaka algebra for a distribution $D$ with rectangular Jacobi symbol $(k,0)$. Then by \eqref{F def 2}, $\mathfrak{m}$ is \textit{fundamental}; that is, $\mathfrak{m}$ is generated as a Lie algebra by $\mathfrak m_{-1}$.

    Remark \ref{Remark: D and F} implies that if $\mathfrak m$ is a Jacobi-Tanaka algebra at a point $\lambda$ with Jacobi-Tanaka symbol $(k,\ell)$, then $\mathfrak m_{-1}\oplus\mathfrak m_{-2}\oplus\cdots \oplus \mathfrak m_{-\ell-1}$ generates $\mathfrak m$, so each $i>\ell+1$ has
    \begin{equation}
        \label{eqn: generators of mi}
        \mathfrak m_{-i} = \bigoplus_{j=1}^{\ell+1} [\mathfrak m_{-j}, \mathfrak m_{-i+j}]
    \end{equation}
\end{remark}

We will soon relate the infinitesimal symmetries of $\big(\check{\mathcal R}_D',\{\check{\mathcal F}^{-i}\}\big)$ to those of the original distribution $(M,D)$. Before we do so, we must consider how the Jacobi symbol constrains the Jacobi-Tanaka algebra.

For each $i>0$, define
\begin{equation}
    \label{def: m^-i, m'}
    \mathfrak{m}^{-i} = \bigoplus_{-i\leq j \leq -1}\mathfrak{m}_j\quad \text{and}\quad \mathfrak{m}' = \bigoplus_{j\leq -2k-2\ell - 3} \mathfrak{m}_j
\end{equation}

Recall that by Remark \ref{Remark: Jacobi-Tanaka identification}, at any $\lambda\in \mathcal R_D'$, the Jacobi-Tanaka symbol at $q(\lambda)$ is isomorphic to the Tanaka symbol of $\big(\mathcal R_D',\{\mathcal F^{-i}\}\big)$ at $\lambda$ after quotienting by $V_{m+1}$. We utilize this identification in the following proposition, which shows that the graded Lie algebra $\mathfrak{m}/\mathfrak{m}'$ is equivalent to the Jacobi symbol $(k,\ell)$.
\begin{prop}
    \label{prop: m rels}
    Let $D$ be a $(3,n)$ distribution on $M$ with constant Jacobi symbol $(k,\ell)$, and let $\big(x,\{e_i\}_{i=-k}^{m},\{f_j\}_{j=-m}^k,\eta\big)$ be vector fields on $\mathcal R_D'$ as in Proposition \ref{Prop: JacobiFrame}. Fix $\lambda\in \mathcal R_D'$ where the vector fields are defined, and let $\check\lambda = q(\lambda)$. Define elements of $\mathfrak m(\check\lambda)$
    \begin{align}
        & \oline{x} = x(\lambda) \mod V_{m+1}(\lambda),\quad & \oline e_{i} = e_{i}(\lambda)\mod \mathcal{F}^{i-m}(\lambda),\\
        &  \oline f_{i} = f_{i}(\lambda)\mod \mathcal{F}^{i-m}(\lambda),\quad &\oline \eta = \eta(\lambda) \mod \mathcal{F}^{-2m-1}(\lambda)
    \end{align}
    for appropriate $i$ and $j$. We have bracket relations in $\mathfrak{m}(\check\lambda)$
    \begin{align}
        [\oline x,\oline e_i] &= \oline e_{i-1}\quad\text{for each }1-k\leq i \leq m
        \\
        [\oline x,\oline f_i] &= \oline f_{i-1}\quad\text{for each }1-m \leq i \leq k
        \\
        [\oline e_i, \oline f_{-i}] &= (-1)^i \eta\quad \text{for each } -k\leq i \leq m 
        %\label{Heisenberg}
    \end{align}
    and all other brackets from $\big(\oline x,\{\oline e_i\}_{i=-k}^{m},\{\oline f_j\}_{j=-m}^k,\oline \eta\big)$ are zero modulo ${\mathfrak{m}}'\big(\check\lambda\big)$.
\end{prop}

\begin{proof}
    Each of the nontrivial relations hold immediately from the properties of the frame from Proposition \ref{Prop: JacobiFrame} and from items (i) and (ii) of Proposition \ref{Prop: Technical}. All relations involving $V_{m+1}$ are demonstrated by item (iii) of Proposition \ref{Prop: Technical}. It remains to show the trivial relations between the $\oline e_i$ and the $\oline f_i$.

    {\bf Brackets among the $\oline e_i$ and the $\oline f_j$:}
    First consider the case where $i+j>0$; assume without loss of generality that $i\geq j$ so that $i>0$. Items (1) and (2) of Lemma \ref{Lemma: invol cond} imply that each of $[e_i,e_j], [e_i,f_j], [f_i,e_j],$ and $[f_i,f_j]$ is contained in $\mathcal J^{(j)}$; item (1) is applicable for $j\leq 0$, while item (2) is applicable for $j>0$. By item (i) of Proposition \ref{Prop: Technical}, $ \mathcal J^{(j)} = \mathcal F^{j-m-1}\subseteq \mathcal F^{i+j-2m-1}$, so that in the grading, each of the brackets vanishes.
    
    On the other hand, if $i+j<0$, then each of $[\oline e_i,\oline e_j]$, $[\oline e_i,\oline f_j]$, and $[\oline f_i,\oline f_j]$ has graded weight at most $-2m-3$, and therefore belongs to $\mathfrak{m}'(\check\lambda)$. Finally, if $i+j=0$, then one can induct on $i$ as follows; let $[\oline e_0, \oline e_0]=0$ be the base case. The Jacobi identity then gives
    \begin{equation}
        [\oline e_{-1},\oline e_1] = \big[[\oline x, \oline e_0],\oline e_1\big] = \big[[\oline x, \oline e_1],\oline e_0\big] +  \big[\oline x,[\oline e_0,\oline e_1]\big] = [\oline e_0,\oline e_0] + [\oline x, 0] = 0
    \end{equation}
    where $[\oline e_1,\oline e_0]=0$ has already been shown. Continuing in this way, one can show that $[\oline e_i, \oline e_j] = [\oline f_i, \oline f_j] = 0$ whenever $i+j=0$.

    {\bf Brackets with $\oline x$:} By the definition of $m$, $[\mathcal C,\mathcal J^{(-m)}] \subseteq \mathcal J^{(-m)}$, so $[\oline x, \oline f_{-m}]=[\oline x, \oline e_{-m}]=0$. In the case that $\ell>0$, we also must show $[\oline x, \oline e_{-k}] = 0$. By Proposition \ref{Prop: JacobiFrame} (1) and Proposition \ref{Prop: Technical} (i), we have that $\mathfrak m_{-m-k-2} = \langle f_{-k-1}\rangle,$ so $[\oline x, \oline e_{-k}] = cf_{-k-1}$ for some $c\in \mathbb R$. Applying the Jacobi identity yields
    \begin{gather}
        c(-1)^{k+1} \oline \eta = [\oline e_{k+1},c\oline f_{-k-1}] = \big[\oline e_{k+1},[\oline x,\oline e_{-k}]\big] 
        \\
        =-[\oline e_{k},\oline e_{-k}] + \big[\oline x,[\oline e_{k+1},\oline e_{-k}]\big] = 0.
    \end{gather}
    Thus $c= 0$ and $[\oline x, \oline e_{-k}]=0$.

    The remaining brackets ($[\oline x, \oline \eta], [\oline e_i,\oline \eta]$, and $[\oline f_i,\oline \eta]$) all have weight at most $-2m-3$ and are therefore contained in $\mathfrak m'(\check\lambda)$.
\end{proof}

Although the above proposition holds only with modulus $\mathfrak{m}(\check\lambda)$, the following corollary holds without this modulus. Note that throughout its proof, the proposition is applied only in graded weight at least $-2m-2$, where the modulus is irrelevant.

\begin{corollary}
    \label{cor: more m rels}
    Let $\lambda$, $\check\lambda$, and $\big(\oline x,\{\oline e_i\}_{i=-k}^{m},\{\oline f_j\}_{j=-m}^k,\oline \eta\big)$ be as above. Then we have bracket relations in ${\mathfrak{m}}(\check\lambda)$ 
    \begin{equation}
        [\oline x,\oline \eta] = 0\quad \text{and}\quad [\oline e_i,\oline \eta] = [\oline f_i, \oline \eta] = 0\text{ for }i>0
    \end{equation}
\end{corollary}
\begin{proof}
    Applying the Jacobi identity and the previous proposition gives that
    \begin{gather}
        \label{x, eta 1}
        [\oline x,\oline \eta] = (-1)^{-k}\big[\oline x,[\oline e_{-k},\oline f_{k}]\big] 
        \\
        = (-1)^{-k}\Big(\big[[\oline x, \oline e_{-k}], \oline f_k\big] + \big[\oline e_{-k},[\oline x, \oline f_{k}]\big]\Big) = (-1)^{-k}\big[\oline e_{-k},\oline f_{k-1}\big],
    \end{gather}
    since $[\oline x, \oline e_{-k}]=0$ by Proposition \ref{prop: m rels}. Expanding $\oline e_{-k} = [\oline x , \oline e_{1-k}]$, we can continue by applying to Jacobi identity:
    \begin{gather}
        (-1)^{-k}[\oline e_{-k},\oline f_{k-1}] =  (-1)^{-k}\big[[\oline x, \oline e_{1-k}],\oline f_{k-1}\big] 
        \\
        = (-1)^{-k}\Big(\big[[\oline x, \oline f_{k-1}],\oline e_{1-k}\big]+\big[\oline x,[\oline e_{1-k},\oline f_{k-1}]\big]\Big)
        \\
        = (-1)^{-k}\Big( [\oline f_{k-2}, \oline e_{1-k}]+ (-1)^{1-k}[\oline x,\oline \eta]\Big) = -[\oline x, \oline \eta] + (-1)^{1-k}[\oline e_{1-k},\oline f_{k-2}]
    \end{gather}
    One can continue by expanding $e_{1-k} = [\oline x ,\oline e_{2-k}]$ and again applying the Jacobi identity. Iterating this procedure for $m+k-2$ more steps, we eventually obtain that
    \begin{gather}
        [\oline x,\oline \eta] = (-m-k+1)[\oline x,\oline \eta] + (-1)^{m-1}[\oline{e}_{m-1},\oline f_{-m}].
    \end{gather}
    Finally, we can write $\oline e_{m-1} = [\oline x, \oline e_{m}]$ and apply the Jacobi identity once more to find that
    \begin{gather}
        [\oline x,\oline \eta] = (-m-k+1)[\oline x,\oline \eta] + (-1)^{m-1}\big[[\oline x, \oline{e}_{m}],\oline f_{-m}\big] \\
        = (-m-k)[\oline x , \oline \eta]+ (-1)^{m}\big[\oline e_m,[\oline x, \oline f_{-m}]\big]
    \end{gather}
    By Proposition \ref{prop: m rels}, $[\oline x, \oline f_{-m}] = 0$; grouping all $[\oline x,\oline \eta]$ terms together then yields that $(m+k+1)[\oline x,\oline \eta] = 0$. Since $m+k+1> 0$, this implies $[\oline x,\oline \eta] = 0$, and the first relation is proven. 

    For the remaining relations, recall from the previous proposition that for each $i>0$, $[\oline e_i,\oline e_0] = [\oline e_i, \oline f_0] = 0$. Applying the Jacobi identity, we therefore obtain
    \begin{equation}
        [\oline e_i, \oline \eta] = \big[\oline e_i, [\oline e_0,\oline f_0]\big] = \big[[\oline e_i,\oline e_0],\oline f_0\big] + \big[\oline e_0, [\oline e_i, \oline f_0]\big] = 0.
    \end{equation}
    Similar argument gives that $[\oline f_i, \oline \eta]=0$.
\end{proof}

We are now able to relate the infinitesimal symmetries of $\big(\check{\mathcal R}_D',\{\check{\mathcal F}^{-i}\}\big)$ to those of the original distribution $(M,D)$.

\begin{prop}
    \label{Prop: Jacobi alg bij}
    Let $D$ be a rank $3$ distribution with $6$-dimensional square on $M$. Fix $\check\lambda\in {\check{\mathcal R}}_D'$, and let ${\mathfrak{m}}(\check{\lambda})$ be the Jacobi-Tanaka algebra of $D$ at $\check\lambda$. The infinitesimal symmetries of $\big({\check{\mathcal{R}}}_D',\{\check {\mathcal F}^{-i}\}\big)$ near $\check \lambda$ are in bijection with those of $(M,D)$ near $\check\pi(\check\lambda)$. In particular, the dimension of the algebra of infinitesimal symmetries of $D$ near $\check \pi(\check \lambda)$ is not greater than the dimension of $\mathfrak{g}\big({\mathfrak{m}}(\check\lambda)\big)$.
\end{prop}
\begin{proof}
    Note that by Lemma \ref{Lemma: invol cond} and Proposition \ref{prop: m rels}, $\check {\mathcal{V}}$ is the full Cauchy characteristic of $\check{\mathcal{F}}^{-m-1}=\check\pi^*(D)$, so one can recover $(M,D)$ from \linebreak[4] $\big(\check{\mathcal{R}}_D',\{\check{\mathcal{F}}^i\}\big)$.
    Since $\big(\check{\mathcal{R}}_D',\{\check{\mathcal{F}}^i\}\big)$ was constructed without reference to a coordinate chart, each infinitesimal symmetry of $(M,D)$ near $x\in M$ lifts to an infinitesimal symmetry of $\big({\check{\mathcal{R}}}_D',\{\check{\mathcal{F}}^i\}\big)$ near any $\check\lambda\in \check\pi^{-1}(x)$, and by the recovery process, the lift of a nontrivial infinitesimal symmetry of $(M,D)$ is nontrivial.
    
    Conversely, any infinitesimal symmetry of $\big(\check{\mathcal{R}}_D',\{\check{\mathcal{F}}^i\}\big)$ near $\check\lambda\in\check{\mathcal{R}}_D'$ must preserve both $\check \pi^*(D) = \check{\mathcal{F}}^{-m-1}$ and $\check {\mathcal{V}}$. Therefore, such an infinitesimal symmetry passes to an infinitesimal symmetry of $D$ near $\check\pi(\check\lambda)$.
    
    Let us argue that the resulting infinitesimal symmetry of $(M,D)$ is nontrivial; we do so by showing any vertical infinitesimal symmetry $\check v$ of \allowbreak$\big({\check{\mathcal{R}}}_D',\{\check{\mathcal{F}}^i\}\big)$ are trivial. Using the frame constructed in Proposition \ref{Prop: JacobiFrame}, let $v = v_f^i \cdot f_i + v_e^i \cdot e_i$ be a lift of $\check v$ to a local section of $\mathcal V$, where $v_e^i$ and $v_f^i$ are smooth locally defined functions on $\mathcal R_D'$. Suppose for contradiction that there is a minimal integer $i_0>0$ such that $v_f^{i_0}(\lambda)$ or $v_e^{i_0}(\lambda)$ is nonzero for some $\lambda \in q^{-1}(\check\lambda)$. Assume without loss of generality that $v_e^{i_0}(\lambda)\neq 0$. By the third property of Proposition \ref{Prop: JacobiFrame},
    \begin{equation}
        [v,f_{-i_0}]\equiv (-1)^{i_0}\eta\mod \mathcal{J}^{-m}
    \end{equation}
    so that $v$ does not preserve $\mathcal{J}^{-m} = \mathcal F^{-2m-1}$. Therefore $v=0$, as desired.

    We therefore have a bijection between the infinitesimal symmetries of $\{\check{\mathcal F}^i\}$ near $\check\lambda$ and those of $D$ near $\check \pi(\lambda)$. The final assertion of the Proposition then follows from Remark \ref{Remark: Tanaka symm}.
\end{proof}

\subsubsection{Tanaka symbols associated with Jacobi-Tanaka algebras}
\label{subsubsection: symplectically flat}
In Tanaka theory, each Tanaka symbol $\mathfrak{m}$ (i.e., a negatively graded Lie algebra) has a corresponding \textit{Tanaka flat} filtered structure, constructed as follows: Let $\mathcal G(\mathfrak m)$ be the simply connected Lie group with Lie algebra $\mathfrak m$. The Tanaka flat filtration on $\mathcal G(\mathfrak m)$ is the left-invariant filtration given by
\begin{equation}
    \mathfrak m^{-i} = \bigoplus_{j\geq -i}\mathfrak m_j
\end{equation}
All other filtered manifolds with Tanaka symbol $\mathfrak m$ have algebras of infinitesimal symmetries which are bounded in dimension by that of the universal Tanaka prolongation $\mathfrak g(\mathfrak m)$, which is the algebra of infinitesimal symmetries for the flat model at any point\cite{Tanaka1970}\cite{Zelenko2009-qt}\cite{CapSlovak}.

% \ND{I think there are several questions that could be addressed here. Most basically, does $\mathcal D(\mathfrak m)$ always have Jacobi-Tanaka symbol $\mathfrak m$? Is the symplectically flat distribution always Tanaka flat? (No!) In constructing the symplectically flat $(3,6,8)$ distribution, we basically considered the Tanaka symbol of $\mathcal D(\mathfrak m)$ (which is a priori only filtered, not Tanaka flat).}

Now suppose that $\mathfrak m$ is a negatively graded Lie algebra which satisfies the conditions of Proposition \ref{prop: m rels} for some integer $m>0$. Let $\mathcal G(\mathfrak m)$ be the simply connected Lie group with Lie algebra $\mathfrak m$, and let $\mathfrak v$ be the left-invariant distribution $\langle \oline e_i,\oline f_i\rangle_{i>0}$ on $\mathcal G(\mathfrak m)$. Because $\mathfrak v$ is a Cauchy characteristic of the left-invariant distribution $\mathfrak m^{-m-1}$, the latter distribution descends to a distribution $\mathcal D(\mathfrak m)$ on the homogeneous space $\mathcal G(\mathfrak m)\to \mathcal M(\mathfrak{\mathfrak m}) = \mathcal G(\mathfrak m)/\mathcal G(\mathfrak v)$. The distribution $\mathcal D(\mathfrak m)$ will be called the \emph{symplectically flat distribution with Jacobi-Tanaka symbol $\mathfrak m$}.

The Tanaka symbol $\mathfrak{n}(\mathfrak m)$ of the distribution $\mathcal{D}(\mathfrak m)$ will be called the \emph{Tanaka symbol induced by the Jacobi-Tanaka algebra $\mathfrak m$}. A primary advantage of the symplectification procedure is that the set of Tanaka symbols induced by a given Jacobi-Tanaka algebra is significantly smaller than the set of all possible Tanaka symbols (see Appendix \ref{Appendix} for a detailed demonstration). Crucially, because the maximally symmetric models are guaranteed to appear exclusively within this restricted set, we can safely exclude all others from our analysis.

\begin{remark}
If $\mathfrak{m}$ is the Jacobi-Tanaka algebra corresponding to a rank $3$ distribution of maximal class (equivalently, for which $\mathcal{D}(\mathfrak{m})$ is of maximal class), then it can be shown \cite{doubrov2008rank3, doubrov2016JacobiCurves} that $\mathcal{D}(\mathfrak{m})$ is equivalent to the flat distribution with Tanaka symbol $\mathfrak{n}(\mathfrak{m})$. However, in general, without the assumption of maximal class, the symplectically flat distribution $\mathcal{D}(\mathfrak{m})$ may not be isomorphic to the flat distribution with Tanaka symbol $\mathfrak{n}(\mathfrak{m})$ so that the symplectification of the flat distribution with Tanaka symbol $\mathfrak n(\mathfrak m)$ may have Jacobi-Tanaka symbol different from $\mathfrak m$, and subsequently the assignment $\mathfrak{m} \mapsto \mathfrak{n}(\mathfrak{m})$ may not be injective. 
\end{remark}

By Proposition \ref{Prop: Jacobi alg bij}, the infinitesimal symmetries of a distribution $D$ with Jacobi-Tanaka symbol $\mathfrak m$ are in bijection with the infinitesimal symmetries of $\big(\check{\mathcal R}_D',\{\check {\mathcal {F}}^{-i}\}\big)$, and are therefore bounded in dimension by the universal Tanaka prolongation $\mathfrak g(\mathfrak m)$, which is the infinitesimal symmetry algebra of $\mathcal D(\mathfrak m)$.

\section{Proof of the conjecture of maximality of class in the case of (3,6) and (3,7)}
\label{section: dim 6-7}

In this section, we will prove Theorem \ref{Thm: (3,6) and (3,7)} using the machinery of the previous section. Let $D$ be a bracket-generating distribution on $M^n$ satisfying the hypotheses of the theorem. From the observations of Section \ref{subsection: Jacobi Symbol}, the distribution $D$ is of maximal class at a point $p\in M$ if and only if equality holds in \eqref{num_of_boxes} for the $k\geq 1$ and $\ell\geq 0$ which define the Jacobi symbol at a generic point $\lambda\in \pi^{-1}(p)$; that is, if $2k+\ell=n-4$. 

First, let us treat the case where $n=6$. The condition \eqref{num_of_boxes} and the conditions that $k\geq 1$ and $\ell\geq 0$ imply that $k=1$ and $\ell=0$, which gives equality in \eqref{num_of_boxes} and implies that $p$ is a point of maximal class for $D$, as desired.

When $n=7$, the only possible choices for $(k,\ell)$ are $(1,1)$ and $(1,0)$. In the first case, we have equality in \eqref{num_of_boxes}, so $D$ is of maximal class at $p$. We now show that the Jacobi symbol $(1,0)$ does not in fact arise from any bracket-generating distribution $D$. Suppose that $n=7$ and the Jacobi symbol of $D$ at $\check\lambda\in \check\pi^{-1}(p)$ is given by $(k,\ell)=(1,0)$. Let $\big(\oline x,\{\oline e_i\}_{i=-k}^{k+\ell},\{\oline f_j\}_{j=-k-\ell}^k,\oline \eta\big)$ be as in Proposition \ref{prop: m rels}; the first four graded pieces $\mathfrak{m}_i$ of the Jacobi-Tanaka algebra $\mathfrak{m} = \mathfrak{m}(\check\lambda)$ are
\begin{gather}
    {\mathfrak{m}}_{-1} = \langle \oline x, \oline e_1, \oline f_1\rangle,\quad {\mathfrak{m}}_{-2} =\langle \oline e_0, \oline f_0\rangle,\quad 
    {\mathfrak{m}}_{-3} = \langle \oline e_{-1}, \oline f_{-1}\rangle,\quad {\mathfrak{m}}_{-4}=\langle \eta\rangle
\end{gather}
Because $\ell = 0$, the symbol $\mathfrak m$ is fundamental. Utilizing this fact along with Corollary \ref{cor: more m rels} yields
\begin{equation}
    {\mathfrak{m}}_{-5} = [{\mathfrak{m}}_{-1},{\mathfrak{m}}_{-4}] = 0.
\end{equation}
Therefore $\mathrm{rank}(\check{\mathcal{F}}^{-4}) = \mathrm{dim}(\mathfrak{m}) = 8$ . However, $\dim(\check{\mathcal{R}}_D') = \dim(\mathcal{R}_D')-\dim(V_{2}) = (2n-4)-1 = 9$. By Remark \ref{Remark: D and F}, this implies that $D$ is not bracket-generating at $p$, which contradicts our assumption. We have now proven Theorem \ref{Thm: (3,6) and (3,7)}.

\section{A (3,8) Distribution of Nonmaximal Class}
\label{section: dim 8}

Now let us analyze rank $3$ distributions in dimension $8$. The restriction \eqref{num_of_boxes} on the Jacobi symbol $(k,\ell)$ implies that at most four Jacobi symbols arise in this situation, namely the rectangular diagrams $(2,0)$ and $(1,0)$ and the skew diagrams $(1,2)$ and $(1,1)$. Among these Jacobi symbols, $(2,0)$ and $(1,2)$ yield equality in \eqref{num_of_boxes} and consequently arise from distributions of maximal class. The Jacobi symbol $(1,0)$ does not arise from a $(3,8)$ distribution, by same argument that concluded Section \ref{section: dim 6-7}. Consequently, only distributions with Jacobi symbol $(1,1)$ are of nonmaximal class, and this symbol corresponds to the skew Young diagram

% \begin{figure}[H]
    \begin{center}
        \begin{tikzpicture}[baseline=(current bounding box.center)]
            % Place the ytableau in a node
            \node (table) {
                \begin{ytableau}
                      e_{2} & e_{1} &  e_{0} &  e_{-1} & \none \\
                     \none & f_{1} &  f_{0} & f_{-1} & f_{-2} \\
                \end{ytableau}
            };
        \alttext{A skew Young diagram, comprised of two rows of four boxes. The boxes of the
first row are labeled e subscript 2, e subscript 1, e subscript 0, e subscript
minus 1, and the boxes of the second row are labeled f subscript 1, f subscript
0, f subscript minus 1, f subscript minus 2. Columns share the same index; that
is, the box labeled e subscript 1 is above the box labeled f subscript 1, and so
on.}
        \end{tikzpicture}
    \end{center}
    % \caption{The decorated Jacobi symbol corresponding to $(k,\ell) = (1,1)$}
% \end{figure}

In fact, if a $(3,8)$ distribution has the above Jacobi symbol, its Jacobi-Tanaka algebra can be determined:

\begin{prop}
    \label{Prop: (3,8) Nonmax Jacobi Alg}
    Let $D$ be a $(3,8)$ distribution with $6$-dimensional square which is not of maximal class. At any point ${\check\lambda} \in \check{\mathcal{R}}_D'$, the Jacobi-Tanaka algebra ${\mathfrak s}={\mathfrak s}({\check\lambda})$ of $D$ has a basis
    \begin{gather}
        \mathfrak s_{-1} = \langle \oline x,\oline e_2\rangle, \quad\mathfrak s_{-2} = \langle \oline e_1,\oline f_1\rangle,\quad \mathfrak s_{-3} = \langle \oline e_{0},\oline f_{0}\rangle,\quad \mathfrak s_{-4} = \langle \oline e_{-1},\oline f_{-1}\rangle
        \\
        \mathfrak s_{-5} = \langle \oline f_{-2}\rangle,\quad \mathfrak s_{-6} = \langle \oline \eta\rangle, \quad \mathfrak s_{-7} = \langle \oline \nu\rangle
    \end{gather}
    with nontrivial bracket relations
    \begin{align}
        & \label{JacobiRel1}
        [\oline x, \oline e_j] = \oline e_{j-1} \quad \text{for }j=0,1,2
        & [\oline x, \oline f_j] = \oline f_{j-1} \quad \text{for }j=-1,0,1
        \\
        \label{JacobiRel2}
       & [\oline e_j,\oline f_{-j}] = (-1)^j\oline \eta \quad\text{for }j=-1,0,1,2
        &-[\oline f_{1}, \oline f_{-2}] = [\oline f_{0}, \oline f_{-1}] = \oline \nu
    \end{align}
\end{prop}
\begin{proof}
    Let $\big(\oline x,\{\oline e_i\}_{i=-k}^{k+\ell},\{\oline f_j\}_{j=-k-\ell}^k,\oline \eta\big)$ be as in Proposition \ref{prop: m rels}. All the bracket relations between ${\mathfrak s}_i$ and ${\mathfrak s}_j$ where $i+j\geq-6$ are given by Proposition \ref{prop: m rels}. Further, Corollary \ref{cor: more m rels} implies that $[{\mathfrak s}_{-1},{\mathfrak s}_{-6}]=0$.
    
    Now consider the relations of ${\mathfrak s}_{-2}$ with ${\mathfrak s}_{-5}$. The Jacobi identity yields 
    \begin{equation}
        [\oline e_1,\oline f_{-2}] = \big[[\oline x,\oline e_2],\oline f_{-2}\big] =  \big [\oline x,[\oline e_2,\oline f_{-2}]\big] + 0 = [\oline x, \oline \eta] = 0
    \end{equation}
    Thus Remark \ref{Remark: generators of m} implies that ${\mathfrak s}_{-7} = \big\langle [\oline f_1,\oline f_{-2}]\big\rangle$. If $[\oline f_1,\oline f_{-2}]=0$, then it is easy to show ${\mathfrak s}_j=0$ for all $j<-7$, and consequently that $\dim\big({\mathfrak s}({\check\lambda})\big) = 10<\dim(\check{\mathcal R}_D') = 11$. However, by Remark \ref{Remark: D and F}, this implies that $D$ is not bracket-generating at $p$, which contradicts our assumption. 
    
    We may therefore define $\oline \nu = -[\oline f_1,\oline f_{-2}]$, which generates ${\mathfrak s}_{-7}$. By the dimension count, ${\mathfrak s}_{j}=0$ for $j<-7$, and the final required relation follows from the Jacobi identity:
    \begin{gather}
        [\oline f_0,\oline f_{-1}] = \big[[\oline x,\oline f_1],\oline f_{-1}\big] = \big[[\oline x,\oline f_{-1}],\oline f_1\big] = [\oline f_{-2},\oline f_1] = \oline \nu
        \\
        [\oline e_0, \mathfrak s_{-4}] = \big[[\oline x , \oline e_{1}],\mathfrak s_{-4}\big] =  [\oline e_{1},\mathfrak s_{-5}] + \big[\oline x,[\oline e_1,\mathfrak s_{-4}]\big] = 0.
    \end{gather}
\end{proof}

\subsection{A Distribution of Nonmaximal Class}
\label{subsection: nonmax model}

% With the Jacobi-Tanaka algebra $\mathfrak s$ of Proposition \ref{Prop: (3,8) Nonmax Jacobi Alg} in hand, we could define the distribution $\mathfrak D$ to be the symplectically flat distribution $\mathcal D(\mathfrak s)$, as described in Section \ref{subsubsection: symplectically flat}. However, the resulting distribution is also Tanaka flat, and we therefore construct $\mathfrak D$ as the Tanaka flat distribution for its Tanaka symbol $\mathfrak n$, which is given explicitly below.

The Jacobi-Tanaka algebra $\mathfrak s$ of Proposition \ref{Prop: (3,8) Nonmax Jacobi Alg} motivates the construction of the distribution $\mathfrak D$ which follows. One can obtain $\mathfrak D$ on the homogeneous space $\mathcal G(\mathfrak s)/\mathcal{G}(\mathfrak v)$ where $\mathcal G(\mathfrak s)$ is the simply connected Lie group with Lie algebra $\mathfrak{s}$ and $\mathcal G(\mathfrak v)$ is the subgroup corresponding to the subalgebra $\mathfrak v = \langle \oline e_i,\oline f_i\rangle_{i>0}$. On this homogeneous space, $\mathfrak D$ is the distribution corresponding to the image of the left invariant filtrand $\mathfrak s^{-m-1}$ on $\mathcal G(\mathfrak s)$. However, the resulting distribution is also Tanaka flat; it is therefore more clear to construct $\mathfrak D$ as the Tanaka flat distribution for its Tanaka symbol $\mathfrak n$, which is given explicitly below.

Let $\mathfrak{n}$ be the Lie subalgebra of $\mathfrak s$ generated by $\oline x, \oline e_0$, and $ \oline f_0$; the basis and nontrivial relations of $\mathfrak n$ are given in \eqref{n basis} and \eqref{n rels}, respectively. Also equip $\mathfrak n$ with the grading described in \eqref{n basis}; that is, the grading where $\mathfrak n_{-1} = \langle \oline x, \oline e_0, \oline f_0\rangle$. The grading on $\mathfrak{n}$ is given explicitly by 
\begin{equation}
    \mathfrak{n} = \mathfrak{n}_{-1}\oplus\mathfrak{n}_{-2}\oplus\mathfrak{n}_{-3} = \langle \oline x,\oline e_0, \oline f_0\rangle \oplus \langle \oline e_{-1},\oline f_{-1},\oline \eta\rangle \oplus \langle \oline f_{-2},\oline \nu\rangle
\end{equation}
and the basis elements have nontrivial relations
\begin{eqnarray}
    ~&[\oline x, \oline e_0] = \oline e_{-1},\quad [\oline x, \oline f_{0}] = \oline f_{-1},\quad [\oline x, \oline f_{-1}] = \oline f_{-2},\\
    ~&[\oline e_0, \oline f_0]=\oline\eta,\quad [\oline f_0, \oline f_{-1}]=\oline\nu.
\end{eqnarray}
Now let $N$ be the simply connected Lie group with Lie algebra $\mathfrak{n}$ and let $\mathfrak{D}$ be the left-invariant distribution on $N$ given by $\mathfrak{n}_{-1}$. This distribution is called the \textit{Tanaka flat distribution} associated to the graded Lie algebra $\mathfrak{n}$.

The algebra of infinitesimal symmetries of $\mathfrak{D}$ is naturally isomorphic to the universal Tanaka prolongation $\mathfrak{g}(\mathfrak{n})$. We will soon see that $\mathfrak{g}(\mathfrak{n})$ and $\mathfrak g(\mathfrak s)$ are isomorphic as Lie algebras, though their gradings differ.

\subsection{The Prolonged Symbol}
\label{subsection: prol symb}
In order to give the universal Tanaka prolongations $\mathfrak{g} = \mathfrak{g}(\mathfrak s)$ and $\mathfrak g(\mathfrak n)$ explicitly, we will abstractly define an adapted basis for $\mathfrak{g}$ along with a Lie algebra involution $\theta$ on $\mathfrak{g}$, which will allow us to list fewer bracket relations in defining the Lie algebra structure. We will then identify $\mathfrak s$ and $\mathfrak n$ with the negative parts of $\mathfrak{g}$ under two gradings and show that under these identifications $\mathfrak{g}$ is the universal Tanaka prolongation of $\mathfrak s$ and $\mathfrak n$, respectively.

Define Lie algebra
\begin{equation}
    \label{def g}
    \mathfrak{g} = (\mathfrak{g}_2\oplus \mathbb{R})\ltimes W
\end{equation}
where $W$ is an isomorphic copy of the adjoint module of $\mathfrak g_2$, considered as an abelian Lie algebra, on which the line $\mathbb R$ acts by scaling.

Let us first define the bracket relations among a basis 
\begin{equation}
    \mathcal{B}_{\mathfrak{g}_2}=\{\oline{h}_1,\oline{h}_2,\oline{x},\oline{f}_1,\oline{f}_0,\oline{f}_{-1},\oline{f}_{-2},\oline{\nu},\oline{x}',\oline{f}_1',\oline{f}_0',\oline{f}_{-1}',\oline{f}_{-2}',\oline{\nu}'\}
\end{equation}
for the Levi subalgebra $\mathfrak{g}_2$ of $\mathfrak{g}$. Let $\mathfrak{h} = \mathrm{span}\{\oline{h}_1,\oline{h}_2\}$ be the Cartan subalgebra, and define a linear map $\theta$ on this basis by
\begin{gather}
    \label{def theta}
    \theta(a) = a'\quad\text{and}\quad \theta(a') = a \quad\text{for each }a\in \{\oline{x},\oline{f}_1,\oline{f}_0,\oline{f}_{-1},\oline{f}_{-2},\oline{\nu}\}
    \\
    \theta(\oline{h}_i)=-\oline{h}_i\quad \text{for }i=1,2
\end{gather}

All bracket relations among elements of the basis can be deduced from the following multiplication tables and the fact that $\theta$ is a Lie algebra involution.

% Optional: match the original array's row spacing (array default is 1.0).
% \renewcommand{\arraystretch}{1.0}

\begin{center}
% \tagpdfsetup{table/header-rows={1}, table/header-columns={1}}
\begin{tabular}{c|cccccccc}
& $\oline{h}_1$ & $\oline{h}_2$ & $\oline{x}$ & $\oline{f}_1$ & $\oline{f}_0$ & $\oline{f}_{-1}$ & $\oline{f}_{-2}$ & $\oline{\nu}$ \\ \hline
$\oline{h}_1$ & $0$ & $0$ & $2\oline{x}$ & $-3\oline{f}_1$ & $-\oline{f}_0$ & $\oline{f}_{-1}$ & $3\oline{f}_{-2}$ & $0$ \\
$\oline{h}_2$ & $0$ & $0$ & $-\oline{x}$ & $2\oline{f}_1$ & $\oline{f}_0$ & $0$ & $-\oline{f}_{-2}$ & $\oline{\nu}$ \\
$\oline{x}$  & $-2\oline{x}$ & $\oline{x}$ & $0$ & $\oline{f}_0$ & $\oline{f}_{-1}$ & $\oline{f}_{-2}$ & $0$ & $0$ \\
$\oline{f}_1$ & $3\oline{f}_1$ & $-2\oline{f}_1$ & $-\oline{f}_0$ & $0$ & $0$ & $0$ & $-\oline{\nu}$ & $0$ \\
$\oline{f}_0$ & $\oline{f}_0$ & $-\oline{f}_0$ & $-\oline{f}_{-1}$ & $0$ & $0$ & $\oline{\nu}$ & $0$ & $0$ \\
$\oline{f}_{-1}$ & $-\oline{f}_{-1}$ & $0$ & $-\oline{f}_{-2}$ & $0$ & $-\oline{\nu}$ & $0$ & $0$ & $0$ \\
$\oline{f}_{-2}$ & $-3\oline{f}_{-2}$ & $\oline{f}_{-2}$ & $0$ & $\oline{\nu}$ & $0$ & $0$ & $0$ & $0$ \\
$\oline{\nu}$ & $0$ & $-\oline{\nu}$ & $0$ & $0$ & $0$ & $0$ & $0$ & $0$ \\
\end{tabular}
\end{center}
\vspace{0.5cm}
 
\begin{center}
\scalebox{0.9}{
\begin{tabular}{c|cccccc}
 & $\oline{x}'$ & $\oline{f}_1'$ & $\oline{f}_0'$ & $\oline{f}_{-1}'$ & $\oline{f}_{-2}'$ & $\oline{\nu}'$ \\ \hline
$\oline{x}$  & $-\oline{h}_1$ & $0$ & $-3\oline{f}_1'$ & $-4\oline{f}_0'$ & $-3\oline{f}_{-1}'$ & $0$ \\
$\oline{f}_1$ & $0$ & $-\oline{h}_2$ & $\oline{x}'$ & $0$ & $0$ & $\oline{f}_{-2}'$ \\
$\oline{f}_0$ & $3\oline{f}_1$ & $-\oline{x}$ & $-\oline{h}_1 - 3\oline{h}_2$ & $4\oline{x}'$ & $0$ & $-3\oline{f}_{-1}'$ \\
$\oline{f}_{-1}$ & $4\oline{f}_0$ & $0$ & $-4\oline{x}$ & $-8\oline{h}_1 - 12\oline{h}_2$ & $12\oline{x}'$ & $12\oline{f}_0'$ \\
$\oline{f}_{-2}$ & $3\oline{f}_{-1}$ & $0$ & $0$ & $-12\oline{x}$ & $-36\oline{h}_1 - 36\oline{h}_2$ & $-36\oline{f}_1'$ \\
$\oline{\nu}$ & $0$ & $-\oline{f}_{-2}$ & $3\oline{f}_{-1}$ & $-12\oline{f}_0$ & $36\oline{f}_1$ & $-36\oline{h}_1 - 72\oline{h}_2$ \\
\end{tabular}
}

\end{center}
\vspace{0.5cm}

One can check that the Lie algebra we have constructed is indeed isomorphic to the split real form $\mathfrak g_2$ of the exceptional Lie algebra $\mathfrak{g}_2^\mathbb C$. In fact, the basis
\begin{equation}
    \Big\{\oline{h}_1,\oline{h}_2, \oline{x}, \oline{f}_1,-\oline{f}_0,\frac 12 \oline{f}_{-1}, \frac 16 \oline{f}_{-2}, -\frac16\oline{\nu}, \oline{x}',\oline{f}_1', -\oline{f}_0', \frac12\oline{f}_{-1}',\frac 16 \oline{f}_{-2}', -\frac16 \oline{\nu}'\Big\}
\end{equation}
is a Chevalley basis, and $\theta$ induces a Cartan involution of the complexified simple Lie algebra $\mathfrak{g}_2^\mathbb{C}$. Now choose $W$ to be the abelian Lie algebra with basis
\begin{equation}
    \mathcal{B}_W=\{\oline{w}_1,\oline{w}_2,\oline{y},\oline{e}_2,\oline{e}_1,\oline{e}_{0},\oline{e}_{-1},\oline{\eta}, \oline{y}',\oline{e}_2',\oline{e}_1',\oline{e}_{0}',\oline{e}_{-1}',\oline{\eta}'\}
\end{equation}
and let $\varphi:\mathfrak{g}_2\to W$ be the linear isomorphism defined by
\begin{equation}
    \varphi=\begin{cases}
        \oline{x}\mapsto \oline{y}
        \\
       \oline{f}_i\mapsto \oline{e}_{i+1}\text{ for }-2\leq i \leq 1
        \\
        \oline{\nu}\mapsto-\oline{\eta}
        \\
        \oline{h}_i\mapsto \oline{w}_i \text{ for }1\leq i \leq 2
    \end{cases}
\end{equation}
and $\varphi(a') = \varphi(a)'$ for each $a\in \{\oline{x},\oline{f}_1,\oline{f}_0,\oline{f}_{-1},\oline{f}_{-2},\oline{\nu}\}$. Let $\mathfrak{g}_2$ act on $W$ via $\rho:\mathfrak{g}_2\to \mathfrak{gl}(W)$ defined by
\begin{equation}
    \rho(a)(b) = \varphi\Big(\mathrm{ad}(a)\big(\varphi^{-1}(b)\big)\Big).
\end{equation}
Under this action, $W$ is isomorphic to the adjoint module of $\mathfrak{g}_2$. Finally, let the copy of $\mathbb{R}$ in the radical of $\mathfrak{g}$ be spanned by an element $r$ which acts as the identity on $W$. Extending the action of $\mathfrak{g}_2$ on $W$ by the trivial action of $\mathfrak{g_2}$ on $\langle r\rangle$, we have now constructed an explicit basis $\mathcal{B} = \mathcal{B}_{\mathfrak{g}_2}\cup\{r\}\cup\mathcal{B}_{W}$ for the universal Tanaka prolongation $\mathfrak{g} = \mathfrak{g}_2\ltimes\big(\langle r\rangle \ltimes W\big)$ along with all its bracket relations. 

Now let us establish two gradings $\mathrm{wt}_1$ and $\mathrm{wt}_2$ on $\mathfrak{g}$. Fix elements
\begin{equation}
    \label{grading elts}
    \oline H_1 = -2\oline h_1-3\oline h_2 + r,\quad \oline H_2 = -4\oline{h}_1-7\oline{h}_2 + r.
\end{equation} For $i=1,2$, define $\mathrm{wt}_i$ by
\begin{equation}
    \label{def Jacobi weight 1}
    [\oline H_i,a] = \mathrm{wt}_i(a)\cdot a\quad \text{for each }a\in \mathcal{B}
\end{equation}

This data can be expressed in the following diagrams, which are weight diagrams for the Cartan subalgebra $\langle \oline{h}_1,\oline{h}_2\rangle$ of $\mathfrak{g}_2$: 

% \begin{figure}[H]
    
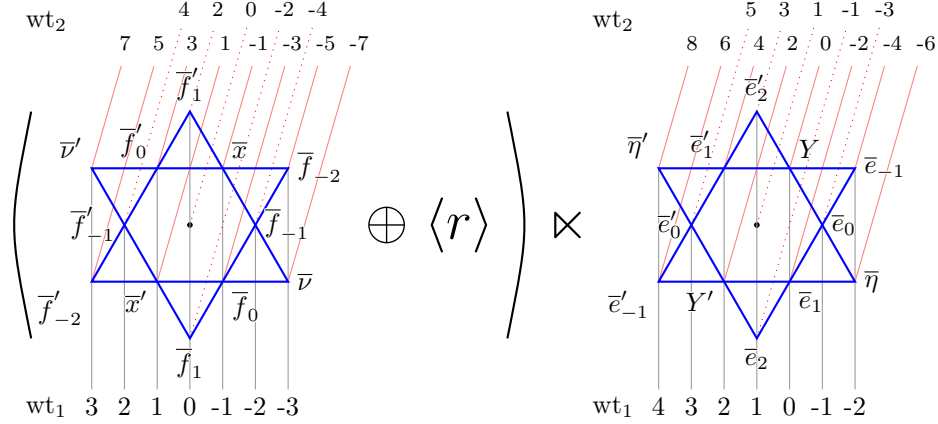
\begin{center}
    \begin{tikzpicture}[baseline=(current bounding box.center)]    
        % LEFT PARENTHESIS
        \draw[thick] (-4.6,-1.5) .. controls (-4.9,0) .. (-4.6,1.5);
        
    \begin{scope}[shift={(-2.5,0)}, scale=1.5]
        % Triangle 1 vertices (original angles 0, 120, 240 rotated by -30)
        \coordinate (T0) at ({cos(-30)}, {sin(-30)});
        \coordinate (T120) at ({cos(90)}, {sin(90)});
        \coordinate (T240) at ({cos(210)}, {sin(210)});
        % Triangle 2 vertices (original angles 60, 180, 300 rotated by -30)
        \coordinate (R60) at ({cos(30)}, {sin(30)});
        \coordinate (R180) at ({cos(150)}, {sin(150)});
        \coordinate (R300) at ({cos(270)}, {sin(270)});
        \filldraw[black] (0,0) circle (0.02);
        \path[name path=TA] (T0) -- (T120);
        \path[name path=TB] (T120) -- (T240);
        \path[name path=TC] (T240) -- (T0);
        \path[name path=RA] (R60) -- (R180);
        \path[name path=RB] (R180) -- (R300);
        \path[name path=RC] (R300) -- (R60);
        \path[name intersections={of=TA and RC, by=B0}];
        \path[name intersections={of=TA and RA, by=B60}];
        \path[name intersections={of=TB and RA, by=B120}];
        \path[name intersections={of=TB and RB, by=B180}];
        \path[name intersections={of=TC and RB, by=B240}];
        \path[name intersections={of=TC and RC, by=B300}];
        
        % Fixed y-coordinate for the weight rows
        \coordinate (wght1row) at (0,-1.6);
        \coordinate (wght2row) at (0,1.8);
    
        % Fixed y-coordinates for top and bottom of dotted lines
        \coordinate (toprow) at (0,-0.8);
        \coordinate (bottomrow) at (0,-1.45);
    
        % Row labels
        \coordinate (labelcol) at (-1.0, 0);
        \node[anchor=east] at (labelcol |- wght1row) {$\mathrm{wt}_1$};
        \node[anchor=east] at (labelcol |- wght2row) {$\mathrm{wt}_2$};
    
        % Drop-down lines for wght1
        \draw[white!60!black] (R180) -- (T240 |- bottomrow);
        \draw[white!60!black] (B180) -- (B180 |- bottomrow);
        \draw[white!60!black] (B120) -- (B240 |- bottomrow);
        \draw[white!60!black] (T120) -- (R300 |- bottomrow);
        \draw[white!60!black] (B60) -- (B60 |- bottomrow);
        \draw[white!60!black] (R60) -- (R60 |- bottomrow);
        \draw[white!60!black] (B0) -- (B0 |- bottomrow);
        
        % wght1 labels
        \node at (T240 |- wght1row) {3};
        \node at (B180 |- wght1row) {2};
        \node at (B240 |- wght1row) {1};
        \node at (R300 |- wght1row) {0};
        \node at (B60 |- wght1row) {-1};
        \node at (B0 |- wght1row) {-2};
        \node at (R60 |- wght1row) {-3};
        
        % wght2 labels
        \foreach \lbl [count=\i, evaluate=\i as \x using {-0.5705 + (\i-1)*0.295}] in {7,5,3,1,-1,-3,-5,-7} {
            \node at (\x, 1.6) {\footnotesize{\lbl}};
        }
        \foreach \lbl [count=\i, evaluate=\i as \x using {-0.05 + (\i-1)*0.295}] in {4,2,0,-2,-4} {
            \node at (\x, 1.9) {\footnotesize{\lbl}};
        }
        % \node at (-0.575,1.6) {\footnotesize{a}};
        
        \begin{scope} % odd wght2
            \clip (-1.5, -1) rectangle (1.5, 1.75);
            % Auxiliary points for w * \oline{x}' (weights != -1, 0, 1)
            \coordinate (auxB180) at ({-2*sqrt(3)/3}, {-2});     % w=5
            \coordinate (auxT120) at ({-sqrt(3)/3}, {-1});       % w=3
            \coordinate (auxR300) at ({sqrt(3)/3}, {1});         % w=-1
            \coordinate (auxB0) at ({2*sqrt(3)/3}, {2});         % w=-3
            
            % Auxiliary points for wght2 0
            \coordinate (auxW0a) at (0, 0);      % w=0
            \coordinate (auxW0b) at ({sqrt(3)/6}, 1);
        
            % Angled lines for wght2 through node and auxiliary point
            \draw[dotted, red] ($(B180)!-1!(auxB180)$) -- ($(B180)!0!(auxB180)$);  % w=5
            \draw[dotted, red] ($(T120)!0!(auxT120)$) -- ($(T120)!-1!(auxT120)$);  % w=3
            \draw[dotted, red] ($(R300)!0!(auxR300)$) -- ($(R300)!2!(auxR300)$);  % w=-1
            \draw[dotted, red] ($(B0)!0!(auxB0)$) -- ($(B0)!2!(auxB0)$);          % w=-3
            \draw[dotted, red] ($(auxW0a)!0!(auxW0b)$) -- ($(auxW0a)!3!(auxW0b)$);% w=1
        \end{scope}
    
        \begin{scope} % even wght2
            \clip (-1.5, -1) rectangle (1.5, 1.40);
            % Auxiliary points for w * \oline{x}' (weights != 0, 1, 2)
            \coordinate (auxT240) at ({-5*sqrt(3)/6}, {-5/2});   % w=6
            \coordinate (auxR180) at ({-7*sqrt(3)/6}, {-7/2});   % w=8
            \coordinate (auxB120) at ({-sqrt(3)/2}, {-3/2});     % w=4
            \coordinate (auxR60) at ({5*sqrt(3)/6}, {5/2});      % w=-4
            \coordinate (auxB300) at ({sqrt(3)/2}, {3/2});       % w=-2
            \coordinate (auxT0) at ({7*sqrt(3)/6}, {7/2});       % w=-6
            
            % Auxiliary points for weights -1, 0, 1
            \coordinate (auxB60) at (0, -1/2);   % w=0
            \coordinate (auxB240) at (0, 1/2);   % w=2
        
            % Angled lines for wght2 through node and auxiliary point
            \draw[white!50!red] ($(T240)!-1!(auxT240)$) -- ($(T240)!0!(auxT240)$); % w=6
            \draw[white!50!red] ($(R180)!0!(auxR180)$) -- ($(R180)!-1!(auxR180)$); % w=8
            \draw[white!50!red] ($(B120)!-1!(auxB120)$) -- ($(B120)!0!(auxB120)$); % w=4
            \draw[white!50!red] ($(R60)!0!(auxR60)$) -- ($(R60)!2!(auxR60)$);     % w=-4
            \draw[white!50!red] ($(B300)!0!(auxB300)$) -- ($(B300)!2!(auxB300)$); % w=-2
            \draw[white!50!red] ($(T0)!0!(auxT0)$) -- ($(T0)!2!(auxT0)$);         % w=-6
            \draw[white!50!red] ($(B60)!-1!(auxB60)$) -- ($(B60)!0!(auxB60)$);     % w=0
            \draw[white!50!red] ($(B240)!0!(auxB240)$) -- ($(B240)!2!(auxB240)$); % w=2
        \end{scope}
    
        \draw[blue, thick] (T0) -- (T120) -- (T240) -- cycle;
        \draw[blue, thick] (R60) -- (R180) -- (R300) -- cycle;
    
        \node[below right] at (B300) {\contour{white}{$\oline{f}_0$}};
        \node[right] at (B0) {\contour{white}{$\oline{f}_{-1}$}};
        \node[right] at (T0) {\contour{white}{$\oline \nu$}};
        \node[right] at (R60) {\contour{white}{$\oline{f}_{-2}$}};
        \node[below] at (R300) {\contour{white}{$\oline{f}_1$}};
        \node[above left] at (B120) {\contour{white}{${\oline{f}_0'}$}};
        \node[left] at (B180) {\contour{white}{$\oline{f}_{-1}'$}};
        \node[above left] at (R180) {\contour{white}{$\oline \nu'$}};
        \node[below left] at (T240) {\contour{white}{$\oline{f}_{-2}'$}};
        \node[above] at (T120) {\contour{white}{$\oline{f}_1'$}};
        \node[above right] at (B60) {\contour{white}{$\oline{x}$}};
        \node[below left] at (B240) {\contour{white}{$\oline{x}'$}};
    \end{scope}
    
    \begin{scope}[shift={(5,0)}, scale=1.5]
        % Triangle 1 vertices (original angles 0, 120, 240 rotated by -30)
        \coordinate (T0) at ({cos(-30)}, {sin(-30)});
        \coordinate (T120) at ({cos(90)}, {sin(90)});
        \coordinate (T240) at ({cos(210)}, {sin(210)});
        % Triangle 2 vertices (original angles 60, 180, 300 rotated by -30)
        \coordinate (R60) at ({cos(30)}, {sin(30)});
        \coordinate (R180) at ({cos(150)}, {sin(150)});
        \coordinate (R300) at ({cos(270)}, {sin(270)});
        \filldraw[black] (0,0) circle (0.02);
        \path[name path=TA] (T0) -- (T120);
        \path[name path=TB] (T120) -- (T240);
        \path[name path=TC] (T240) -- (T0);
        \path[name path=RA] (R60) -- (R180);
        \path[name path=RB] (R180) -- (R300);
        \path[name path=RC] (R300) -- (R60);
        \path[name intersections={of=TA and RC, by=B0}];
        \path[name intersections={of=TA and RA, by=B60}];
        \path[name intersections={of=TB and RA, by=B120}];
        \path[name intersections={of=TB and RB, by=B180}];
        \path[name intersections={of=TC and RB, by=B240}];
        \path[name intersections={of=TC and RC, by=B300}];
        
        % Fixed y-coordinate for the weight rows
        \coordinate (wght1row) at (0,-1.6);
        \coordinate (wght2row) at (0,1.8);
    
        % Fixed y-coordinates for top and bottom of dotted lines
        \coordinate (toprow) at (0,-0.8);
        \coordinate (bottomrow) at (0,-1.45);
    
        % Row labels
        \coordinate (labelcol) at (-1.0, 0);
        \node[anchor=east] at (labelcol |- wght1row) {$\mathrm{wt}_1$};
        \node[anchor=east] at (labelcol |- wght2row) {$\mathrm{wt}_2$};
    
        % Drop-down lines for wght1
        \draw[white!60!black] (R180) -- (T240 |- bottomrow);
        \draw[white!60!black] (B180) -- (B180 |- bottomrow);
        \draw[white!60!black] (B120) -- (B240 |- bottomrow);
        \draw[white!60!black] (T120) -- (R300 |- bottomrow);
        \draw[white!60!black] (B60) -- (B60 |- bottomrow);
        \draw[white!60!black] (R60) -- (R60 |- bottomrow);
        \draw[white!60!black] (B0) -- (B0 |- bottomrow);
        
        % wght1 labels
        \node at (T240 |- wght1row) {4};
        \node at (B180 |- wght1row) {3};
        \node at (B240 |- wght1row) {2};
        \node at (R300 |- wght1row) {1};
        \node at (B60 |- wght1row) {0};
        \node at (B0 |- wght1row) {-1};
        \node at (R60 |- wght1row) {-2};
    
        % wght2 labels
        \foreach \lbl [count=\i, evaluate=\i as \x using {-0.5705 + (\i-1)*0.295}] in {8,6,4,2,0,-2,-4,-6} {
            \node at (\x, 1.6) {\footnotesize{\lbl}};
        }
        \foreach \lbl [count=\i, evaluate=\i as \x using {-0.05 + (\i-1)*0.295}] in {5,3,1,-1,-3} {
            \node at (\x, 1.9) {\footnotesize{\lbl}};
        }
        % \node at (-0.575,1.6) {\footnotesize{a}};
        
        \begin{scope} % odd wght2
            \clip (-1.5, -1) rectangle (1.5, 1.75);
            % Auxiliary points for w * \oline{x}' (weights != -1, 0, 1)
            \coordinate (auxB180) at ({-2*sqrt(3)/3}, {-2});     % w=5
            \coordinate (auxT120) at ({-sqrt(3)/3}, {-1});       % w=3
            \coordinate (auxR300) at ({sqrt(3)/3}, {1});         % w=-1
            \coordinate (auxB0) at ({2*sqrt(3)/3}, {2});         % w=-3
            
            % Auxiliary points for wght2 0
            \coordinate (auxW0a) at (0, 0);      % w=0
            \coordinate (auxW0b) at ({sqrt(3)/6}, 1);
        
            % Angled lines for wght2 through node and auxiliary point
            \draw[dotted, red] ($(B180)!-1!(auxB180)$) -- ($(B180)!0!(auxB180)$);  % w=5
            \draw[dotted, red] ($(T120)!0!(auxT120)$) -- ($(T120)!-1!(auxT120)$);  % w=3
            \draw[dotted, red] ($(R300)!0!(auxR300)$) -- ($(R300)!2!(auxR300)$);  % w=-1
            \draw[dotted, red] ($(B0)!0!(auxB0)$) -- ($(B0)!2!(auxB0)$);          % w=-3
            \draw[dotted, red] ($(auxW0a)!0!(auxW0b)$) -- ($(auxW0a)!3!(auxW0b)$);% w=1
        \end{scope}
    
        \begin{scope} % even wght2
            \clip (-1.5, -1) rectangle (1.5, 1.40);
            % Auxiliary points for w * \oline{x}' (weights != 0, 1, 2)
            \coordinate (auxT240) at ({-5*sqrt(3)/6}, {-5/2});   % w=6
            \coordinate (auxR180) at ({-7*sqrt(3)/6}, {-7/2});   % w=8
            \coordinate (auxB120) at ({-sqrt(3)/2}, {-3/2});     % w=4
            \coordinate (auxR60) at ({5*sqrt(3)/6}, {5/2});      % w=-4
            \coordinate (auxB300) at ({sqrt(3)/2}, {3/2});       % w=-2
            \coordinate (auxT0) at ({7*sqrt(3)/6}, {7/2});       % w=-6
            
            % Auxiliary points for weights -1, 0, 1
            \coordinate (auxB60) at (0, -1/2);   % w=0
            \coordinate (auxB240) at (0, 1/2);   % w=2
        
            % Angled lines for wght2 through node and auxiliary point
            \draw[white!50!red] ($(T240)!-1!(auxT240)$) -- ($(T240)!0!(auxT240)$); % w=6
            \draw[white!50!red] ($(R180)!0!(auxR180)$) -- ($(R180)!-1!(auxR180)$); % w=8
            \draw[white!50!red] ($(B120)!-1!(auxB120)$) -- ($(B120)!0!(auxB120)$); % w=4
            \draw[white!50!red] ($(R60)!0!(auxR60)$) -- ($(R60)!2!(auxR60)$);     % w=-4
            \draw[white!50!red] ($(B300)!0!(auxB300)$) -- ($(B300)!2!(auxB300)$); % w=-2
            \draw[white!50!red] ($(T0)!0!(auxT0)$) -- ($(T0)!2!(auxT0)$);         % w=-6
            \draw[white!50!red] ($(B60)!-1!(auxB60)$) -- ($(B60)!0!(auxB60)$);     % w=0
            \draw[white!50!red] ($(B240)!0!(auxB240)$) -- ($(B240)!2!(auxB240)$); % w=2
        \end{scope}
    
        \draw[blue, thick] (T0) -- (T120) -- (T240) -- cycle;
        \draw[blue, thick] (R60) -- (R180) -- (R300) -- cycle;
    
        \node[below right] at (B300) {\contour{white}{$\oline{e}_1$}};
        \node[right] at (B0) {\contour{white}{$\oline{e}_{0}$}};
        \node[right] at (T0) {\contour{white}{$\oline \eta$}};
        \node[right] at (R60) {\contour{white}{$\oline{e}_{-1}$}};
        \node[below] at (R300) {\contour{white}{$\oline{e}_2$}};
        \node[above left] at (B120) {\contour{white}{${\oline{e}_1'}$}};
        \node[left] at (B180) {\contour{white}{$\oline{e}_{0}'$}};
        \node[above left] at (R180) {\contour{white}{$\oline \eta'$}};
        \node[below left] at (T240) {\contour{white}{$\oline{e}_{-1}'$}};
        \node[above] at (T120) {\contour{white}{$\oline{e}_2'$}};
        \node[above right] at (B60) {\contour{white}{$\oline{y}$}};
        \node[below left] at (B240) {\contour{white}{$\oline{y}'$}};
    \end{scope}
        
        % MATH SYMBOL
        \node at (1.3,0) {\Huge $\oplus\ \langle r \rangle \hspace{0.7cm}\ltimes$};
        
        % RIGHT PARENTHESIS
        \draw[thick] (1.7,-1.5) .. controls (2.0,0) .. (1.7,1.5);
    \alttext{Two weight diagrams side by side, showing how a Lie algebra g decomposes into a
reductive factor and a nilpotent factor. Inside a large pair of parentheses,
reading left to right: a first star-shaped diagram, then a direct-sum sign and a
one-dimensional factor spanned by r; the parenthesis then closes, followed by a
semidirect-product sign and a second star-shaped diagram. So the whole
expression is (first star direct sum r) semidirect product (second star). Each
star is a Star-of-David hexagram of two overlapping triangles, the twelve-root
diagram of the exceptional root system G 2, with a dot at the center for the
zero weight. The left star is a copy of the Lie algebra g 2; the right star is
an adjoint module of g 2. Each vertex is a basis vector positioned by two
weights, relative to a two-dimensional Cartan subalgebra: the first weight, wt
1, is read off the vertical columns, with an integer scale along the bottom (3
down to minus 3 on the left, 4 down to minus 2 on the right); the second weight,
wt 2, is read off two families of slanted parallel lines, with scales along the
top. The vertices are labeled by the algebra's basis vectors, for example x, nu,
and f subscript i with primed counterparts on the left, and Y, eta, and e
subscript i with primes on the right, as listed in the accompanying text.}
    \end{tikzpicture}
    \end{center}
    % \end{figure}
%     \caption{A decomposition of $\mathfrak g$ into its reductive and nilpotent factors, which are represented with weight diagrams for $\langle \overline H_1, \overline H_2\rangle$}
% \end{figure}
It is easily seen from the above diagram that the negative parts of $\mathfrak g$ with respect to $\mathrm{wt}_1$ and $\mathrm{wt}_2$ are isomorphic to $\mathfrak{n}$ from Section \ref{subsection: nonmax model} and $\mathfrak s$ from Proposition \ref{Prop: (3,8) Nonmax Jacobi Alg}, respectively. That is, $\mathfrak g$ satisfies item (i) of Definition \ref{def: Tanaka prol}. Item (ii) of the definition is also easily checked for both gradings. However, to see that $\mathfrak g$ is maximal among all graded Lie algebras satisfying (i) and (ii) in each grading, we have relied upon code utilizing the {$\mathsf{Differential Geometry}$} package in \textsc{Maple}. Altogether, we have proven the following proposition.

\begin{prop}
    \label{prop: Jacobi-Tanaka prolongation}
    The graded Lie algebra $\mathfrak g$ defined above is isomorphic to the universal Tanaka prolongation of $\mathfrak{n}$ (respectively, $\mathfrak s$) when equipped with the grading $\mathrm{wt}_1$ (respectively, $\mathrm{wt}_2$).
\end{prop}

\subsection{Proof of Theorem \ref{Thm: main theorem}}
\label{subsection: maximal symmetry}

We are now in a position to prove Theorem \ref{Thm: main theorem}. To see that $\mathfrak{D}$ has maximal infinitesimal symmetries among those $(3,8)$ distributions of nonmaximal class with $6$-dimensional square, note that by Proposition \ref{Prop: (3,8) Nonmax Jacobi Alg}, all such distributions have Jacobi-Tanaka algebra isomorphic to $\mathfrak s$, and therefore infinitesimal symmetries of dimension at most $\mathrm{dim}\big(\mathfrak g(\mathfrak s)\big) = 29$ by Remark \ref{Remark: Tanaka symm} and Proposition \ref{prop: Jacobi-Tanaka prolongation}.

In order to complete the proof of Theorem \ref{Thm: main theorem}, it remains only to show that every $(3,8)$ distribution of maximal class with $6$-dimensional square has infinitesimal symmetries of dimension less than $29$.
In the papers \cite{doubrov2008rank3} and \cite{doubrov2016JacobiCurves}, the second and third authors of the current paper demonstrate that the maximally symmetric $(3,n)$ distributional germ of maximal class has Jacobi symbol $(1,n-6)$; i.e., its skew Young diagram has maximal shift. Because we are momentarily restricting our attention to distributions of maximal class, Proposition \ref{prop: m rels} implies that the Jacobi-Tanaka algebra is determined by the Jacobi symbol. In \cite{doubrov2025large}, the same authors demonstrate that the Jacobi-Tanaka algebra corresponding to the Jacobi symbol $(1,n-6)$ has universal Tanaka prolongation of dimension $\mathrm{Fib}_{n-1}+n+2$, where $\mathrm{Fib}_{n}$ is the $n$th Fibonacci number with $\mathrm{Fib}_1 =\mathrm{Fib}_2 = 1$. By Remark \ref{Remark: Tanaka symm}, the maximal dimension of infinitesimal symmetry algebra for a $(3,8)$ distribution of maximal class is $\mathrm{Fib}_{7}+10 = 23$, and the proof of Theorem \ref{Thm: main theorem} is complete.

\section{Interpretation via the split octonions}
\label{octonions_sec}
The following exposition of the split octonions and its relation to the maximally symmetric $(2,3,5)$ distributional germ follows \cite{AgrachevRollingBalls} closely; we add an interpretation of the $(3,6,8)$ distribution $\mathfrak{D}$ using the tangent bundle $T\oct$ and its natural algebraic structure.

Let $\oct$ be the split octonion algebra, which can be viewed as the vector space sum $\mathbb{H}\oplus\mathbb{H}$ of quaternion algebras equipped with multiplication law
\[
    (a+\ell b)(c + \ell d) = (ac + d\oline{b}) + \ell(\oline a d + cb)
\]
for $a,b,c,d\in \mathbb{H}$. It is well-known that the automorphism group of $\oct$ is the split real form $G_2$ of the Lie group $G_2^\mathbb C$; let $\phi:G_2\to \mathrm{Aut}(\oct)$ be this action. For any $x = a + \ell b\in\oct,$ define conjugate $\oline{x} = \oline{a} - \ell b$. This conjugation gives rise to the isotropic quadratic form $Q(x) = \oline{x}x = |a|^2-|b|^2$, which satisfies $Q(xy)=Q(x)Q(y)$ for all $x,y\in\oct$, making $\oct$ a (non-associative) composition algebra. The polarization of $Q$ is then a signature $(4,4)$ nondegenerate bilinear form $Q(x,y) = \frac{1}{2}(\oline xy+\oline yx)$. Write
\begin{equation}
    \mathbb R^7 = \{x\in \oct: Q(1,x)=0\}
\end{equation} for the subspace of pure imaginary octonions. With this identification, $\mathbb{R}^7$ inherits the cross-product $x\times y = \mathrm{Im}(xy) = \frac{1}{2}(xy-yx)$, which satisfies
\begin{gather}
    \label{cross prod antisymm}
    x \times y = -y\times x
    \\
    \label{cross prod orthogonality}
    Q(x, x\times y) = 0
    \\
    \label{cross prod magnitude}
    Q(x\times y) = Q(x)Q(y)-Q(x,y)^2
\end{gather}
for all $x,y\in \mathbb{R}^7$. The cone $Q^{-1}(0)$ is precisely the set of zero divisors in $\oct$. Define
\begin{equation}
    K:=\{x\in\mathbb{R}^7 : Q(x)=0\}\setminus\{0\},
\end{equation}
the cone of pure imaginary zero divisors.

Since $K\subseteq \mathbb R^7$, for each $x\in K$ we can identify
\begin{equation}
\label{TxK}
    T_xK = \{y\in \mathbb{R}^7: Q(x,y) = 0\}
\end{equation}

The algebra $\oct$ is an alternative algebra, which means the following identities hold:
\begin{equation}
    (xy)x = x(yx),\quad x(xy) = (xx)y, \quad (yx)x = y(xx)
\end{equation}
Writing out the commutator definition of the cross-product and applying these identities yields
\begin{align}
\label{x times x times y}
    x\times (x\times y) &= 0 \quad\text{for any }x\in K, y\in T_xK.
\end{align}

Using the identification \eqref{TxK}, define the distribution $\Delta$ on $K$ by
\begin{equation}
\label{def Delta}
    \Delta(x) = \{y \in T_xK : x\times y=0\}.
\end{equation}
One can observe that $\Delta$ has rank $3$. Indeed, a simple computation shows that for any nonzero zero divisor $x\in \mathbb{R}^7$, the equation $x\times y = 0$ reduces to a single quaternionic equation, and for each solution $y\in \mathbb{R}^7$ of $x\times y=0$, equation \eqref{cross prod magnitude} implies that $Q(x,y)=0$, so that $y\in T_xK$. Thus $\Delta_x$ has codimension $4$ in $T_x\mathbb{R}^7$, and $\dim(\Delta_x)=3$.

It will later be useful to consider the Pfaffian system $I_\Delta$ of $1$-forms which annihilate $\Delta$. The $\mathbb{R}^7$-valued $1$-form $\omega_0 = x\times dx$ on $\mathbb{R}^7$ defines $\Delta$ when restricted to $K$. Because the $G_2$ action on $\mathbb{R}^7$ preserves the cross-product, $\omega_0$ satisfies the equivariance property 
\begin{equation}
    \label{omega_0 equivar}
    \phi(g)^*\omega_0 = \phi(g)\circ \omega_0
\end{equation}
By \eqref{cross prod orthogonality} and \eqref{x times x times y}, one can see that $\omega_0|_{TK}$ in fact takes values in the distribution $\Delta$. We can therefore write
\begin{equation}
    \label{def: I Delta}
    I_\Delta =  \{Q(u,\omega_0): u\in \mathbb{R}^7\} 
\end{equation}

Let $\mathcal{E}$ be the Euler vector field on $K$; that is, $\mathcal{E}(x) := x$. It is easy to show that $\mathcal{E}(x)$ is a Cauchy characteristic of $\Delta.$ Let $s:K\to \mathbb{P}K$ be the projectivization of the cone $K$. Because the Euler vector field $\mathcal{E}$ is a Cauchy characteristic vector field of $\Delta$, the rank $2$ distribution $\boldsymbol{\Delta}:=s_*\Delta$ is well-defined.

Because the automorphisms of $\oct$ preserve the norm $Q$, the unit $1\in \oct$, and the Euler vector field $\mathcal{E}$, the action of $G_2$ on $\oct$ preserves $K$ and passes to a smooth action on $\mathbb{P} K$ which preserves the distribution $\boldsymbol{\Delta}$. Further, choosing a basis for $\mathbb R^7$ from $K$, one can deduce that any automorphism in $G_2$ which acts trivially on $K$ acts trivially on all of $\oct$, so the action of $G_2$ on $K$ is effective.

Viewing $\Delta\subseteq \oct^2$ as a cone in a real vector space, consider the projectivization $\mathbb P_{\oct^2}(\Delta)\subseteq \mathbb P(\oct^2)$, which has a natural projection to $\mathbb PK$. We use $\mathbb P_{\oct^2}(\Delta)$ to distinguish this projectivization from the fiberwise projectivization $\mathbb P \Delta$. We now construct a distribution $\boldsymbol{\Delta}'$ on $\mathbb P_{\oct^2}(\Delta)$ which projects to $\boldsymbol{\Delta}$ and is locally equivalent to the distribution $\mathfrak D$ defined in Section \ref{subsection: nonmax model}.

Viewing tangent vectors as jets of parametrized curves, $T\oct$ is naturally an algebra with multiplication
\begin{equation}
\label{tang mult}
    j^1_0\big(x(t)\big)\cdot j^1_0\big(y(t)\big) = j^1_0\big(x(t)\cdot y(t)\big)
\end{equation}
for smooth curves $x(t)$ and $y(t)$ in $\oct$. Under the canonical identification $T\oct = \oct ^2$, this multiplication can be written
\begin{equation}
    (x_0,x_1)\cdot (y_0,y_1) = (x_0y_0,x_0y_1+x_1y_0)\quad\text{for }x_i,y_i\in \oct.
\end{equation}
We can also use \eqref{tang mult} to equip $TG_2$ with Lie group structure. We then have a natural Lie group isomorphism $TG_2 = G_2\ltimes W$, where $W$ is the adjoint module of $G_2$, which is naturally identified with ${\mathfrak{g}}_2$ as a module. We denote by $\Phi$ the natural action of $TG_2$ on $T\oct$ via automorphisms 
\begin{equation}
\label{def tan action}
    \Phi\Big(j^1_0\big(g(t)\big)\Big)\Big(j^1_0\big(x(t)\big)\Big) = j^1_0\Big(\phi\big(g(t)\big)\big(x(t)\big)\Big)
\end{equation}
for smooth curves $g(t)\in G_2$ and $x(t)\in \oct$.

Identifying the adjoint module $W$ with $T_eG_2$ gives a natural action of $W$ on $\oct$ via $\phi'=T_e\phi$. Under the identifications $TG_2\cong G_2\ltimes W$ and $T\oct = \oct^2$, the action \eqref{def tan action} can then be written
\begin{equation}
\label{def TG action}
    \Phi(g, X)(x_0,x_1) = \Big(\phi(g)(x_0), \phi(g)\circ \phi'(X)(x_0) + \phi(g)(x_1)\Big)
\end{equation}
where $x = (x_0,x_1)\in TK\subseteq \oct^2$ and $(g,X)\in G_2\ltimes W$. 

Since the action of $TG_2$ on $T\oct$ is by automorphisms, it preserves the pure imaginary subbundle $T\mathbb{R}^7\subseteq T\oct$. Let $(x_0,x_1) = (x_0^1,x_0^2,\ldots, x_0^7,x_1^1,x_1^2,\ldots x_1^7)$ be coordinates on the vector space $T\mathbb{R}^7 \cong \mathbb{R}^7\oplus \mathbb{R}^7$ corresponding to the canonical ordered basis $(i,j,k,\ell,\ell i,\ell j, \ell k)$ of $\mathbb{R}^7$.

Now define an $\mathbb{R}^7$-valued $1$-form $\omega_1 = x_0\times dx_1 + x_1\times dx_0$ on the manifold $T\mathbb{R}^7$. A short computation shows that $\omega_1$ satisfies transformation law
\begin{equation}
    \label{omega_1 transformation law}
    \Phi(g,X)^*\omega_1 = \phi(g)\big(\omega_1 + \phi'(X)\circ \pi^*\omega_0\big).
\end{equation}
Restricting $\omega_1$ to $TTK$, consider the distribution $\Delta'$ on $TK$ which is annihilated by the Pfaffian system
\begin{equation}
\label{def I Delta'}
    I_{\Delta'}= \pi^*(I_\Delta)\oplus\{Q(v,\omega_1): v\in \Gamma(TK)\} 
\end{equation}
where $\pi:TK\to K$ is the canonical projection. Note that we mildly abuse notation above by writing $v$ for $v\circ \pi$, considered as an $\mathbb{R}^7$-valued function on $TK$.
\begin{lemma}
    \label{lemma: TG2 invar}
    The distribution $\Delta'$ is invariant under the action of $TG_2$ defined in \eqref{def TG action}.
\end{lemma}
\begin{proof}
    For any $(g,X)\in G_2\ltimes W$, applying the definition \eqref{def TG action} yields
    \begin{gather}
        \Phi(g,X)^*(\pi^*\omega_0) = \big(\pi\circ \Phi(g,X\big)\big)^*\omega_0 = \big(\phi(g)\circ\pi\big)^*\omega_0 = \pi^*\big(\phi(g)\circ \omega_0\big)
    \end{gather}
    where the last equality follows by \eqref{omega_0 equivar}. Since $G_2$ acts on $\oct$ by automorphisms, this implies that $\Phi(g,X)^*(\pi^*I_\Delta) = \pi^*I_\Delta$.

    Now let $x = (x_0,x_1)\in TK$ be arbitrary, and let $\tilde x = (\tilde x_0,\tilde x_1) = \Phi(g,X)(x)$. The transformation law \eqref{omega_1 transformation law} implies that for any $v\in \Gamma(TK)$,
    \begin{gather}
        \Phi(g,X)^*\big(Q(v,\omega_1)\big)_x = Q\Big(v(\tilde x_0),\big(\Phi(g,X)^*\omega_1\big)_x\Big) 
        \\
        = Q\Big(v(\tilde x_0),\phi(g)\big(\omega_1+\phi'(X)\circ \pi^*\omega_0\big)_{x}\Big)
    \end{gather}
    Since $\phi(g)$ preserves $Q$, this can be rewritten
    \begin{gather}
        = Q\big(\phi(g^{-1})v(\tilde x_0),(\omega_1)_x+\phi'(X)\circ (\pi^*\omega_0)_{x}\big)
        \\
        \label{invar_proof_expr}
        =Q\big(\phi(g^{-1})v(\tilde x_0),(\omega_1)_x\big)
        +
        Q\big(\phi(g^{-1})v(\tilde x_0),\phi'(X)\circ (\pi^*\omega_0)_{x}\big)
    \end{gather}
    Since $\tilde x_0 = \phi(g)(x_0)$, we have that $\phi(g^{-1})v(\tilde x_0)\in T_{x_0}K$, and therefore the first summand of \eqref{invar_proof_expr} is in $(I_{\Delta'})_x$; because $(\pi^*\omega_0)_{x}\in (I_{\Delta'})_x$, so too is the second summand of \eqref{invar_proof_expr}.
    Thus $\Phi(g,X)^*Q(v,\omega_1)_x\in (I_{\Delta'})_x$, as desired.
\end{proof}

When computing rank of $\Delta'$, we may restrict our focus to a single point, since $TG_2$ acts transitively on $TK$. To this end, fix $p=(p_0,p_1)= (i + \ell,0)\in TK$.

\begin{lemma}
    \label{lemma: Delta' rank}
    The distribution $\Delta'$ on $TK$ defined in \eqref{def I Delta'} has rank $7$.
\end{lemma}
\begin{proof}
    Recall that we have coordinates $(x_0^1,\ldots, x_0^7,x_1^1,\ldots, x_1^7)$ on $T\mathbb{R}^7$ corresponding to the standard basis of $\mathbb{R}^7\subseteq \oct$. Applying \eqref{TxK}, $T_pTK\subseteq T_p(T\mathbb{R}^7)$ is defined by the 1-forms
    \begin{gather}
        Q(p_0,dx_0) = dx_0^4-dx_0^1 = 0
        \\
        Q(p_0, dx_1) + Q(p_1, dx_0) = dx_1^4-dx_1^1 = 0.
    \end{gather}
    By direct computation, we have equivalence of forms on $T_pTK$
    \begin{gather}
        (\pi^*\omega_0)_{p} = p_0\times dx_0 = dx_0^5\cdot p_0 + (dx_0^2+dx_0^7)\cdot z_1 + (dx_0^3-dx_0^6)\cdot z_2
    \end{gather}
    where $z_1 = k +\ell j$ and $z_2 = -j+\ell k$. Similarly, on $T_pTK$ we have
    \begin{gather}
        \label{omega_1_p}
        (\omega_1)_p = dx_1^5\cdot p_0 + (dx_1^2+dx_1^7)\cdot z_1 + (dx_1^3-dx_1^6)\cdot z_2
    \end{gather}
    Since $T_{p_0}K=\{v\in \mathbb R^7: Q(p_0,v)=0\}$, we have for any $v\in T_{p_0}K$ that 
    \begin{align}
        Q\big(v,(\omega_1)_p\big) 
        & = Q(v, z_1)(dx_1^2+dx_1^7) + Q(v,z_2) (dx_1^3-dx_1^6)
        \\ & = (v_3-v_6)(dx_1^2+dx_1^7) + (-v_2-v_7)(dx_1^3-dx_1^6).
    \end{align}
    Therefore, by \eqref{def I Delta'}, defining forms of $(\Delta')_p\subseteq T_pTK$ can then be written
    \begin{gather}
        (I_{\Delta'})_p = \mathrm{span}\{dx_0^5,dx_0^2+dx_0^7,dx_0^3-dx_0^6,dx_1^2+dx_1^7,dx_1^3-dx_1^6\}
    \end{gather}
    Therefore, $\Delta'_p$ has rank $\dim(TK)-\dim\big((I_{\Delta'})_p\big)=12-5=7$.
\end{proof}

Let us now compute the rank of $(\Delta')^2$. Recall the defining forms of $(\Delta')^2\subseteq TTK$ are given by
\begin{equation}
    I_{\Delta'}^{(1)} = \{\varphi\in I_{\Delta'}: d\varphi|_{\Delta'}=0\}
\end{equation}

\begin{lemma}
    \label{lemma: Delta' growth}
    The square $(\Delta')^2$ of $\Delta'$ has rank $10$, and $(\Delta')^3 = TTK$.
\end{lemma}
\begin{proof}
    As above, fix $p_0 = i + \ell$, $z_1 = k+\ell j$, and $z_2 = -j +\ell k$, so that $\Delta_{p_0} = \mathrm{span}\{p_0,z_1,z_2\}$. Because $\Delta'$ is homogeneous, it suffices to show that $\dim\big((I_{\Delta'}^{(1)})_{p}\big) = 2$, where $p= (p_0,0)\in TK$. 
    
    Fix a vector $u\in \mathbb{R}^7$ and a vector field $v\in \Gamma(TK)$ to obtain an arbitrary differential 1-form
    \begin{equation}
        \label{def varphi}
        \varphi =Q(u,\pi^*\omega_0) +  Q(v, \omega_1) \in I_{\Delta'}
    \end{equation}
    and suppose that $d\varphi|_{\Delta'}=0$ so that $\varphi\in I_{\Delta'}^{(1)}$. 

    Firstly, we show that $Q(v,\omega_1)_{p} = 0$. Fix $z_3 = \ell i$, and observe that 
    \begin{gather}
        \label{z3 cross conditions}
        p_0\times z_3 = p_0, \quad z_1\times z_3 = -z_1,\quad \text{and}\quad z_2\times z_3 = -z_2.
    \end{gather}
    We can compute
    \begin{equation}
        \label{dN}
        d\big(Q(v,\omega_1)\big) = Q(dv,\omega_1) + Q(v,d\omega_1)
    \end{equation}
    where $dv$ is the differential of the function $v:TK\to \mathbb{R}^7$. Consider the vectors $(z_1,0),(z_2,0),$ and $(0,z_3)\in \Delta'_p$. For $a=1,2$, we can use the definitions of $\omega_0$ and $\omega_1$ to compute
    \begin{gather}
        0=d\varphi\big((z_a,0)\wedge(0,z_3)\big) = d\big(Q(v,\omega_1)\big)\big((z_a,0)\wedge(0,z_3)\big) 
        \\
        = Q\Big(dv(z_a), \omega_1\big((0,z_3)\big)\Big) + Q\Big(v,d\omega_1\big((z_a,0)\wedge (0,z_3)\big)\Big)
        \\
        = Q\big(dv(z_a),x_0\times z_3\big) + Q(v, 2\cdot z_a\times z_3)
        = Q\big(dv(z_a),x_0\big) - 2\cdot Q(v, z_a)
    \end{gather}
    where the last equality is by \eqref{z3 cross conditions}. Because $Q(x_0,v)=0$, we have that
    \begin{equation}
    \label{dv cond}
        Q(v,dx_0) + Q(dv,x_0)=0,
    \end{equation} so we can continue the above equalities
    \begin{gather}
        Q\big(dv(z_a),x_0\big) - 2\cdot Q(v, z_a) = - 3\cdot Q(v,z_a),
    \end{gather}
    and therefore $Q(v,z_a) = 0$ for $i=1,2$. Since $\Delta_{p_0} = \mathrm{span}\{p_0,z_1,z_2\}$ and $v_{p_0}\in T_{p_0}K$, this implies that $v_{p_0}\in \Delta^\perp_{p_0}$. Because \eqref{omega_1_p} shows $(\omega_1)_p$ takes values in $\Delta_{p_0}$, we have now shown that $Q(v,\omega_1)_{p}=0$ and hence that $\varphi = Q(u,\pi^*\omega_0)$.

    We can now see that
    \begin{equation}
        d\varphi = Q(u,d\omega_0) = Q(u,2dx_0\times dx_0).
    \end{equation}
    Again restricting to the point $p$, we have that 
    \begin{gather}
        d\varphi_{p}\big((p_0,0),(z_1,0)) = Q(u,2p_0\times z_1) = 0
        \\
        d\varphi_{p}\big((p_0,0),(z_2,0)) = Q(u,2p_0\times z_2) = 0
        \\
        d\varphi_{p}\big((z_1,0),(z_2,0)) = Q(u,2z_1\times z_2) = Q(u,4p_0).
    \end{gather}
    Therefore, $\varphi_{p}\in (I_{\Delta'}^{(1)})_p$ if and only if $u_{p_0}\in p_0^\perp = T_{p_0}K$. Since 
    \[
        Q\big(u_{p_0},\pi^*(\omega_0)_{p_0}\big) \equiv 0\Longleftrightarrow u_{p_0}\in \Delta^\perp_{p_0}\subseteq T_{p_0}K,
    \]
    we have that
    \begin{equation}
        \dim\big(I_{\Delta'}^{(1)}\big)_{p_0} = \dim(T_{p_0}K) - \dim(\Delta_{p_0}^\perp) = 6-4 = 2.
    \end{equation}
    This demonstrates that $(\Delta')^2 = \pi^*\big(\Delta^2\big)$ has rank $10$. To see that $(\Delta')^3 = TTK$, observe that
    \begin{equation}
        \pi_*\big((\Delta')^{3}\big) = \pi_*\big(\pi^*\big(\Delta^2\big) + [\Delta',\pi^*\big(\Delta^2\big)]\big) = \Delta^2 + [\Delta, \Delta^2] = \Delta^3 = TK
    \end{equation}
    and that since $(\Delta')^2 = \pi^*(\Delta^2)$ contains the vertical subbundle for the projection $TTK\to TK$, so does $(\Delta')^3$.
\end{proof}

Also consider the rank 4 distribution $C$ on $TK$ defined at $x= (x_0 , x_1)\in TK$ by
\begin{equation}
\label{def C}
    C(x) = \mathbb{R}\tilde{\mathcal{E}}(x) \oplus \{(0,y_1)\in T_xTK: y_1\in \Delta_{x_0}\}.
\end{equation}
where $\tilde{\mathcal{E}}(x) = x$ is the Euler vector field on $TK\subseteq T\oct$.

\begin{prop}
    \label{prop: C is Cauchy Char}
    The distribution $C$ on $TK$ is invariant under the action $\Phi$ of $TG_2$ and is a Cauchy characteristic of the distribution $\Delta'$.
\end{prop}
\begin{proof}
 Let $\tilde x = (\tilde x_0, \tilde x_1) = \Phi(g,X)(x)$ be the image of $x$ under the action \eqref{def TG action}. The action $\Phi$ is by linear maps, so that for any $(g,X)\in G_2\ltimes W$, any $x=(x_0,x_1)\in TK$, and any $y=(y_0, y_1)\in T_xTK$,
\begin{align}
    \label{y trans law}
    T_x\Phi(g,X)(y) & = \Big(\phi(g)(y_0),\phi(g)\circ\phi'(X)(y_0) + \phi(g)(y_1)\Big)
\end{align}
As with $\tilde x$, define $\tilde y = (\tilde y_0,\tilde y_1) = T_x\Phi(g,X)(y)$.

For the first assertion, observe that for any $y=(0,y_1)\in T_xTK$ with $y_1\in \Delta_{x_0}$, 
\begin{align}
    \tilde x_0\times \tilde y_1 & = \phi(g)(x_0)\times \phi(g)(y_1)
    \\
    & = \phi(g)(x_0\times y_1) = 0
\end{align}
where the last equality follows from the fact that $y_1\in \Delta_{x_0}$. Similarly, the transformation law \eqref{y trans law} gives immediately that $\Phi$ preserves $\tilde{\mathcal{E}}$. Therefore, $C$ is invariant under the action $\Phi$.

To see that $C$ is a Cauchy characteristic of $\Delta'$, fix a vector $y = (0,y_1)\in T_xTK$ with $y_1\in\Delta_{x_0}$, and consider the insertion operator $\iota_y:(T^*_xTK)^{\wedge 2}\to T_x^*TK$; we will show that $\iota_y\big(d(I_{\Delta'})_x\big) \subseteq (I_{\Delta'})_x$, so that $y$ is a Cauchy characteristic of $\Delta'$. Since $y$ is vertical, we have that
\begin{align}
    \iota_y\circ d\pi^*\omega_0 = \iota_y\circ \pi^*(d\omega_0) = 0.
\end{align}
Also, for any section $v\in \Gamma(TK)$, 
\begin{gather}
    \iota_y \circ d\big(Q(v,\omega_1)\big) = Q(v_x,\iota_y d\omega_1) 
    \\
    = 2\cdot Q\big(v_{x_0},\iota_y(dx_0\times dx_1)\big) = 2\cdot Q(v_{x_0},dx_0\times y_1)
\end{gather}
where the first equality follows by \eqref{dN} and the fact that $y$ is vertical. Restricting the above form to $\Delta'$ yields image
\begin{gather}
    Q(v_{x_0},\Delta_{x_0}\times y_1) \subseteq Q(v_{x_0},\Delta_{x_0}\times \Delta_{x_0}) \subseteq Q(v_{x_0},\mathbb{R}x_0) = 0
\end{gather}
since $v_{x_0}\in T_{x_0}K$. Thus $y$ is a Cauchy characteristic of $\Delta'$. Similarly, we have
\begin{gather}
    \iota_{\tilde{\mathcal{E}}}\circ d\pi^*\omega_0 = \iota_{\tilde{\mathcal{E}}}\circ \pi^*(dx_0\times dx_0) = 2 \cdot x_0\times dx_0 = 2\cdot \pi^*\omega_0
    \\
    \iota_{\tilde{\mathcal{E}}}\circ d\omega_1 = \iota_{\tilde{\mathcal{E}}}\cdot(dx_0\times dx_1 + dx_1\times dx_0) = 2\cdot \omega_1.
\end{gather}
Therefore, $\tilde{\mathcal{E}}$ is also a Cauchy characteristic of $\Delta'$, as desired.
\end{proof}

The quotient of $TK$ by the (global) foliation generated by the distribution $C$ is in fact as the projectivization $\mathbb P_{\oct}\Delta$ of $\Delta$. Define the map
    \begin{equation}
    \label{Psi def}
        \psi: TK\to \Delta; \psi(x_0,x_1) = \big(N(x_0)x_0,x_0\times x_1\big)
    \end{equation}
    where $N$ is the Euclidean norm on $\oct\cong\mathbb{R}^8$. To see that $\psi$ takes values in $\Delta$, observe that the scalar triple product on $\mathbb{R}^7$ defined by
    \begin{equation}
        \label{def: trip prod}
        (\mathbb{R}^7)^3\to \mathbb{R}; (a,b,c)\mapsto Q(a,b\times c)
    \end{equation}
    is cyclic. Therefore
    \begin{equation}
        Q\big(x_0,x_0\times x_1\big) = Q\big(x_1,x_0\times x_0\big) = 0
    \end{equation}
    so that $x_0\times x_1\in T_{N(x_0)x_0}K$. By \eqref{x times x times y}, $N(x_0)x_0\times (x_0\times x_1) = 0$, and $\psi(x_0,x_1) \in \Delta_{x_0}$. In fact, because $\Delta_{x_0} = \ker\big((x_0\times -)|_{T_{x_0}K}\big)$, the rank-nullity theorem implies
    \begin{equation}
        \label{Delta is cross im}
        \Delta_{x_0} = x_0\times T_{x_0}K
    \end{equation}

    Now let $\varpi:\Delta\to \mathbb{P}_{\oct^2}(\Delta)\subseteq \mathbb P\oct^2$ be the projectivization of the cone $\Delta\subseteq \oct^2$, and let $\Psi:=\varpi\circ\psi$.
    
    \begin{prop}
        The map $\Psi: TK\xrightarrow{\psi}\Delta\xrightarrow{\varpi}\mathbb{P}_{\oct^2}(\Delta)$ is a surjective submersion, and at each point $p\in TK$, $\mathrm{ker}(T_p\Psi) = C(p)$.
    \end{prop}
    \begin{proof}
        Fix $p = (p_0,p_1)\in TK$. Including $\Delta$ into $T\mathbb R^7 = \mathbb R^7\times \mathbb R^7$, we can write 
        \begin{gather}
            T_p\psi 
            = \Big(d\big(N(x_0)x_0),dx_0\times x_1 + x_0\times dx_1\Big)_p
        \end{gather}
        If $v = (v_0,v_1)\in \ker(T_p\psi)$, then it is easy to see that $v_0=0$. Thus
        \begin{equation}
             0 = T_p\psi(v) = \big(0,d\omega_1(v)\big) = (0,p_0\times v_1)
        \end{equation}
        so that $v_1\in \Delta_{p_0}$. Therefore 
        \begin{equation}
            \ker(T_p\psi) = \{(0,v_1): v_1\in \Delta_{p_0}\},
        \end{equation}
        which is the second summand of \eqref{def C}. The map $\psi$ is surjective by \eqref{Delta is cross im}.

        Now $\psi$ has homogeneous degree $2$, so $T_p\psi\big(\tilde{\mathcal{E}}(p)\big) = 2\tilde{\mathcal{E}}\big(\psi(p)\big)$, and \linebreak[4] $\mathrm{ker}(T_{\psi(p)}\varpi) = \mathbb{R}\tilde{\mathcal{E}}\big(\psi(p)\big)$. Therefore 
        \begin{equation}
            \ker(T_p\Psi) = T_p\psi^{-1}\Big(\mathbb R\tilde{\mathcal{E}}\big(\psi(p)\big)\Big) = C(p)
        \end{equation}
        as desired. Since $\varpi$ and $\psi$ are surjective, so too is $\Psi=\varpi\circ \psi$; counting ranks gives that $\Psi$ is also a submersion.
    \end{proof}

    By Proposition \ref{prop: C is Cauchy Char}, the distribution $\Delta'$ passes to a rank $3$ distribution $\boldsymbol{\Delta}':=(\varpi\circ \psi)_*\Delta'$ on the $8$-dimensional manifold $\mathbb{P}_{\oct^2}(\Delta)$.

    It is well-known that the action of ${G}_2$ on $\mathbb{P}K$ induced by $\phi$ is infinitesimally effective; that is if $X\in {\mathfrak{g}}_2$ has $\phi'(X)=0\in \Gamma(T\mathbb PK)$, then $X=0$. We will use this fact to prove the following lemma:
    \begin{lemma}
        \label{lemma: TG2 inf eff}
        The action of $TG_2$ on $\mathbb{P}_{\oct^2}(\Delta)$ induced by $\Phi$ is infinitesimally effective.
    \end{lemma}
    \begin{proof}
        Suppose that $\big(g(t),X(t)\big)\in G_2\ltimes W$ with $\big(g(0),X(0)\big) = (1_{G_2},0)$ is a curve such that the vector field $\Phi'\big(g'(0),X'(0)\big)$ on $TK$ takes its values in the characteristic $C$; that is, the vector field has trivial infinitesimal action on $\mathbb P_{\oct^2}(\Delta)$. Differentiating the first component of \eqref{def TG action}, this implies that $\phi'\big(g'(0)\big)$ takes its values in the line distribution $\mathbb{R}\mathcal E$ on $K$, so that $\phi'\big(g'(0)\big)$ acts trivially on $\mathbb{P}K$. Because the action of $G_2$ on $\mathbb{P}K$ induced by $\phi$ is infinitesimally effective, this implies $g'(0)=0$. Hence $\Phi'\big(g'(0),X'(0)\big)$ must take its values in the second summand of \eqref{def C}.
        
        Differentiating the second component of \eqref{def TG action} yields
        \begin{gather}
            \Phi'\big(g'(0),X'(0)\big)_{(x_0,x_1)} 
            \\
            \label{inf eff 1}
            = \frac{d}{dt}\Big(\phi\big(g(t)\big)(x_0), \phi\big(g(t)\big)\circ \phi'\big(X(t)\big)(x_0) + \phi\big(g(t)\big)(x_1)\Big)
            \\
            = \Big(0,\phi'\big(X'(0)\big)(x_0)\Big) \in 0\oplus\Delta_{x_0}\subseteq  \oct\oplus\oct
        \end{gather}
        which demonstrates that the vector field $x_0\mapsto \phi'\big(X'(0)\big)(x_0)$ is a characteristic of $\Delta$, and is therefore collinear with $\mathcal E$, and hence acts trivially on $\mathbb PK$. Because $\phi$ is infinitesimally effective, this implies $X'(0)=0$. 
    \end{proof}

    Finally, let us show that the algebra of infinitesimal symmetries of $\boldsymbol{\Delta}'$ is isomorphic to the Lie algebra $\mathfrak{g}=\mathfrak{g}(\mathfrak s)$ defined in \eqref{def g}.

    \begin{prop}
        The Lie algebra of infinitesimal symmetries of $\boldsymbol{\Delta}'$ is isomorphic to the Lie algebra $\mathfrak{g}$.
    \end{prop}
    \begin{proof}
        Lemmas \ref{lemma: Delta' rank} and \ref{lemma: Delta' growth} and the fact that $\mathrm{rank}(C)=4$ together imply that $\boldsymbol{\Delta}'$ has small growth vector $(3,6,8)$. By Theorem \ref{Thm: main theorem}, it suffices to show that the algebra of infinitesimal symmetries of $\boldsymbol{\Delta}'$ at any point has dimension $29$. Lemmas \ref{lemma: TG2 invar} and \ref{lemma: TG2 inf eff} demonstrate that the algebra of infinitesimal symmetries for $\boldsymbol{\Delta}'$ includes $\Phi'(T{\mathfrak{g}}_2)$ where $T{\mathfrak{g}}_2$ is the Lie algebra of $TG_2$, which has dimension $28$. Finally, observe that the vector field $\mathcal E'$ on $TK$ given by $\mathcal E'(x_0,x_1) = (0,x_1)$ generates a flow $F_t:TK\to TK$ that satisfies 
        \begin{gather}
            \label{euler flow 1}
            F_t^*(\pi^*\omega_0) = \pi^*\omega_0,\quad F_t^*(\omega_1) = e^t\omega_1
        \end{gather}
        so that for any $u\in \mathbb{R}^7$,
        \begin{gather}
            \label{euler flow 2}
            F_t^*\big(Q(u,\pi^*\omega_0)\big) = Q\big(u,F_t^*\pi^*\omega_0\big) = Q(u,\omega_0).
        \end{gather}
        For any $v\in \Gamma(TK)$, $F_t^*v = v\circ F_t = v$ because $v$ is constant along the fibers of $TK\to K$. Thus
        \begin{gather}
            \label{euler flow 3}
            F_t^*\big(Q(v,\omega_1)\big) = Q(v,e^t\omega_1) = e^t Q(v,\omega_1).
        \end{gather}
        Since \eqref{euler flow 2} and \eqref{euler flow 3} are both sections of $I_{\Delta'}$, we have that $\mathcal E'$ is a symmetry of $\Delta'$. 
        
        Further, because $\pi_*\mathcal E'=0$ and $\mathcal E'(x_0,x_1)$ depends not only on $x_0$ but also on $x_1$, $\mathcal E'$ is not equivalent to a vector field of the form \eqref{inf eff 1} modulo $C$. One can verify that $[\mathcal E',C]\subseteq C$, so that $\mathcal E'$ passes to a (nonzero) vector field on $\mathbb{P}_{\oct^2}(\Delta)$ which is a symmetry of $\boldsymbol{\Delta}'$. Therefore, the algebra of infinitesimal symmetries of $\boldsymbol{\Delta}'$ has dimension at least $\dim(T{\mathfrak{g}}_2) + 1 = 29$, as desired.
    \end{proof}

\section{Properties of abnormal extremal trajectories of $\mathfrak D$}
\label{section_corank}
%\end{appendices}
In this section we give control-theoretic properties of abnormal extremals for our $(3,8)$ distribution $\mathfrak D$. On the space of Lipschitz curves that are almost everywhere tangent to a given distribution $D$, called \emph{horizontal curves of $D$}, consider any variational problem that assigns a cost to each such curve—for example, the problem of minimizing length with respect to a sub-Riemannian metric. The Pontryagin Maximum Principle \cite{Pontryagin, Agrachev_Sachkov_2004} characterizes, through the Hamiltonian formalism, a class of curves, known as Pontryagin extremal trajectories, among which the minimizers of such problems (with fixed endpoints) must lie. The abnormal extremal trajectories associated with $D$ of Definition \ref{abnorm_def} are exactly the Pontryagin extremal trajectories for which the Lagrange multiplier multiplying the cost vanishes; they depend only on the distribution $D$, not on the particular cost function defining the variational problem.
%and they are also called \emph{singular curves of the distribution $D$} \cite{Montgomery_book}. 

%While abnormal extremal trajectories can be described purely geometrically using the canonical symplectic form on the cotangent bundle of $M$ (\cite{Liu-Sussmann, Montgomery_book,Zelenko99}, see also subsection \ref{subsec: Char Line Distr} below), for brevity we use here 

A description of abnormal extremal trajectories equivalent to the one in Definition \ref{abnorm_def} is as critical points of the endpoint map: Given a point $q_0$ and a time $T$, denote by $\Omega_{q_0}(T)$ the set of all horizontal curves of $D$ starting at $q_0$ defined on $[0,T]$, and by $F_{q_0, T} : \Omega_{q_0}(T) \to M$ the \emph{endpoint map} that takes each $\gamma \in \Omega_{q_0}(T)$ to the endpoint $\gamma(T)$. Note that $\Omega_{q_0}(T)$ has the structure of a $(L^\infty([0,T]))^l$-manifold, where $l$ is the rank of $D$. The following statement is very well known \cite{Agrachev-Sarychev}:
\begin{prop}
\label{abn_def}
A horizontal curve $\gamma : [0, T] \to M$ is an \emph{abnormal extremal trajectory} of the distribution $D$ if and only if it is a critical point of the mapping $F_{q_0,T}$, that is, if $\operatorname{Im} \left( d(F_{q_0,T})_\gamma \right) \neq T_{\gamma(T)} M$, where $d (F_{q_0,T})_\gamma$ denotes the differential of the endpoint map $F_{q_0,T}$ at $\gamma$. The \emph{corank} of the abnormal extremal trajectory $\gamma$ is defined as the codimension of $\operatorname{Im} \left( d_\gamma F_{q_0,T} \right)$ in $T_{\gamma(T)}M$ and is denoted by $\mathrm{corank}(\gamma)$. 
\end{prop}
Obviously, the corank of abnormal extremal trajectory is at least $1$. An abnormal extremal which projects to a given abnormal extremal trajectory $\gamma$ is called an \emph{abnormal lift} of $\gamma$. The goal of this section is to prove the following theorem

\begin{theorem}
\label{corank2}
Every abnormal extremal trajectory of the distribution $\mathfrak D$ has corank at least 2.
\end{theorem}
 This is the first example of its kind among bracket-generating distributions with generic small growth vector for a given rank and ambient dimension as far as we are aware.
\begin{proof} 
For regular abnormal extremal trajectories this is a direct consequence of nonmaximality of class. 
%As discussed in Subsection \ref{corank_abnormal_subsection}, all abnormal extremal trajectories originating at any point of the ambient manifold have a corank of at least 2. This is the first example of this kind for a bracket-generating distribution with the maximal possible growth vector for a given rank and a given dimension of the ambient manifold. 
This can be shown using the same arguments as in the proof of Theorem 4.1 of \cite{day2025canonicalframes},  which is based on \cite[Section 4]{Agrachev-Sarychev}. One only needs to take into account that in the real analytic category—as is the case for left-invariant distributions on Lie groups—the image of the differential of the endpoint map not only contains, but is actually equal to, the image of the canonical projection's differential applied to the maximal proper subspace in the Jacobi flag at the endpoint of the corresponding abnormal lift. 

However, in general a rank $3$ distribution $D$ with $6$-dimensional square may have nonregular abnormal extremals, i.e., those whose lifts contain a point of $\mathbb P(D^2)^\perp$: they may 
\begin{enumerate}
\item [(a)] entirely lie in $\mathbb P (D^2)^\perp$ and/or
\item [(b)] switch between $\mathbb P(D^2)^\perp$ and $\mathbb P D^\perp$.
\end{enumerate}
The latter requires the analysis of the dynamical system given by a vector field generating the characteristic line distribution near $\mathbb P(D^2)^\perp$, which is precisely the set of its stationary points. In general this analysis might be a nontrivial task (see, for example, an analogous situation in the case of rank $2$ distribution in \cite{zel99}). However, for our distribution $\mathfrak D$ it is easy to find first integrals of this dynamical system (and even explicitly integrate it) to exclude the situation in (b), namely we have the following
\begin{lemma}
\label{no_b_lemma}
If an abnormal extremal contains a point in $\mathbb P\big(\mathfrak D^\perp\setminus (\mathfrak D^2)^\perp\big)$ then it lies entirely in $\mathbb P\big(\mathfrak D^\perp\setminus (\mathfrak D^2)^\perp\big)$, i.e., it is a regular abnormal extremal. Consequently, all nonregular abnormal extremals are entirely contained in $\mathbb P(\mathfrak D^2)^\perp$. 
\end{lemma}
\begin{proof}
To start with, let us introduce some terminology from the elementary symplectic geometry. Given a vector field $X$ on $M$,  define the function $H_X$ on $T^*M$, the \emph{Hamiltonian} or the \emph{quasi-impulse} associated with the vector field $X$, and let $\vec H_x$ be the corresponding Hamiltonian vector field, so that $i_{\vec H_x}\sigma =-d H_X$. Then it is well known that
for any two vector fields $X_1$ and $X_2$ on $M$  
\begin{equation}
\label{poison1}
\overrightarrow {H_{X_1}} (H_{X_2})= d\, H_{X_2}\bigl(\overrightarrow {H_{X_1}}\bigr) =H_{[X_1, X_2]} 
\end{equation}
(which in turn is equal to $\{H_{X_1}, H_{X_2}\}$, the \emph{Poisson brackets} of $H_{X_1}$ and $H_{X_2}$).

Now let $v_1, v_2, v_3$ be the local basis of left invariant vector fields for our model $\mathfrak D$ so that in the notation of \eqref{n basis} 
\begin{equation}
\label{translation}
v_1=\oline e_0, \quad v_2=\oline f_0, \quad \textrm{ and } \quad  v_3=\oline x
\end{equation}
(these new notations are also aligned with notations of Appendix \ref{Appendix} below and in particular with \eqref{plane_1} below).
Let $u_i:=H_{v_i}$, $u_{ij}:=H_{[v_i, v_j]}$, and
$u_{[i,[j,k]]}:=H_{[v_i, [v_j,v_k]]}$.
For the rest of the proof it is more convenient to work with objects, submanifolds, characteristic distributions, and their generators, in the cotangent bundle instead of its projectivization. The characteristic distribution on $\mathfrak D^\perp$ is obtained as a kernel of the restriction of the canonical symplectic form $\sigma$ to $\mathfrak D^\perp$ will be denoted by the same letter $\mathcal C$ that previously denoted its pushforward to $\mathbb P (\mathfrak D^\perp)$.

Then, taking into account that
$$\mathfrak D^\perp=\{u_1=u_2=u_3=0\}$$
and using \eqref{poison1}, it is easy to show (\cite[Lemma 1]{doubrov2008rank3},\cite{doubrov2016JacobiCurves}) that the characteristic line distribution of $\mathfrak D$ satisfies 
\begin{equation}
\label{H_D_intersection}
\mathcal C(\lambda) =T_\lambda \mathfrak D^\perp\cap \langle \vec u_1(\lambda), \vec u_2(\lambda), \vec u_3(\lambda)  \rangle =\langle \mathfrak C(\lambda)\rangle,
\end{equation}
where
\begin{equation}
\label{Characteristic_rank3}
\mathfrak C=u_{23}\vec u_1+u_{31} \vec u_2+u_{12}\vec u_3.
\end{equation}
so the set  of all stationary points of $\mathfrak C$ (on $\mathfrak D^\perp$) is exactly 
\begin{equation}
\label{D^2_perp_u}
(\mathfrak D^2)^\perp=\{u_{12}=u_{31}=u_{23}=0\}\cap \mathfrak D^\perp.
\end{equation}

Then, from \eqref{n rels} and \eqref{translation}, using \eqref{poison1}, it follows that
\begin{align}
\mathfrak C(u_{12}) &= 0, \qquad \mathfrak C(u_{31}) = 0, \label{u_with_1}\\
\mathfrak C(u_{23}) &= u_{31}\,u_{[2,[2,3]]} + u_{12}\,u_{[3,[2,3]]}\label{u_23},
\end{align}
where $\mathfrak C(f)$ denotes the directional derivative of the vector field $\mathfrak C$ on the function $f$.

In particular, \eqref{u_with_1} implies that \(u_{12}\) and \(u_{31}\) are constant along any regular abnormal extremal. Moreover, \eqref{u_23} shows that if these two functions vanish along a regular abnormal extremal, then \(u_{23}\) is also constant along that extremal.

Now assume that an abnormal extremal $\Gamma$ contains both points in $\mathfrak D^\perp \setminus (\mathfrak D^2)^\perp$ and points in $(\mathfrak D^2)^\perp$. Then, by \eqref{D^2_perp_u}, the functions $u_{12}$, $u_{31}$, and $u_{23}$ vanish at every point of
$
\Gamma \cap (\mathfrak D^2)^\perp.
$
Since $u_{12}$ and $u_{31}$ are constant along $\Gamma$, and $u_{23}$ is also constant whenever $u_{12}=u_{31}=0$, it follows that $u_{12}$, $u_{31}$, and $u_{23}$ vanish identically along $\Gamma$. Hence, again by \eqref{D^2_perp_u}, we obtain
$
\Gamma \subset (\mathfrak D^2)^\perp,
$
which contradicts the assumption that
$
\Gamma \setminus (\mathfrak D^2)^\perp \neq \varnothing
$.
\end{proof}
It remains to analyze the corank of abnormal extremal trajectories whose abnormal lifts are entirely contained in  $\mathbb P(\mathfrak D^2)^\perp$.
Again, for simplicity let us work with the objects on the cotangent bundle instead of its projectivization.
Similarly to \eqref{H_D_intersection}, a nonregular abnormal extremal $\Gamma$ must be tangent at each of its points $\lambda$ to 
\begin{equation}
\label{H_D_intersection_1}
\mathcal A_1(\lambda):= 
T_\lambda (\mathfrak D^2)^\perp\cap \langle \vec u_1(\lambda), \vec u_2(\lambda), \vec u_3(\lambda)  \rangle
\end{equation}

Note that the fiber of $(\mathfrak D^2)^\perp$ can be naturally identified with $\mathfrak n_{-3}^*$ using the structure of Lie group on the base manifold $N$ on which the distribution $\mathfrak D$ is defined. Using \eqref{poison1} again, it is easy to show that the push-forward of $\mathcal A_1(\lambda)$ by the canonical projection $\pi:(\mathfrak D^2)^\perp\to N$ coincides with the kernel of the endomorphism corresponding to $\lambda$ under the identification above via the map given by the equation \eqref{dual} of Appendix \ref{Appendix}. This and \eqref{plane_1} imply that $\mathcal A_1$ is a rank $2$ distribution on $(\mathfrak D^2)^\perp$.\footnote{Note that if the same constructions are done for $(3,6,8)$ distributions $D$ with generic Tanaka symbol, then $\mathcal A_1$ is trivial, i.e. equal to zero, at a generic point of $(D^2)^\perp$. This is because generic elements of a generic plane in $\mathfrak{sl}_3(\mathbb R)$ have trivial kernels. The only other case of Tanaka symbols of $(3,6,8)$ distributions, when $\mathcal A_1$ is a rank $2$ distribution, is the case when the Tanaka symbol corresponds to the plane as in \eqref{plane_2}.}

From \eqref{n rels} and \eqref{translation}, using \eqref{poison1}, one can show that 
\begin{equation}
\label{C1-quasi}
\mathcal A_1=\langle \vec u_1, u_{[3, [2,3]]} \vec u_2-u_{[2,[2,3]]}\vec u_3\rangle
\end{equation}
By analogy with \eqref{J0 fact}, define $\mathcal Z^{(0)}$ as the lift of $\mathfrak D$ to $\mathbb P(\mathfrak D^2)^\perp$; that is,
\begin{equation}
    \label{Z0 fact}
    \mathcal{Z}^{(0)}(\lambda) = \pi^*\mathfrak D(\lambda)=\{v\in T_\lambda \mathbb P(\mathfrak D^2)^\perp: T\pi(v) \in \mathfrak D\}
\end{equation}To show that any nonregular abnormal extremal trajectory has corank at least $2$ it suffices to show that if $Y$ is a vector field tangent to the distribution $\mathcal A_1$, then 
\begin{equation}
\label{corank2_irreg}
\dim \,T_\lambda\pi\left(\mathrm{span}\left \{(\mathrm{ad}\, Y)^j (\mathcal Z^{(0)})(\lambda):j\in \mathbb N\right\}\right)\leq 6,\quad  \forall\ \lambda \in  \mathbb P(\mathfrak D^2)^\perp.
\end{equation}

Now by analogy with \eqref{J(-i) def} let 
\begin{equation}
\label{Z(-i) def}
    \mathcal{Z}^{(-i)} = \mathcal{Z}^{(1-i)} + [Y,\mathcal{Z}^{(1-i)}]\quad\text{for all }i\geq 1.
    \end{equation}
Assume that 
\[
    Y=\alpha_1  \vec u_1 +\alpha_2 (u_{[3, [2,3]]} \vec u_2-u_{[2,[2,3]]}\vec u_3).
\]
for some smooth functions $\alpha_1$ and $\alpha_2$. Then by direct computations from \eqref{n rels} and \eqref{translation}, using \eqref{poison1}, one can show that
\begin{align}
\mathcal{Z}^{(-1)}&=
\langle -\alpha_2(u_{[3,[2,3]]} \vec u_{12}-u_{[2,[2,3]]}\vec u_{13}), 
\\
&\quad\quad\alpha _1 \vec u_{12}+\alpha_2 u_{[2, [2,3]]} \vec u_{23}, \alpha _1 \vec u_{13}+\alpha_2 u_{[3, [2,3]]} \vec u_{23} \rangle+ \mathcal{Z}^{(0)},
\\
\label{flagZ}
\mathcal{Z}^{(-2)}&=\langle\alpha_2^2(u_{[2,[2,3]]}^2+u_{[3,[2,3]]}^2)(u_{[3,[2,3]]}\vec u_{[2,[2,3]]}-u_{[2,[2,3]]}\vec u_{[3,[2,3]]})\rangle+\mathcal Z^{(-1)},
\\
\mathcal Z^{(-3)}&=\mathcal Z^{(-2)}.
\end{align}
Also, 
\begin{equation}
\label{Zdiff}
\dim \mathcal Z^{(-1)}-\dim \mathcal Z^{(0)}\leq 2,
\end{equation}
because the three vector fields appearing in the span in the first line of \eqref{flagZ} are always linearly dependent. Together, \eqref{flagZ} and \eqref{Zdiff} imply \eqref{corank2_irreg}.
In fact, our proof shows that the equality in \eqref{corank2_irreg} holds for generic $\lambda$ and $Y$, meaning that generic nonregular abnormal extremal trajectories have corank 2.
\end{proof}

Note that a necessary condition for an abnormal extremal trajectory to be a length minimizer with respect to a sub-Riemannian metric on $D$, under the additional assumption that it is not a sub-Riemannian normal extremal trajectory,  is that it admits an abnormal lift lying entirely in $\mathbb{P}(D^2)^\perp$. This is known as the \emph{Goh condition} \cite{Agrachev-Sarychev}. This is also a necessary condition for an abnormal extremal trajectory  to be \emph{rigid}, i.e. isolated in the $C^1$ (or, more precisely,  $W^1_\infty$)-topology in the set of all horizontal curves of $D$ connecting given two points \cite{Agrachev-Sarychev, Bryant_Hsu, zel99}. 

Another second order condition necessary for rigidity of a smooth abnormal extremal trajectory $t\mapsto \gamma(t)$, the \emph{generalized Legendre condition} is also satisfied here. 
It states that if $y$ is a vector field tangent to $\gamma$, there is an abnormal lift   $\big([p(t)], \gamma(t)\big)$ of $\gamma(t)$ to $\mathbb P (D^2)^\perp$, such that for every $t$ the quadratic form on $D\big(\gamma(t)\big)/\big\langle y\big(\gamma(t)\big)\big\rangle$ given by 
 \begin{equation}
 \label{Legendre}
     w\mapsto p(t) \Big(\Big[\widetilde w, \big[\widetilde w, y\big(\gamma(t)\big)\big]\Big]\Big)
 \end{equation}    
     is sign semidefinite, where $\widetilde w$ is a local section of $D$ that extends a representative of the coset $w$; one can show that the expression in \eqref{Legendre} is independent of the choice of $\widetilde w$. However, one can use \eqref{n rels}, \eqref{translation}, and \eqref{poison1} to show that the quadratic form in \eqref{Legendre} is degenerate at all points $[p(t)]$ along the abnormal lift. In fact, the quadratic form is represented up to scaling by the rank 1 matrix 
     \[
     \begin{pmatrix}
         u_{[2,[2,3]]}^2 & u_{[2,[2,3]]}u_{[3,[2,3]]}\\
         u_{[2,[2,3]]}u_{[3,[2,3]]}&u_{[3,[2,3]]}^2
     \end{pmatrix}.
     \]
This demonstrates that the \textit{strong generalized Legendre condition}, which is a sufficient condition for rigidity requiring sign definiteness of the quadratic form \eqref{Legendre}, does not hold along $\gamma$. Therefore, the rigidity of such trajectories is an open question which we leave for future work.

We close this section with an observation suggesting an alternative symplectification procedure suited to some distributions of nonmaximal class; in particular, the alternative approach is effective for those Tanaka symbols of $(3,6,8)$ corresponding to the orbits of the planes \eqref{plane_1} and \eqref{plane_2} in $\mathfrak {sl}_3(\mathbb R)$ under the adjoint action, as outlined in Subsection \ref{(3,6,8)_subsec} of Appendix \ref{Appendix}. For these symbols, one can imitate the construction of Section \ref{subsection:characteristic}, replacing $W_D$ by $\mathbb P(D^2)^\perp$. 

In more detail, let $D$ be a $(3,6,8)$ distribution with the same Tanaka symbol as $\mathfrak D$. The contact distribution on $\mathbb PT^*M$ induces a corank $1$ distribution $\xi_1$ on the $9$-dimensional manifold $\mathbb P (D^2)^\perp$; the distribution $\xi_1$ has a rank $4$ Cauchy characteristic distribution $\mathcal C_1$. Note that the rank $2$ distribution $\mathcal A_1$ defined by \eqref{H_D_intersection_1} (or more precisely its pushforward to the projectivization $\mathbb P(D^2)^\perp$) is a subdistribution of $\mathcal C_1$,  and we have that at a generic point of $\mathbb P(D^2)^\perp$ 
\begin{equation}
\label{maxclass_alt} 
    \mathrm{span}\left\{ (\mathrm{ad}\, \mathcal A_1)^j   \mathcal Z^{(0)}:j\in \mathbb N\right\}=\xi_1,
\end{equation}
where $\mathcal Z^{(0)}$ is defined as in \eqref{Z0 fact} with $\mathfrak D$ replaced by $D$. Observe that the iterated brackets with $\mathcal A_1$ in \eqref{maxclass_alt} generate a larger set than the iterated brackets in \eqref{corank2_irreg}: the former allow the parameters $\alpha_1$ and $\alpha_2$ defining $Y$ to vary, while the latter consider only brackets with a fixed $Y$. Thus in contrast to \eqref{Zdiff}, we have
\[
    \dim\, (\mathrm{ad}\, \mathcal A_1) \mathcal Z^{(0)}-\dim\, \mathcal Z^{(0)}=3
\]
at a generic point of $(D^2)^\perp$.
Relation \eqref{maxclass_alt} can be interpreted as a condition of maximality of class for this alternative symplectification procedure (see also discussions in \cite[subsection 4.2.3]{BI2026generalized}). We hope that this observation will be useful for further treatment of $(3,n)$ distributions of nonmaximal class in the original sense of \cite{doubrov2016JacobiCurves} and Subsection \ref{subsection: CanFlag}.

\appendix
\section{Tanaka symbols of (3,8)-distributions with 6-dimensional square}
\label{Appendix}

A natural alternative approach to proving the final statement of Theorem \ref{Thm: main theorem} would proceed in two steps:
\begin{enumerate}
    \item [I.] Classify all Tanaka symbols of $(3,6,8)$  and $(3,6,7,8)$ distributions;
    \item [II.] Show that the Tanaka algebraic prolongation of the Tanaka symbol induced by the Jacobi-Tanaka algebra of our model is the most symmetric one. Since the space of Tanaka symbols depends on continuous parameters (as demonstrated below), this step would necessitate using a version of Tanaka theory adapted for distributions with non-constant Tanaka symbols \cite{Hong_Hwang}.
\end{enumerate}

It can be shown that $(3,6,7,8)$ distributions have only two possible Tanaka symbols. However, as we will see subsequently, the classification of $(3,6,8)$ distributions reduces to the classification of a pair of $3\times 3$ matrices up to conjugation (more precisely, the classification of planes of traceless matrices up to conjugation). In the terminology of the representation theory of quivers and finite-dimensional algebras, this is known as a ``wild'' problem \cite{Drozd1979, Drozd1980}. Coupled with the complexities of the second step, the strategy becomes largely unfeasible.

On the other hand, the symplectification procedure stipulates that among all Tanaka symbols, only those induced by Jacobi-Tanaka algebras are relevant for our problem of finding the maximally symmetric model. There are only two such symbols for $(3,6,8)$ distributions and one for $(3,6,7,8)$. Consequently, task II boils down to comparing the dimensions of the Tanaka algebraic prolongations of just these three Tanaka symbols.

\subsection{Case of (3, 6, 8) distributions} 
\label{(3,6,8)_subsec}
Let $D$ be a $3$-dimensional vector distribution on a $8$-dimensional manifold $M$ such that
$$\dim D^{-2} = 6, \quad \dim D^{-3} = 8$$
so that $D^{-3} = TM$. Recall that the Tanaka symbol $\mathfrak{n}(q)$ at each point $q \in M$ has graded pieces 
$\mathfrak n_{-i}(q):=D^{-i}(q)/D^{1-i}(q)$ for each $i>0$. The Tanaka symbol $\mathfrak n = \mathfrak n(q)$ of this distribution at a point $q$ has the form
\begin{equation}
    \mathfrak{n} = \mathfrak{n}_{-1} \oplus \mathfrak{n}_{-2} \oplus \mathfrak{n}_{-3},
\end{equation}
where $\dim \mathfrak{n}_{-1} = 3$, $\dim \mathfrak{n}_{-2} = 3$, and $\dim \mathfrak{n}_{-3} = 2$.

Identify $\mathfrak{n}_{-1}$ with a $3$-dimensional vector space $V$. Then $\mathfrak{n}_{-2}$ is naturally identified with $V \wedge V$. Fixing a volume form $\Omega$ on $V$, we can also identify $\mathfrak{n}_{-2}$ with $ V \wedge V \cong V^*$: 

$$v_1\wedge v_2 \to  \iota_{v_2}\circ \iota_{v_1} (\Omega), \quad \forall v_1, v_2\in V.$$ Then the Lie algebra structure of the symbol $\mathfrak{n}$ is uniquely determined by a surjective map
\begin{equation}
\label{-1_-2_-3}
\mathfrak{n}_{-1} \otimes \mathfrak{n}_{-2} \to 
%\mathbb{R}^2
\mathfrak{n}_{-3}, \quad \text{or} \quad V \otimes V^* \to 
%\mathbb{R}^2
\mathfrak{n}_{-3}
\end{equation}
such that 
\begin{equation}
\label{Levi}
x\otimes y\mapsto [x, y], x\in \mathfrak {n}_{-1}, y\in \mathfrak{n}_{-2},
\end{equation}
or its dual 
%\mathbb{R}^2 
\begin{equation}
\label{dual}
\mathfrak{n}_{-3}^*\to V^* \otimes V \cong \operatorname{End}(V).
\end{equation}
Since the latter map is injective, it is uniquely determined by its image, i.e by a plane in $\mathrm{End}(V)$. Moreover, the Jacobi identities are equivalent to the fact that the image of this dual map consists of traceless endomorphisms. Hence, the space of Tanaka symbols $\mathfrak{n}$ is in the bijective correspondence with the space of orbits of the Grassmannian of planes in $\mathfrak{sl}_3(\mathbb R)$  with respect to the adjoint action of $\mathrm{SL}_3(\mathbb R)$ on its Lie algebra.

Let us express this plane in terms of structure constants of an appropriate basis $\mathfrak n$: let $V=\mathrm{span} (v_1, v_2, v_3)$ so that $\Omega(v_1, v_2, v_3)=1$
Then, if $(v^1,v^2, v^3)$ is the basis of $V^*$ dual to $(v_1, v_2, v_3)$, then  the identification $ V \wedge V \cong V^*$ is given by 
\begin{equation}
\label{basis_identification}
v_1\wedge v_2 \mapsto v^3, \quad v_3\wedge v_1\mapsto v^2, \quad v_2\wedge v_3\mapsto v^1.    
\end{equation}
Finally, let $\mathfrak{n}_{-3}=\mathrm{span}(w_1, w_2)$. Taking onto account the identification $ V \wedge V \cong V^*$ and \eqref{basis_identification}, assume that the map  \eqref{-1_-2_-3}, is given by 
\[ v_i\otimes v^j\mapsto \sum_{s=1}^2 a_{ij}^s w_s.\]
The only nontrivial Jacobi identity on $\mathfrak n$ is obtained from the triple $(v_1,v_2,v_3)$, which is equivalent via \eqref{basis_identification} to the condition that the matrices $A_s:=(a_{ij}^s)$ are traceless.

The plane in the image of the map in \eqref{dual} is spanned by the elements of $\mathrm{End}(V)$ represented in the basis $(v_1, v_2, v_3)$ by the matrices $A_1^T$ and $A_2^T$, the transposes of $A_1$ and $A_2$.

The intrinsic difficulty in classifying these Tanaka symbols stems directly from the intricate continuous moduli of the underlying orbit space of matrix planes. The classification of planes of traceless $3\times 3$ matrices up to conjugation is equivalent to classifying pairs of traceless matrices up to simultaneous conjugation by $\mathrm{SL}(3,\mathbb{R})$ and linear changes of basis in the plane. As established by Drozd \cite{Drozd1979, Drozd1980}, the underlying algebraic problem of classifying pairs of matrices (in all dimensions) up to simultaneous conjugation is strictly ``wild''. 
%While tame representation types also admit continuous families of representations, the crucial distinction lies in the indecomposable blocks. In a tame problem, the indecomposable representations of any given dimension belong to families parameterized by at most a single continuous variable. In stark contrast, 
Roughly speaking, wildness means the representation category contains the representations of the free algebra on two generators, so indecomposable blocks occur in families whose number of continuous parameters grows arbitrarily large with the representation dimension. Although this wildness is formally a property of the total classification in all dimensions, its consequences are acutely felt even in low dimensions. Already in dimension three, the moduli space lacks any simple description and depends on multiple continuous parameters rather than a simple discrete or 1-dimensional classification for indecomposable blocks, like in the case of one matrix.

The dimension of the Grassmannian of planes in $8$-dimensional space $\mathfrak{sl}_3(\mathbb R)$ is equal to $2\times (8-2)=12$. The generic orbits under the adjoint action of $\mathrm {SL}_3(\mathbb R)$ have trivial stabilizers, i.e. they are 8 dimensional and therefore parametrized by $12-8=4$ parameters. From the perspective of geometric invariant theory, the ring of polynomial invariants for this conjugation action is generated entirely by traces of words in the matrices. This was proven in the foundational works of Procesi \cite{Procesi1976}, and the $3\times 3$ matrix case was more explicitly elaborated by Formanek \cite{Formanek1979} and Teranishi \cite{Teranishi1986}. For generic orbits, the classification heavily relies on these invariants.
%a pair of traceless $3\times 3$ matrices generally depends on $8$ continuous parameters up to simultaneous conjugation ($16$ dimensions for the pair minus $8$ for the group $\mathrm{SL}(3,\mathbb{R}){\color{red}\cong\mathrm{SL}(3,\mathbb R)}$), which drops to $4$ continuous parameters when considering the plane itself (accounting for the $4$-dimensional $\mathrm{GL}(2,\mathbb{R})$ action of basis changes). 
However, this invariant-theoretic approach fundamentally fails for the nilpotent cone---because traces of all words in jointly nilpotent matrices necessarily vanish, these polynomial invariants cannot separate nilpotent orbits. 
%In the algebraic quotient, the orbit closures of all such nilpotent planes intersect at the origin, rendering them entirely undetectable and indistinguishable by trace invariants.

Despite the continuous moduli of generic planes, there are exactly two orbits of planes in $\mathfrak{sl}_3(\mathbb{R})$ of minimal dimension (equal to 2), both in the nilpotent cone. They are orbits of the following planes:

\begin{align}
\Pi_1=\left\{\begin{pmatrix}0&a&b\\0&0&0\\0&0&0\end{pmatrix}: a,b\in \mathbb R  \right\}\label{plane_1}\\
\Pi_2=\Pi_1^T=\left\{\begin{pmatrix}0&0&0\\a&0&0\\b&0&0\end{pmatrix}: a,b\in \mathbb R  \right\}\label{plane_2}
\end{align}

The first orbit consists of planes of $3\times 3$ matrices which are nilpotent of rank $1$ and have the same image. The second orbit consists  of planes of $3\times 3$ matrices which are nilpotent and have the same two dimensional kernel. 

To justify that the orbits of $\Pi_1$ and $\Pi_2$ have minimal dimension, first recall a classical result: up to conjugation, $\mathfrak{sl}_3(\mathbb R)$ has no subalgebras of codimension $1$ and exactly two subalgebras $\mathfrak{p}_1$ and $\mathfrak{p}_2$ of codimension 2, which geometrically correspond to the stabilizers of a $1$-dimensional subspace (a line) and a $2$-dimensional subspace (a plane) in $\mathbb{R}^3$. This two-fold classification can be traced back to Lie's  and Engel's foundational work on finite-dimensional Lie algebras of vector fields on the projective plane \cite{Lie1893}. For modern algebraic formulations and exhaustive classifications of the maximal subalgebras of real simple Lie algebras, we refer the reader to the works of Dynkin \cite{Dynkin1952}, Komrakov \cite{Komrakov1990}, and Winternitz et al.\ \cite{Winternitz}.

If $(x_1, x_2, x_3)$ are standard coordinates in $\mathbb{R}^3$, the first subalgebra $\mathfrak{p}_1$ can be taken to stabilize the $x_1$-axis, and consequently, it stabilizes the plane $\Pi_1$. The second subalgebra $\mathfrak{p}_2$ can be taken to stabilize the coordinate $(x_2, x_3)$-plane, and consequently, it stabilizes the plane $\Pi_2$.

To justify that these planes $\Pi_i$ are the only planes stabilized by $\mathfrak{p}_i$ for $i=1,2$, observe that both of these subalgebras are parabolic, hence they define the grading on $\mathfrak{sl}_3(\mathbb{R})$. For example, for $\mathfrak{p}_1$, we have the decomposition:
\begin{equation}
\label{sl_3_decomp}
\mathfrak{sl}_3(\mathbb{R}) = \mathfrak{g}_{-1} \oplus \mathfrak{g}_0 \oplus \mathfrak{g}_1,
\end{equation}
where in fact $\mathfrak{g}_{1} = \Pi_1$, $\mathfrak{g}_{-1} = \Pi_2$, and $\mathfrak{g}_0 \cong \mathfrak{gl}_2(\mathbb{R})$, which corresponds to the block-diagonal structure of $3\times 3$ matrices according to the choice of the parabolic subalgebra $\mathfrak{p}_1$. 
The decomposition \eqref{sl_3_decomp} decomposes $\mathfrak{sl}_3(\mathbb{R})$ into $\mathfrak{g}_0$-modules, where $\mathfrak{g}_{-1}$ and $\mathfrak{g}_{1}$ are irreducible, and $\mathfrak{g}_0$ is further decomposed into a $3$-dimensional module ($\cong \mathfrak{sl}_2(\mathbb{R})$) and a $1$-dimensional module. This implies that $\mathfrak{g}_{-1}$ and $\mathfrak{g}_{1}$ are the only planes ($2$-dimensional invariant subspaces) stabilized by $\mathfrak{g}_0$. 
Finally, $\Pi_1 = \mathfrak{g}_{1}$ is the only plane stabilized by $\mathfrak{p}_1 = \mathfrak{g}_0 \oplus \mathfrak{g}_{1}$. The argument showing that $\Pi_2$ is the only plane stabilized by $\mathfrak{p}_2$ is completely analogous (one simply needs to swap $\mathfrak{g}_{-1}$ with $\mathfrak{g}_1$).

A remarkable, though expected, fact is that those two orbits correspond exactly to Tanaka symbols induced by the two possible Jacobi-Tanaka algebras corresponding to $(3,6,8)$ distributions, namely 
\begin{itemize}
\item The first orbit corresponds to the Tanaka symbol 
induced by the Jacobi-Tanaka algebra $\mathfrak s$ of nonmaximal class defined in Section \ref{subsection: nonmax model}, i.e., to our  $29$-dimensional model;
\item The second orbit corresponds to the Tanaka symbol induced by the Jacobi-Tanaka algebra of the maximal class  with Jacobi symbol given by the $2\times 5$ rectangular Young diagram (i.e., $k=2$ and $\ell=0$ in the notation of \eqref{Jacobi Symb Diagram}). The infinitesimal symmetry algebra of the corresponding symplectically flat distribution was computed in \cite{doubrov2008rank3}: it is $18$-dimensional and isomorphic to the semidirect sum $$(\mathfrak{gl}_2 \oplus  \mathfrak{sl}_2)\ltimes (\mathbb R^2\otimes V_5\oplus\langle \eta\rangle),$$ where  $\mathfrak{sl}_2$ acts irreducibly on a five-dimensional space $V_5$, $\mathfrak{gl}_2$ acts in the standard way on $\mathbb R^2$, $\mathfrak{heis}_{11}=\mathbb R^2\otimes V_5+\langle \eta\rangle$  is the 11-dimensional Heisenberg algebra with center $\langle \eta\rangle$ and for every line $w \in\mathbb R^2$,  $w\otimes V_5$ is a Lagrangian  subspace in the space $\mathbb R^2\otimes V_5$ with respect to the conformal symplectic structure induced by the brackets on $\mathfrak{heis}_{11}$.  
%the notation of the  Proposition \ref{prop: m rels} if $w_1, w_2Heres $\mathfrak{sl}_2$ and $\mathfrak{glstructure constitubyhe symmetry algebra of the curve of flags with the rectangular $2 \times 5$ diagram: $\mathfrak{sl}_2$ corresponds to the projective parameter/ acts irreducibly on each row of the diagram) and $\mathfrak{gl}_2$ is simply the action on the plane of highest weight vectors. In the Heisenberg (contact) grading—i.e., when $\mathfrak{sl}_2 + \mathfrak{gl}_2$, embedded in $\mathfrak{csp}(10)$, are of weight zero—the first prolongation is $0$.
\end{itemize}

%the space of endomorphisms of $3$-dimensional space or, equivalently, a plane in the space of $3 \times 3$ matrix up to a conjugatio.

\subsection{Case of (3, 6, 7, 8) distributions}
In this case, the image of the map \eqref{dual} defines a line in $\mathrm{End}(V)$ generated by a traceless matrix. However, in addition, the map 
\begin{equation}
\label{Levi13}
\mathfrak{n}_{-1}\otimes \mathfrak {n}_{-3} \to \mathfrak n_{-4},
\end{equation}
given by analogy with \eqref{Levi},  defines a line in $\mathfrak n_{-1}^*=V^*$, which in turn is determined by the  plane in $V$ annihilating the elements of this line. Therefore the Tanaka symbol is uniquely determined by a pair $([A], W)$, where $[A]$ is a line generated by $A\in \mathrm {End}(V)$ and $W$ is a plane in $V$, up to the natural action of $\mathrm {GL}(V)$. The Jacobi identity in weight $-3$ is again equivalent to the condition that $\mathrm{tr}(A)=0$, while the Jacobi identity in weight $-4$ is equivalent to the condition that $A$ preserves $W$ and acts on it by scaling. This leads to exactly two orbits: one in which this scaling factor is zero and another in which this scaling is nonzero; fixing a basis $(v_1,v_2,v_3)$ of $V=\mathfrak n_{-1}$, the orbits are represented by the pairs $([A_1], W)$ and $([A_2], W)$, where
\begin{equation}
    A_1 = \begin{pmatrix}
        0 & 0 & 1 \\
        0 & 0 & 0 \\
        0 & 0 & 0
    \end{pmatrix},
    \quad
    A_2 = \begin{pmatrix}
        1 & 0 & 0 \\
        0 & 1 & 0 \\
        0 & 0 & -2
    \end{pmatrix},
    \quad \textrm{and}\quad W = \langle v_1,v_2\rangle
\end{equation}
in some basis $(v_1,v_2,v_3)$ of $V=\mathfrak n_{-1}$.

The first orbit corresponds to the Tanaka symbol induced by the Jacobi-Tanaka algebra of maximal class with Jacobi symbol with skew diagram consisting of two rows of length $5$ each and with shift $2$ (i.e., $k=1$ and $\ell=2$ in the notation of \eqref{Jacobi Symb Diagram}). The infinitesimal symmetry algebra of the corresponding symplectically flat distribution was computed in \cite{doubrov2008rank3} and \cite{DayDoubrovZelenko2026}: it is $23$-dimensional and is described via a rational normal curve in $\mathbb{RP}^4$ and its secant varieties \cite{doubrov2025large}.

The second orbit is also interesting, because the prolongation of the corresponding Tanaka symbol is the split real form of $C_3$, so it is the flat model for the parabolic geometry $C_3/P_{12}$ (i.e., associated with the grading of $C_3$ with both short roots are marked). In particular, the infinitesimal symmetry algebra of the corresponding flat model is $21$-dimensional. It is a distribution of maximal class with the Jacobi-Tanaka symbol as in the model for the first orbit, but it is not symplectically flat (and, as expected, has a smaller dimensional symmetry algebra). This provides an example where \emph{the symplectification procedure includes a parabolic geometry as a particular case of a more general class of distributions (those with given Jacobi-Tanaka algebra) for which one has a uniform construction of absolute parallelism, and the parabolic flat model is 
not the flat model in this unification.}

Finally, it turns out that this model is exactly the symplectification of the flat $(4,7)$-distribution corresponding to the parabolic geometry $C_3/P_2$ (i.e., associated with the grading of $C_3$ with the middle root being marked). Note that the symplectification of such $(4,7)$ distribution is different from the symplectification described here, based on constructions in \cite{doubrov2016JacobiCurves} for distributions of arbitrary rank, by the fact that the characteristic distribution on the submanifold analogous to $W_D$ is of rank greater than 1  (in fact, of rank 3) at a generic point of this submanifold. More details on this will be given in \cite{KZ}.

%\newpage
\bibliographystyle{plain}
\bibliography{Bibliography.bib}

\end{document}